\documentclass[11pt]{article}
\usepackage[english]{babel}
\usepackage[utf8]{inputenc}
\usepackage{mathtools,amsmath,amssymb,amsthm}
\usepackage{graphicx,multirow}
\usepackage{fancyhdr,tabularx}
\usepackage{lastpage}
\usepackage{mathrsfs}
\usepackage{color, float}
\usepackage{latexsym}
\usepackage{tensor}
\usepackage{tikz}
\usepackage{hyperref}
\usepackage{xparse}
\usepackage{textcomp}
\usepackage{ulem}
\usepackage[top = 1.0cm,bottom = 3cm,left = 2.5cm, right = 2.5cm]{geometry}
\usepackage{hyperref}
\usepackage{cancel,soul}

\newcommand{\degree}{\si{\degree}}

\def \hueco{\noalign{\medskip}}
\def \beq{\begin{equation}}
\def \eeq{\end{equation}}
\def \ba{\begin{array}}
\def \ea{\end{array}}
\def \dis{\displaystyle}
\def \n{\boldsymbol{n}}
\def \x{\boldsymbol{x}}
\def \m{\boldsymbol{m}}
\def \z{\boldsymbol{z}}
\def \J{\boldsymbol{J}}

\newtheorem{theorem}{Theorem}[]

\newtheorem{lemma}[theorem]{Lemma}
\newtheorem{corollary}[theorem]{Corollary}
\newtheorem{remark}[theorem]{Remark}

\numberwithin{theorem}{section}

\title{\sc{A new  optimal control algorithm for the Keller-Segel problem}}
\author{
F. Guill\'en-Gonz\'alez\thanks{EDAN and IMUS, Universidad de Sevilla, Seville (Spain). Email: guillen@us.es}, \quad
F. Palmero-Ramos\thanks{Department of Mathematics and Statistics, University of Massachusetts Amherst, Amherst MA (USA). Email: fpalmeroramo@umass.edu},\quad
M.A. Rodr\'iguez-Bellido\thanks{EDAN and IMUS, Universidad de Sevilla, Seville (Spain). Email: angeles@us.es}\quad
\&\quad G.~Tierra\thanks{Department of Mathematics, University of North Texas, Denton TX (USA). Email: gtierra@unt.edu}
}
\date{\today}

\begin{document}
\maketitle

\begin{abstract}
In this work we introduce a new optimal control algorithm for the Keller-Segel chemo-attraction system, where both boundary and distributed controls are considered and both are associated with introducing/removing the amount of chemical substances in the system. The key idea of our approach is to design the optimal control algorithm after discretizing the state problem system, which is done using an upwind finite volume scheme in space and a semi-implicit finite difference in time.
Then, the discrete optimal control is approximated identifying the gradient of the reduced discrete cost via the discrete adjoint scheme. Finally, to minimize the reduced cost functional, we use a gradient descent type method (Adam scheme). Moreover, several numerical results are presented to illustrate the efficiency of the proposed approach.
\end{abstract}


{\bf The Mathematics Subject Classification (MSC):  35Q92, 49M25, 49M41, 65K10, 92C17.} 


\section{Introduction}

Optimal control problems can be understood as PDE problems with control variables acting on the bulk and/or the boundary of the domain to optimize the behavior of the observed variables of the state problem with respect to some cost functional. This is a fascinating and complex topic due to its mathematical challenges  \cite{Lions,Troltzsch} as its applicability to a wide range of applications such as optimizing industrial processes or 
designing treatment for diseases, 
to name a few.

\

Chemotaxis effects refer to organisms or cells orienting/moving in relation to the distribution of chemical agents, being called \textit{attractive} if the movement is towards a higher concentration of the chemical substance and \textit{repulsive} the other way around. Models related with chemotaxis effects lead naturally to optimal control problems, in the sense that the chemical agent can be used as a control variable to induce a desired dynamic on the distribution of the living organisms. 

\

To provide a context for the ideas presented in this work, we first discuss related works available in the literature focusing on  optimal control  problems for chemotaxis systems.
 The authors in \cite{RY01} analyze a distributed (nonnegative) optimal control in the chemical agent for the $2D$ Keller-Segel problem with small enough  amount of initial cells, obtaining global in time solutions. This type of control only model the case where  chemoattractant is introduced.
Later on,   this contribution was extended in  \cite{R03,R04} and  similar techniques were  used in \cite{R08,R13} 
for the $1D$ Keller-Segel problem with  boundary  control on Neumann type  for the chemical agent. Moreover, in \cite{FMcC03} the authors provide a study of an optimal control problem for a $1D$ attraction system with a  production term using the Michaelis-Menten saturation factor, where they introduce two bilinear controls to represent either the amount of the cells or the proportion of the chemoattractant that is being  removed.  On the other hand, \cite{LiuYuan21} deals with the analysis of a distributed optimal control problem for a $2D$ attraction-repulsion chemotaxis system. 
Feldhordt in his PhD thesis \cite{Feldhordt}, 
analyze a boundary optimal control problem for a chemotaxis system in $\mathbb{R}^d$ ($d\geq2$)  with bounded chemotactic sensitivity function and propose numerical methods for finding a solution.

\

 In order to model the case where chemical signal can be introduced or removed it is usual to take bilinear controls.  For instance, the authors in \cite{Braz-KS-logistic} study existence of optimal control and a necessary optimality system for a bilinear optimal control problem associated to the $2D$ Keller-Segel system considering a logistic reaction term.
Other related works include \cite{Andre-KS-consumo,Andre-KS-consumo-weak}, where the authors study a bilinear optimal control problem associated to a $3D$ attraction-consumption model, providing the existence of a weak optimal solution in \cite{Andre-KS-consumo-weak}, while in \cite{Andre-KS-consumo} optimality conditions are obtained assuming the existence of  strong solutions. 

\

There are also works more oriented towards a specific application. For instance, the authors in \cite{LebiedzMaurer}, describe numerical optimal control of E. coli bacterial chemotaxis based on a 1D two-component PDE model  and they present a numerical scheme to force cell aggregation patterns by applying a Robin boundary  control of chemoattractant. Other examples are \cite{dAdM15}, where the authors study a distributed optimal control problem for a two-dimensional mathematical model of cancer invasion and \cite{TangYuan}, a paper that deals with a chemotaxis-haptotaxis model which described the process of cancer invasion on the macroscopic scale. 

Additionally, there are relevant works coupling chemotaxis models with hydrodynamics effects such as \cite{MADAEJ18}, where the authors analyze an optimal distributed control problem where the state equations are given by a stationary chemotaxis model coupled with the Navier-Stokes equations. Another interesting example is found in \cite{JEJ21}, a paper that study an optimal control problem associated to a 3D chemotaxis-Navier-Stokes model.

\

In recent times many works devoted to the numerical approximation of different chemotaxis PDE systems have appeared in the literature. In the following we refer to the ones most related to this work: In \cite{Saito} Saito presents an upwind Finite Element (FE) scheme that preserves positivity and mass conservation at the discrete level for a parabolic-elliptic Keller-Segel model. Then the work \cite{Saad-et-al-2014} focuses on developing an upwind Finite Volume (FV) scheme which preserves a discrete maximum principle for a degenerate Keller-Segel system. For the classical Keller-Segel model,  an upwind FV  scheme was proposed in \cite{Zhou-Saito}, which preserves mass conservation and positivity, deriving some inequalities on the discrete free energy. 
After that, the authors in \cite{JV-Rafa-KS-FE}  derive a FE approximation of Keller-Segel, proving a priori energy estimates and discrete positivity under constraints implying that the diffusion term dominates the chemotaxis effects. More recently, in \cite{Dani-Rafa-KS-DG-23}, the authors also focus on Keller-Segel but they consider an upwind Discontinuos Galerkin (DG) scheme which preserves positivity and the decreasing energy property. Additionally it is worth to mention the work presented in \cite{Diego-MA-19-FE-RegularizationEnergy}, whose focus is on a FE approach with a regularization of the energy, which allows the authors to preserve positivity approximately while assuring the decreasing energy property for a repulsion-production chemotaxis system. Interestingly, it is possible to consider similar ideas for an attraction-consumption chemotaxis problem and also preserve positivity approximately while assuring the decreasing energy property \cite{G-KS-consumo-24}.

\

In this work, we introduce a new optimal control algorithm for the Keller-Segel chemo-attraction system, where both boundary and distributed controls are considered and both are associated with introducing/removing the amount of chemical substances in the system. 
The main idea behind our approach is to design the optimal control algorithm after discretizing the state problem system (\textit{discretize-then-optimize} approach).  This discretization  is done using the positivity-preserving and mass conservative upwind FV scheme presented in \cite{Zhou-Saito}. 
Once the state scheme has been chosen, 
we compute exactly the discrete gradient of the reduced cost owing to  the  adjoint scheme associated to the sensibility  scheme. The discrete gradient will be used to approximate the discrete optimal control by an optimizer algorithm. In fact, we minimize the reduced cost functional using the Adam scheme \cite{KB15}, which is a very efficient gradient descent type method that allows us to perform numerical experiments in a reasonable amount of time. In order to be able to manage the computational cost of the algorithm with a desktop computer, we present our algorithm in $1D$ domains, but the ideas can be extended naturally to higher dimensions without loss of generality.

\

This work is organized as follows. Section~\ref{sec:model} is devoted to introduce the optimal control problem. In particular, we first present a general formulation of a chemotaxis  control problem and then we specify the particular problem that we will focus on, consisting on the case of a quadratic tracking cost with bilinear distributed control and boundary control, of either bilinear or Robin type.
Then in Section~\ref{sec:DiscOpt} we describe the \textit{discretize-then-optimize} approach that we are considering in this work, as well as the Adam minimization algorithm.
We present several numerical results in Section~\ref{sec:simulations} to illustrate the well behavior of the proposed algorithm and its applicability to simulate complex situations with different combinations of observation and control domains.
Finally, concluding remarks are provided in Section~\ref{sec:conclusions}.

\section{The optimal control problem}\label{sec:model}
In this section we describe the model that we are considering in this work. We start by introducing a general optimal control problem for a chemotaxis system, together with the derivation of the associated adjoint problem obtained by taking the homogeneous part of the linear problem satisfied by the derivative of the states variables $(u,v)$ with respect to the control variables $(f,g)$. After that we detail the particular problem that we are focusing on, by specifying the choice of the cost functional and the reaction terms considered in the system.

\subsection{General formulation of a chemotaxis optimal control problem}

Let $\Omega\subset \mathbb{R}^d$ ($d=1,2,3$) be a bounded spatial domain with polyhedral boundary $\partial\Omega$ and $[0,T]$ be a finite time interval.
We denote by $\Omega_o, \Omega_c \subset \Omega$ the corresponding observation and control domains, and $\partial\Omega_c\subset \partial\Omega$ represents the region of the boundary where the boundary control will have an effect on.

The domain distributed control is denoted by $f(t,\x): (0,T)\times \Omega_c\to \mathbb{R}$, the boundary control is represented by $g(t,\x):(0,T)\times \partial\Omega_c\to \mathbb{R}$.  The states variables associated to the density of living organisms (cells for instance)  and the concentration of chemical are denoted by $u(t,\x)$, $v(t,\x):(0,T)\times\Omega\to \mathbb{R}$, respectively. 

The cost functional is defined as the addition of a term tracking the state variables, another one measuring the final time states and two quadratic terms associated to both controls functions, the domain control and the boundary control:
\beq\label{eq:CostFunctional}
\ba{rcl}
J((u,v),(f,g))
&=&\dis
\frac{1}{T\, |\Omega_o|}\int_0^T\int_{\Omega_o} F(u,v) d\x dt 
+ \frac{1}{ |\Omega_o|}\int_{\Omega_o} G(u(T,\cdot),v(T,\cdot)) d\x
\\ \hueco
&+&\dis
 \frac{\alpha_f}2 \frac{1}{T\, |\Omega_c|} \int_0^T\int_{\Omega_c} f^2 d\x dt 
+ \frac{1}{T\, |\partial \Omega_c|}\frac{\alpha_g}2 \int_0^T\int_{\partial\Omega_c} g^2 d\sigma dt\,,
\ea
\eeq
where $\alpha_f\ge 0$ and $\alpha_g\ge 0$ denote the weights of the control terms.

The chemotaxis problem that we consider is the following, find $(u(t,\x),v(t,\x))$ such that:
\beq
\left\{\ba{rcl}
u_t + \nabla\cdot (-\nabla u + u \nabla v) 
&=&
R(u,v)\,,
\\ \hueco
v_t + \nabla\cdot (-\nabla v) 
&=&
S(u,v,f\, 1_{\Omega_c})\,,
\ea\right.
\eeq
with generic functionals $R$ and $S$ modeling reaction terms, where $S$ depends on the distributed control $f\, 1_{\Omega_c}$. By simplicity, we have taken the diffusion and convection coefficients equal to $1$. This system is complemented with initial conditions
\beq
u(0,\x)=u_0(\x)\,,\quad v(0,\x)=v_0(\x) ,\quad \x\in \Omega,
\eeq
and boundary conditions to represent that the cell variable is considered insulated inside the domain, while the flux of the chemical variable could be determined by a functional $P$ which depends on the boundary control $g\, 1 _{\partial\Omega_c}$ such that:
\beq
(-\nabla u + u \nabla v)\cdot \n \big|_{\partial\Omega}
=0
\quad
\mbox{ and }
\quad
(-\nabla v)\cdot \n \big|_{\partial\Omega}
=
-P(v,g\, 1 _{\partial\Omega_c})\,,
\eeq
where $\n$ denotes the outward normal vector to the boundary $\partial\Omega$. In particular, $P\ge 0$ represents chemical entering in the domain and $P\le 0$ means chemical going out through the boundary. 

\subsection{Linearized problems}\label{sec:subseclinearization}
Now we compute the linear problem satisfied by the derivative of the states variables  $(u(f,g),v(f,g))$ around a 
control pair $(f,g)$  in the direction associated with any pair of controls $( \bar f, \bar g)$. By introducing the  notation
$$
(U_f,V_f):=\left(\frac{\partial u(f,g)}{\partial f}(\bar f),\frac{\partial v(f,g)}{\partial f}(\bar f)\right)
\quad 
\hbox{ and }
\quad 
(U_g,V_g):=\left(\frac{\partial u(f,g)}{\partial g}(\bar g), \frac{\partial v(f,g)}{\partial g}(\bar g)\right),
$$
then, the pair $(U_f,V_f)$ solve the linear problem
\beq\label{eq:problemUfVf}
\left(
\ba{c}
U _t 
\\\hueco
V _t 
\ea
\right)
+ 
\nabla\cdot
\left(
\begin{array}{c}
   -\nabla U + u \nabla V+ U \nabla v \\ -\nabla V   
\end{array}
\right)
-
\left(
\begin{array}{cc}
  R_u(u,v) & R_v(u,v) \\ S_u(u,v,f) & S_v (u,v,f) 
\end{array}
\right)
\left(
\begin{array}{c}
  U  \\ V   
\end{array}
\right)
=
\left(
\begin{array}{c}
  0  \\ S_f(u,v,f \, 1_{\Omega_c}) \bar f   \, 1_{\Omega_c}
\end{array}
\right)\,,
\eeq
subject to the boundary and initial conditions 
$$
(-\nabla U + u \nabla V+ U \nabla v)\cdot \n \big|_{\partial\Omega}
\,= 0, \quad
(-\nabla V)\cdot \n \big|_{\partial\Omega}+P_v(v,g\,1_{\partial\Omega_c} )V\,1_{\partial\Omega_c} 
\,=\,0
$$
and
$$
U(0,\x)=0\,,
\quad 
V(0,\x)=0\,. 
$$
Moreover, the pair $(U_g,V_g)$ solve the linear problem
\beq  \label{eq:problemUgVg}
\left(
\begin{array}{c}
  U  \\ V   
\end{array}
\right)_t
+ \nabla\cdot
\left(
\begin{array}{c}
   -\nabla U + u \nabla V+ U \nabla v \\ -\nabla V   
\end{array}
\right)
-
\left(
\begin{array}{cc}
  R_u(u,v) & R_v(u,v) \\ S_u(u,v,f) & S_v (u,v,f) 
\end{array}
\right)
\left(
\begin{array}{c}
  U  \\ V   
\end{array}
\right)
=
\left(
\begin{array}{c}
  0  \\ 0
\end{array}
\right)\,,
\eeq
subject to the boundary conditions 
$$
(-\nabla U + u \nabla V+ U \nabla v)\cdot \n \big|_{\partial\Omega}=0,
\quad
(-\nabla V)\cdot \n \big|_{\partial\Omega}+P_v(v,g\, 1_{\partial\Omega_c} )V \, 1_{\partial\Omega_c}
= -P_g(v,g\,1_{\partial\Omega_c})\bar g\, 1_{\partial\Omega_c}  ,
$$
and the initial conditions
$$
U(0,\x)
\,=\, 0, \quad 
V(0,\x)
\,=\,0\,.
$$

\subsection{Adjoint problem}

The adjoint problem corresponding with the homogeneous part of the linearized problems \eqref{eq:problemUfVf} and  \eqref{eq:problemUgVg} consists on finding $(\varphi,\psi)$ such that
\beq\label{eq:AdjointPb}
\ba{c}\dis
- \left(
\begin{array}{c}
  \varphi  \\ \psi   
\end{array}
\right)_t
+ \nabla\cdot
\left(
\begin{array}{c}
   -\nabla \varphi  \\ -\nabla \psi  + u \nabla \varphi
\end{array}
\right)
- \left(
\begin{array}{c}
  \nabla v \cdot \nabla \varphi  \\ 0   
\end{array}
\right)
-
\left(
\begin{array}{cc}
  R_u(u,v) & S_u(u,v,f) \\  R_v(u,v) & S_v(u,v,f)  
\end{array}
\right)
\left(
\begin{array}{c}
  \varphi  \\ \psi   
\end{array}
\right)
\\ \hueco\dis
=
\frac{1}{T\, |\Omega_o|}\left(
\begin{array}{c}
  F_u(u,v)  \, 1_{\Omega_o} \\ F_v(u,v) \, 1_{\Omega_o}  
\end{array}
\right)\,, 
\ea
\eeq
subject to the boundary conditions
$$
-\nabla \varphi \cdot \n \big|_{\partial\Omega}
\,=0 , \quad 
(-\nabla \psi + u \nabla \varphi ) \cdot \n \big|_{\partial\Omega} + P_v(v,g) \, 
\psi \, 1_{\partial\Omega_c}\,=\,0\,,
$$
and the final time conditions:
$$
\varphi(T,\x)=\frac{1}{ |\Omega_o|}G_u(u(T,\x),v(T,\x)) \, 1_{\Omega_o}
\quad
\mbox{ and }
\quad 
\psi(T,\x)= \frac{1}{ |\Omega_o|}G_v(u(T,\x),v(T,\x))\, 1_{\Omega_o}\,. 
$$
By reformulating the problem via the so-called reduced cost functional 
$$
\widetilde J (f,g) = J\Big(\big(u(f,g),v(f,g)\big),(f,g)\Big)
$$
it is possible to identify the gradient of $\widetilde J$ using the 
Riesz's Identification Theorem  with respect to the inner product of $L^2\big((0,T)\times \Omega_c\big)\times L^2\big((0,T)\times \partial\Omega_c\big)$ by the following expression: 
$$
\widetilde J' (f,g) =
\left(
\begin{array}{c}
  \widetilde J_f (f,g)  \\
   \noalign{\vspace{-2ex}}\\
     \widetilde J_g (f,g)
\end{array}
\right)
=
\left(
\begin{array}{c}
  S_f(u,v,f) \, \psi \, 1_{\Omega_c}+  \frac{\alpha_f}{T\, |\Omega_c|}  f   \, 1_{\Omega_c}
   \\
   \noalign{\vspace{-2ex}}\\
  - P_g(v,g) \, \psi \, 1_{\partial\Omega_c} +  \frac{\alpha_g}{T\, |\partial\Omega_c|} g \, 1_{\partial\Omega_c}
\end{array}
\right) .
$$
 Consequently, when there are no constraints  on the controls, the first-order  optimality condition can be written as
$$
\widetilde J' (f,g) =
\left(
\begin{array}{c}
  S_f(u,v,f) \, \psi \, 1_{\Omega_c}+ \frac{\alpha_f}{T\, |\Omega_c|} f   \, 1_{\Omega_c}
   \\
   - P_g(v,g) \, \psi \, 1_{\partial\Omega_c} + \frac{\alpha_g}{T\, |\partial\Omega_c|} g \, 1_{\partial\Omega_c}
\end{array}
\right)
=
\left(
\begin{array}{c}
  0
   \\
 0
\end{array}
\right)\,,
$$
which depends on the state problem for $(u,v)$ and the adjoint one for $(\varphi,\psi)$. 
For the Robin boundary control  case,  the  constraint  $g\ge 0$ is imposed, and then the optimality condition $\widetilde J_g(f,g)=0$ must be changed by  
$$
\widetilde J_g(f,g)(\bar g -g)\ge 0\quad \hbox{for any} \quad \bar g\ge 0\,.
$$

 \subsection{Study case:  Quadratic tracking cost with  bilinear distributed control  and boundary control, of either bilinear or Robin type.}
 
Let's now introduce the particular optimal control problem that we focus on this work. The target dynamics for the cell density $u$ is a variable called $u_d:(0,T)\times \Omega_o\to \mathbb{R}$ and the problem is defined by taking 
$$
F(u,v)= \frac12 (u-u_d)^2 \, 1_{\Omega_o}\,,
\quad 
G(u,v)=0\,, 
\quad 
R(u,v)=0
\quad\mbox{ and }\quad 
S(u,v,f)=-v + u + f\, v\, 1_{\Omega_c}\,.
$$
Then, the cost functional introduced in \eqref{eq:CostFunctional} reads
\beq
J(f,g)
\,=\,
\frac{1}{2\,T\, |\Omega_o|}\int_0^T\!\!\int_{\Omega_o}(u-u_d)^2 d\x dt
+ \frac{\alpha_f}{2\,T\, |\Omega_c|}\int_0^T\!\!\int_{\Omega_c} f^2 d\x dt
+ \frac{\alpha_g}{2\,T\, |\partial\Omega_c|} \int_0^T \int_{\partial\Omega_c} g^2 d\sigma dt\,.
\eeq
subject to the initial boundary valued PDE problem
\beq\label{eq:PDEsystem}
\left\{\ba{rcl}
u_t + \nabla\cdot (-\nabla u + u \nabla v) 
&=&
0\,,
\\ \hueco
v_t + \nabla\cdot (-\nabla v) +v - u
&=&
 f\, v\, 1_{\Omega_c}\,,
 \\ \hueco
(-\nabla u + u \nabla v)\cdot \n \big|_{\partial\Omega}
=0,
\quad
(-\nabla v)\cdot \n \big|_{\partial\Omega}
&=&
-P(v,g\, 1 _{\partial\Omega_c})\,,
 \\ \hueco
 u(0,\x)=u_0(\x)\,,\quad v(0,\x)&=&v_0(\x) .
\ea\right.
\eeq
where for simplicity, we have taken all reaction coefficients equal to value $1$. 
The choice of $S(u,v,f)$ implies that we are imposing a distributed bilinear control term $f\, v\, 1_{\Omega_c}$, which models the injection of chemical in locations where $f> 0$ and subtraction in regions where $f< 0$. 
Moreover, by taking $G(u,v)=0$ we are considering that there is no final time control, that is, the system is just tracking the quadratic control. Furthermore, by considering $R(u,v)=0$ we are imposing that there is no reaction in the $u$-equation, hence one has conservation of the total quantity of  cells, that is, the following relation holds:
\begin{equation} \label{cons-u}
\int_\Omega u(t,\cdot)d\x=\int_\Omega u_0(\cdot)d\x\,,\quad \forall\, t>0\,.
\end{equation}

We complement the study by considering two different type of boundary controls: Robin and bilinear boundary controls. 
In particular, these boundary control are modeled using the following expressions:
\beq\label{eq:typesBoundaryControls}
P(v,g)
\,=\,
\left\{\ba{ll}
\sigma(g-v) & \mbox{Robin boundary control}\,,
\\ \hueco
 g\, v & \mbox{bilinear boundary control}\,.
\ea\right.
\eeq

In the Robin case, $g\ge0$ in $(0,T)\times\partial\Omega_c $ models the amount of external chemical added in the system and parameter $\sigma>0$ represents the permeability coefficient. This type of boundary control has been studied in \cite{LebiedzMaurer}. On the other hand, in the bilinear boundary control case, $g\ge0$ represents that there is injection of chemical thorough the boundary, and $g\le0$ represents that chemical is being extracted.  This type of boundary control has been studied in 
\cite{LebiedzMaurer}. In any case, the amount of variable $v$ evolves as given in the following expression:
\begin{equation} \label{total-v}
\frac{d}{dt}  \int_\Omega v(t,\cdot) \, d\x
+ \int_\Omega v(t,\cdot) \, d\x
= \int_\Omega u_0  \, d\x
+ \int_\Omega f \, v \, 1_{\Omega_c}  \, d\x
+ \int_{\partial\Omega_c} P(v,g) \, d\x \,.
\end{equation}

\section{Discretize-then-Optimize approach.}\label{sec:DiscOpt}

In the following we present our algorithm in $1D$ domains, but the ideas can be extended naturally to higher dimensions without loss of generality. We first present an upwind FV numerical scheme  to approximate the state variables and we establish the well-posedness of the system. Then we compute the associate discrete adjoint problem, and from there we derive an explicit formula to compute the discrete gradient. Additionally we review the main steps of the descend method that we adopt, the  Adam algorithm.

\

We consider the Keller-Segel chemotaxis system defined in the interval domain $Q=(-L,L)\times (0,T)$, that is $\Omega=(-L,L)$. Then, the initial-boundary valued controlled PDE problem reads:
\beq\label{origsys}
\left\{\ba{rcl}
u_t- u_{xx}+(u \,v_x)_x
&=&
0\,,
\\ \hueco
v_t-  v_{xx}+ v-u
&=&f\,v\, 1_{\Omega_c}, 
\ea\right.
\eeq
with boundary conditions
$$
\ba{rcl}
(- u_x+u \,v_x)(t,-L)
&=&
(- u_x+u \,v_x)(t,L)
=
0\,,
\\ \hueco
v_x(t,-L)
&=&
-P(v(t,-L),g(t,-L)),
\\ \hueco
- v_x(t,L)
&=&
-P(v(t,L),g(t,L))\,,
\ea
$$
initial conditions
$$
u(0,x)=u_0(x), \quad v(0,x)=v_0(x).
$$
and the continuous cost functional:
\begin{align*}
    J(u,f,g)&=\frac{1}{2} \frac{1}{T\, |\Omega_o|}\int_0^T\!\!\int_{\Omega_o}(u-u_d)^2  \, dx dt
    +\frac{\alpha_f}{2} \frac{1}{T\, |\Omega_c|}\int_0^T\!\!\int_{\Omega_c} f^2   \, dx dt\\
    & +\frac{\alpha_g}{2} \frac{1}{T} \int_0^T \left(g(t,-L)^2+g(t,L)^2\right)  \, dt. 
\end{align*}


 We discretize state problem \eqref{origsys}  using an upwind finite volume scheme in space and a semi-implicit finite difference in time, with an upwind approximation for the distributed control term. To minimize the cost functional, we will use a gradient descent method (Adam  scheme) with the identification of gradient of the reduced functional $\widetilde J (f,g)$ given above, via the computation of the discrete adjoint problem for $(\varphi, \psi)$.

 To do this, we divide $\overline\Omega$ into $J$ intervals $K_j=[x_{j-1},x_j]$ of length $dx=2L/J$, such that $\overline\Omega=\bigcup_{j=1}^J K_j$. 
 We also divide $[0,T]$ into $N$ intervals $I_n=[t_{n-1},t_n]$ of length $dt=T/N$, such that $[0,T]=\bigcup_{n=1}^N I_n$. 
We then first approximate the state system for $(u,v)$, and then the adjoint system for $(\varphi,\psi)$. 

\

We consider the discrete space $\mathbb{P}_0[t,x]$ to approach all variables, controls, states and adjoint variables, and the discrete $L^2$ scalar product defined as
$$
\int_0^T\!\!\int_{-L} ^L u\, v=\sum_{n=1}^N\sum_{j=1}^J dt\, dx\, u_j^n v_j^n,
\quad\quad\quad
\forall\, u=(u_j^n),v=(v_j^n)\in\mathbb{P}_0[t,x]\,.
$$

\

Then, for any discrete controls $f=(f_j^n)\in\mathbb{P}_0[t,x]$ and $g_k=(g_k^n)\in\mathbb{P}_0[t]$ (with $k=1,J$ denoting the two endpoints of the boundary), we consider the corresponding discrete states $u=(u_j^n),v=(v_j^n)\in\mathbb{P}_0[t,x]$ and  the discrete cost functional:
$$
    J(u,f,g)=\frac{1}{2} \frac{1}{T\, |\Omega_o|}\int_0^T\!\!\int_{-L}^L (u-u_d)^2 1_{\Omega_o}  
    +\frac{\alpha_f}{2} \frac{1}{T\, |\Omega_c|} \int_0^T  \!\!\int_{-L}^L f^2  1_{\Omega_c}
    +\frac{\alpha_g}{2} \frac{1}{T} \int_0^T (g_1^2+g_J^2)\,. 
$$

\subsection{Numerical scheme for the state variables}

In order to describe the upwind FV scheme, given an interval $K_j$ we denote its index set of adjacent interior intervals to $K_j$ as 
$${\cal B}_j^o=\{ j-1,j+1\}\cap \{ 2,\cdots, J-1\}$$
and the corresponding boundary intervals as 
$${\cal B}_j^\partial=\{ j-1,j+1\}\cap \{ 1,J\}.$$
We can define the linear and decoupled (first $v^n_j\approx v|_{I_n\times K_j}$ after $u^n_j\approx u|_{I_n\times K_j}$) state discrete problem as:

\

\textbf{Initialization:} Let
$
u_j^0=\oint_{I_j}u_0$ and $v_j^0=\oint_{I_j}v_0$ $\forall\, j=1,\dots,J.
$

\

\textbf{Step $n$:}  a) Given $(u_j^{n-1})_j$ and $(v_j^{n-1})_j$, and the control $(f_j^{n})_j$, 
compute $(v_j^{n})_j$  solving for all $j=1,\dots,J$:

\beq\label{eq:scheme_v}
\ba{c}\dis
dx\frac{v_j^n-v_j^{n-1}}{dt}
+\sum_{k\in {\cal B}_j^o }    \frac{v_{j}^n-v_k^{n}}{dx}
    + dx\, v_j^n - dx\, u_j^{n-1}
\\ \hueco\dis
    =dx [ (f^n_j)_+v^{n-1}_j+(f_j^n)_-v^n_j]\, 1_{\Omega_c}
    + \sum_{k\in {\cal B}_j^\partial } P(v,g)^n_k\, 1_{\partial\Omega_c}\,.
\ea
\eeq
    
In the case of Robin control,  an implicit approximation of the boundary term is considered 
    $$P(v,g)^n_k=\sigma(g^n_k-v^n_k)
    $$ and for the bilinear  boundary control  we take the non-centered in time approximation 
    $$P(v,g)^n_k= (g^n_k)_+ v^{n-1}_k + (g^n_k)_- v^{n}_k\,.
    $$

\
    
  b) Given $(u_j^{n-1})_j$ and $(v_j^{n})_j$, compute $(u_j^{n})_j$  solving 
      

\beq\label{eq:scheme_u}
dx\frac{u_j^n-u_j^{n-1}}{dt}
+\sum_{k\in {\cal B}_j^o }  \left\{  \frac{u_{j}^n-u_k^{n}}{dx}
+\left(\frac{v_{k}^n-v_j^{n}}{dx}\right)_+u_j^n  +\left(\frac{v_{k}^n-v_j^{n}}{dx}\right)_-u_{k}^n \right\}=0 \,. 
\eeq

\

\begin{remark}

Scheme \eqref{eq:scheme_v} satisfy the following time dynamics of global amount of chemical, which is a discrete version of \eqref{total-v},
$$
\sum_{j=1}^J dx\, \frac{v_j^n - v_j^{n-1}}{dt} 
+\sum_{j=1}^J dx\, v_j^n 
= 
 \sum_{j=1}^J dx\, [(f^n_j)_+v^{n-1}_j+(f_j^n)_-v^n_j ]\, 1_{\Omega_c}
+ \sum_{k=1,J} P(v,g)^n_k \, 1_{\partial\Omega_c}\,.
$$

\

Moreover, adding scheme \eqref{eq:scheme_u} in all volumes $j=1,\cdots, J$, the effects of diffusion and chemotaxis cancel, then solutions of \eqref{eq:scheme_u} satisfy the following conservation property, which is a discrete version of  \eqref{cons-u}:
\begin{equation} \label{cons-u-dx}
\sum_{j=1}^J dx\, u_j^n = \sum_{j=1}^J dx\, u_j^{n-1} .
\end{equation}

\end{remark}

\

At this point, we are in position to prove the well-posedness of the scheme. 
\begin{theorem}\label{th:well-posed-scheme}
The decoupled and linear scheme \eqref{eq:scheme_v}-\eqref{eq:scheme_u} is uniquely solvable, positivity-preserving and mass-conservative
\end{theorem}
\begin{proof} 
 The proof of the theorem is divided in two steps:

\

{\it Step 1.} Existence and uniqueness of  $v^n\ge 0$ solution of \eqref{eq:scheme_v}  if $u^{n-1},v^{n-1}\ge 0$. 

\

In fact, solving $(v^n_j)_j$ is equivalent to solve a linear quadratic linear system with symmetric and positive definite matrix which also has positive strictly dominant  diagonal. In addition,  to get $v^n_j\ge 0$ unconditionally (without imposing constraints),  it is  helpful the explicit-implicit approximation of the distributed and boundary bilinear control given by 
$$
[(f^n_j)_+v^{n-1}_j+(f_j^n)_-v^n_j ]\, 1_{\Omega_c}
\quad \hbox{and} \quad
 (g^n_k)_+ v^{n-1}_k + (g^n_k)_- v^{n}_k
$$
and the  implicit approximation of the Robin  boundary control
$$
\sigma(g^n_k-v^n_k)\,.
$$

Note that the right hand side vector is nonnegative if $u^{n-1},v^{n-1}\ge 0$.

%
%

\

{\it Step 2.}
  Existence and uniqueness of $u^n\geq0$ solution of \eqref{eq:scheme_u} if $u^{n-1},v^{n}\ge 0$. 
  
  \
 
 The linear scheme  for $u$ \eqref{eq:scheme_u} can be rewritten as the algebraic linear system $A\, u+C\!h(v^n)\, u =b(u^{n-1})\in \mathbb{R}^J$, with $A$ (symmetric) definite positive and  strictly dominant diagonal matrix (due to A being the approximation of the time derivate and diffusion terms), 
  and the chemotaxis matrix  $C\!h(v^n)$ which is conservative (the sum by columns vanish) and satisfies   the  property  $(C\!h(v^n) u_+,u_-)\ge 0 $. Consequently,  it suffices to prove that the matrix $A+ Ch(v^n) $ is regular. 
 
%
\begin{lemma}  \label{sol-pos}
 For any $b\in \mathbb{R}^J_+$, there exists $u\in \mathbb{R}^J_+$ such that $A\, u+Ch(v^n)\, u =b$.
 \end{lemma} 
 \begin{proof}  We will use Leray-Schauder theorem in $\mathbb{R}^J$ related to the operator
 $$
 R: w\in \mathbb{R}^J \to u\in \mathbb{R}^J 
 \quad \hbox{solution of} \quad
 A\,u=b-Ch(v^n)\,w_+ \,.
 $$
 One has that $R$ is a continuos function and all possible fixed-points of  $\lambda \,R$ ($\lambda \in [0,1]$) satisfy 
 $$A\,u + \lambda \,Ch(v^n)\,u_+=\lambda \,b\,,
 $$
 that is $u=(u_j)\in \mathbb{R}^J$ is solution of the following  ``truncated" scheme 
 $$
\frac{dx }{dt} u_j
+\sum_{k\in {\cal B}_j^o }  \left\{  \frac{u_{j}-u_k}{dx}
+\lambda \left(\frac{v_{k}^n-v_j^{n}}{dx}\right)_+(u_j)_+  
+\lambda \left(\frac{v_{k}^n-v_j^{n}}{dx}\right)_-(u_{k})_+ \right\}
= \lambda\, b_j\, . 
 $$
  In particular, by testing by $u_-$, one has $u\ge 0$ (see \cite{Saito}). On the other hand, 
  adding in all volumes $j=1,\cdots, J$, the effects of diffusion and chemotaxis cancel (see \eqref{cons-u-dx}), yielding to
  
 $$
 \frac{dx}{dt} \|u\|_1= \frac1{dt}\sum_j dx \, u_j =\lambda \sum_j b_j \le \sum_j b_j \,.
 $$
 Then, all possible fixed-points of $\lambda \,R$ are bounded, hence there exists a fixed point of $R$.
 \end{proof}

 \begin{corollary}  For any $b\in \mathbb{R}^J$, there exists $u\in \mathbb{R}^J$ such that $A\, u+Ch(v^n)\, u =b$. Consequently, the matrix $A+ Ch(v^n) $ is regular.
  \end{corollary}
 
  \begin{proof} {\it Step1.} For any $c\in \mathbb{R}^J_-$, there exists $w\in \mathbb{R}^J_-$ such that $A\, w+Ch(v^n)\, w =c$. Indeed, taking $w=-u$ with $u\ge 0$, by using Lemma \ref{sol-pos} one can deduce that any solution of $A\, u+Ch(v^n)\, u =b=-c\ge 0$ satisfy $u=-w\geq0$, hence $w\leq0$.
  
 {\it  Step 2.} For any $b\in \mathbb{R}^J$, we can write $b=b_+ + b_-$, therefore $z=u+w$ satisfies $A\, z+Ch(v^n)\, z =b$, where $u\ge 0$ satisfies $A\, u+Ch(v^n)\, u =b_+$ and  $w\le 0$ satisfies $A\, w+Ch(v^n)\, w =b_-$
  \end{proof} 
  This concludes the proof of Theorem \ref{th:well-posed-scheme}.
  \end{proof} 

\subsection{Discrete sensitivity problem}

We  now compute the linear problem satisfied by the derivative of discrete states  $(u,v)=(u(f,g),v(f,g))$ around a discrete control $(f,g)$ in any direction $(F,G)$, which is the discrete equivalent of section~\ref{sec:subseclinearization}.

\subsubsection{Differentiating with respect to distributed control}

Given a discrete control $f=(f^m_i)\in\mathbb{P}_0[t,x]$ and  
fixing $m$ and $i$, 
we can compute the problem satisfied by $\displaystyle U^n_j=\frac{\partial u^n_j(f,g)}{\partial f^m_i},$ $\displaystyle V^n_j=\frac{\partial v^n_j(f,g)}{\partial f^m_i}$ for any $n\ge m$: 

\

\textbf{Initialization:}
Let $V_j^0=0$,  $U_j^0=0$ $\forall\, j=1,\dots,J.$

\

\textbf{Step $n$:} a) Compute $(V_j^n)_j$ solving:


\beq\label{eq:scheme_Vf}
\ba{c}\dis
dx\frac{V_j^n-V_j^{n-1}}{dt}
+\sum_{k\in {\cal B}_j^o }\frac{V_{j}^n-V_k^{n}}{dx}
+ dx \,V_j^n    - dx \, U_j^{n-1} 
\\ \hueco\dis
    - dx[(f^n_j)_+V^{n-1}_j+(f_j^n)_-V^n_j]\, 1_{\Omega_c}
     - \sum_{k\in {\cal B}_j^\partial }(P_v(v,g) V)^n_k \, 1_{\partial\Omega_c}
\\ \hueco\dis
= dx [ H(f^n_j)v_j^{n-1}+H(-f_j^n)v_j^n ] 
 1_{\Omega_c} \delta_{nm}^{ji}\,.
\ea
\eeq
b) Compute $(U_j^n)_j$ solving:
\beq
\ba{c}\dis\label{eq:scheme_Uf}
dx\frac{U_j^n-U_j^{n-1}}{dt}
+\sum_{k\in {\cal B}_j^o }\left\{\frac{U_{j}^n-U_k^{n}}{dx}
+\left(\frac{v_{k}^n-v_j^{n}}{dx}\right)_+U_j^n  +\left(\frac{v_{k}^n-v_j^{n}}{dx}\right)_-U_{k}^n \right\}
\\ \hueco\dis
+\sum_{k\in {\cal B}_j^o }\frac{V_{k}^n-V_j^{n}}{dx} \Big\{
 u_j^n  H(v_{k}^n-v_j^{n}) 
 +u_{k}^n H(v_{j}^n-v_k^{n}) \Big\}
 =0\,.
\ea
\eeq
Hereafter, $H(\cdot)$ is the Heaviside function with $H(0)=0.5$, and $\delta_{nm}^{ji}$ is the Kronecker delta which is equal to $1$ only if $n=m$ and $j=i$.

\subsubsection{Differentiating with respect to boundary control}

Similarly, given  $g=(g_l^m)\in\mathbb{P}_0[t]$ ($l=1,J$), we can compute the problem satisfied by  $\displaystyle  U^n_j=\frac{\partial u^n_j(f,g)}{\partial g^m_l}, V^n_j=\frac{\partial v^n_j(f,g)}{\partial g^m_l}$ for any $n\ge m$:

\

\textbf{Initialization:}
Let $V_j^0=0$,  $U_j^0=0$ $\forall\, j=1,\dots,J.$

\

\textbf{Step $n$:} 
 a) Compute $(V_j^n)_j$ solving:
%
\beq\label{eq:scheme_Vg}
\ba{c}\dis
dx\frac{V_j^n-V_j^{n-1}}{dt}
   +\sum_{k\in {\cal B}_j^o }\frac{V_{j}^n-V_k^{n}}{dx}
   + dx \,V_j^n    - dx \, U_j^{n-1} 
\\ \hueco\dis
    - dx[(f^n_j)_+V^{n-1}_j+(f_j^n)_-V^n_j]\, 1_{\Omega_c}
     - \sum_{k\in {\cal B}_j^\partial }(P_v(v,g) V)^n_k \, 1_{\partial\Omega_c}
\\ \hueco\dis
    =  \sum_{k\in {\cal B}_j^\partial }
     (P_g(v,g) )_k^n \, 1_{\partial\Omega_c} \delta_{nm}^{kl}\,.
\ea
\eeq
b) Compute $(U_j^n)_j$ solving:
\beq\label{eq:scheme_Ug}
\ba{c}\dis
dx\frac{U_j^n-U_j^{n-1}}{dt}
+\sum_{k\in {\cal B}_j^o }\left\{\frac{U_{j}^n-U_k^{n}}{dx}
+\left(\frac{v_{k}^n-v_j^{n}}{dx}\right)_+U_j^n  +\left(\frac{v_{k}^n-v_j^{n}}{dx}\right)_-U_{k}^n \right\}
\\ \hueco\dis
+\sum_{k\in {\cal B}_j^o }\frac{V_{k}^n-V_j^{n}}{dx} \Big\{
u_j^n  H(v_{k}^n-v_j^{n}) 
+u_{k}^n H(v_{j}^n-v_k^{n}) \Big\}
=0\,.
\ea
\eeq

Concretely, in the case of 
Robin control we have
$$(P_v(v,g)V)^n_k=-\sigma\, V^n_k\quad \hbox{and}  \quad (P_g(v,g))^n_k=\sigma,
$$
 and for the bilinear control case we have
$$(P_v(v,g)V)^n_k= (g^n_k)_+ V^{n-1}_k + (g^n_k)_- V^{n}_k 
\quad \hbox{and}  \quad (P_g(v,g))^n_k= H(g^n_k) v^{n-1}_k  + H(-g^n_k) v^{n}_k \,.
$$

Previous linearized schemes, for any $F$ and $G$, can be rewritten as a square linear system $A (U,V)^t=b$ (with matrix $A$ and vector $b$  depending on $F$ or $G$ respectively), where the matrix $A$ is the same matrix for the $(u,v)$-scheme \eqref{eq:scheme_v}-\eqref{eq:scheme_u}, which is regular. Then, similarly to Theorem \ref{th:well-posed-scheme}, one has the following result.

\begin{theorem}
The  linearized $(U,V)$ schemes \eqref{eq:scheme_Vf}-\eqref{eq:scheme_Uf} and \eqref{eq:scheme_Vg}-\eqref{eq:scheme_Ug}
are uniquely solvable.
\end{theorem}

%

\subsection{Implementation of the discrete adjoint system}

The implementation of the discrete adjoint system is derived from the  discrete linearized problem satisfied by $(U,V)$ (either scheme \eqref{eq:scheme_Vf}-\eqref{eq:scheme_Uf} or scheme \eqref{eq:scheme_Vg}-\eqref{eq:scheme_Ug}). Then, the adjoint variables 
  $(\boldsymbol{\varphi},\boldsymbol{\psi})$, for $\boldsymbol{\varphi}=(\varphi_j^n)_j^n$ and $\boldsymbol{\psi}=(\psi_j^n)_j^n$ satisfy the following decoupled system (first $\varphi^n_j \approx \varphi |_{I_n\times K_j}$ after $\psi^n_j \approx \psi |_{I_n\times K_j}$):

\

{\bf Initialization:} $\psi_j^{N+1}=0$, $\varphi_j^{N+1}= 0 $, for all $j=1,\dots,J$. 

\

For $n=N,\dots,1$,

{\bf Step} $n$: a)  Given $\psi^{n+1}$ and $\varphi^{n+1}$,  solve  
$\varphi^{n}=(\varphi_j^{n})_j$

\beq\label{eq:schemevarphi}
\ba{c}\dis
dx\frac{\varphi_j^n-\varphi_j^{n+1}}{dt}
+ \sum_{k\in {\cal B}_j^o } \frac{\varphi_{j}^n -\varphi_k^n}{dx}
  - dx\,\psi_j^{n+1}  
 -dx \sum_{k\in {\cal B}_j^o } 
\left(\frac{v_{k}^n-v_j^n}{dx}\right)_+\frac{\varphi_{k}^n-\varphi_j^n}{dx}
\\ \hueco\dis
= \frac{1}{T|\Omega_o|}dx(u_j^n-(u_d)_j^n)\, 1_{\Omega_o},
\ea
\eeq
b) Solve $\psi^{n}=(\psi_j^{n})_j$
\beq\label{eq:schemepsi}
\ba{c}\dis
dx\frac{\psi_j^n-\psi_j^{n+1}}{dt}
+\sum_{k\in {\cal B}_j^o }\frac{\psi_{j}^n-\psi_k^n}{dx}
+ dx \, \psi_j^n
-dx[(f_j^n)_-\psi_j^n + (f_j^{n+1})_+\psi_j^{n+1}] \, 1_{\Omega_c}
\\ \hueco\dis
- \sum_{k\in {\cal B}_j^\partial }(P_v(v,g)^* \psi)^n_k \, 1_{\partial\Omega_c}
+ \sum_{k\in {\cal B}_j^o } \left(H(v_j^n-v_{k}^n)u_{k}^n+H(v_{k}^n-v_j^n)u_j^n\right) 
\frac{\varphi_{k}^n-\varphi_j^n}{dx} \, 1_{\Omega_c}
=0,
\ea
\eeq
where $P_v(v,g)^*$ denotes the adjoint operator of $P_v(v,g)$ with 
$(P_v(v,g)^* \psi)^n_j=-\sigma\, \psi^n_j$ for Robin and  
$(P_v(v,g)^* \psi)^n_j=(g^{n+1}_k)_+ \psi^{n+1}_k + (g^n_k)_- \psi^{n}_k$ for bilinear control, respectively.

\begin{corollary} The corresponding  discrete adjoint problem \eqref{eq:schemevarphi}-\eqref{eq:schemepsi} is uniquely solvable (by duality this  adjoint scheme is  $A^t (\varphi,\psi)^t=(J_u,J_v)^t$).
\end{corollary} 

\subsection{Computation of discrete gradient}

By defining the discrete reduced cost functional as $\widetilde J(f,g)=J(u(f,g),f,g)$, its gradient   with respect to the distributed control $f$ can be identified by the discrete $L^2$ scalar product using the $\mathbb{P}_0[t,x]$ functions:
\begin{align*}
  \widetilde  J_f(f,g)_j^n&=
  \psi_j^n [ H(f_j^n)v_j^{n-1}+H(-f_j^n)v_j^n ]\, 1_{\Omega_c}
    + \frac{\alpha_f}{T\, |\Omega_c|} 
    f_j^n \, 1_{\Omega_c},
\end{align*}
for any  $j=1,\cdots,J$ and $ n=1,\cdots,N$. 
Similarly, the gradient with respect to the boundary control $g$, is identified by
\begin{align*}
    \widetilde J_g(f,g)_k^n&= 
    \psi_k^n 
    P_g(v,g)_k^n\, 1_{\partial\Omega_c}
    + \frac{\alpha_g}{T\, |\partial\Omega_c|} 
    g_k^n \, 1_{\partial\Omega_c},
     \quad k=1,J,\quad n=1,\cdots,N\,.
\end{align*}
Notice that 
$P_g(v,g)_k^n=\sigma$ for Robin control and 
$P_g(v,g)_k^n=H(g_k^n)v_k^{n-1}+H(-g_k^n)v_k^{n}$ for bilinear control.

\subsection{Adam descent gradient algorithm}

In this section we present the details associated with the descent gradient algorithm considered to minimize functional $\widetilde J(f,g)$, which is based on adapting the ideas of the Adam method presented in \cite{KB15} to be applied onto our problem. The Adam method (whose name comes from adaptive moment estimation) is an efficient stochastic optimization method that only requires the computation of first-order gradients and individual adaptive learning rates for different parameters, which are based on estimates of first and second moments of the gradients. To describe the algorithm we need to introduce the following parameters and quantities:
$$
\ba{rcl}
\alpha & \equiv & \mbox{step size for the descent method}
\\
\beta_1, \beta_2& \equiv & \mbox{exponential decay rates for the moment estimates}
 \\
(m^k_1,m^k_2)=\m^k& \equiv & \mbox{first moment vector at step $k$}
 \\
(z^k_1,z^k_2)=\z^k& \equiv & \mbox{second moment vector at step $k$}
\ea
$$
where $\alpha>0$, $\beta_1,\beta_2\in[0,1)$ and we denote by $\widetilde{\m}, \widetilde{\z}$ the corresponding bias-corrected moment estimates. The considered algorithm reads:
\begin{itemize}
\item \textbf{Initialization:} Consider the following initial guesses
\begin{enumerate} 
\item $(f^0,g^0)=(f_0,g_0)$, with $(f_0,g_0)$ being some initial guess of the control functions. If no information is known about these functions, just consider $(f_0,g_0)=(0,0)$. 
\item $\m^0=(0,0)^t$.
\item $\z^0=(0,0)^t$.
\end{enumerate}
\item \textbf{Algorithm:}  Given $(f^k,g^k)$, to compute $(f^{k+1},g^{k+1})$ as follows:  
\begin{enumerate}
\item 
Compute $(u^k,v^k)$ by solving numerical scheme \eqref{eq:scheme_v}-\eqref{eq:scheme_u} using $(f^k,g^k)$. 
\item 
Compute $(\varphi^k,\psi^k)$ by solving numerical scheme \eqref{eq:schemevarphi}-\eqref{eq:schemepsi} using $(f^k,g^k, u^k,v ^k)$. 
\item
Update $\J^k$ using $(f^k,g^k)$ such that 
$$
\J^k=
\left(
\begin{array}{c}
  \widetilde J_f (f^k,g^k)  \\ 
    \noalign{\vspace{-2ex}}\\
     \widetilde J_g (f^k,g^k)
\end{array}
\right)
=
\left(
\begin{array}{c}
  \frac{1}{T\, |\Omega_o|} S_f(u ^k,v^k,f^k) \, \psi^k \, 1_{\Omega_c}+ 
  \frac{\alpha_f}{T\, |\Omega_c|} f^k  \, 1_{\Omega_c}
   \\
     \noalign{\vspace{-2ex}}\\
    - \frac{1}{T\,  |\Omega_o|} P_g(v ^k,g ^k) \, \psi^k \, 1_{\partial\Omega_c} + 
    \frac{\alpha_g}{T\, |\partial\Omega_c|} g^k \, 1_{\partial\Omega_c}
\end{array}
\right)
$$
\item Iterate until $\|\J^{k}\|\leq \texttt{tol}$ for a given tolerance \texttt{tol} or until the number of iterations exceeds a prescribed threshold \texttt{MaxNumIter}.
\item Update the first moment estimate as $\m^k=\beta_1\m^{k-1}+(1-\beta_1)\J^k$.
\item Update the second raw moment estimate as $\z^k=\beta_2\z^{k-1}+(1-\beta_2)(\J^k)^2$ (with $(\J^k)^2$ denoting the element-wise square).
\item  Compute the bias-corrected first moment estimate as $\widetilde\m^k=\m^k/(1-\beta_1^k)$ (with $\beta_1^k$ denoting $\beta_1$ to the power $k$).
\item  Compute the bias-corrected second raw moment estimate as $\widetilde\z^k=\z^k/(1-\beta_2^k)$ (with $\beta_2^k$ denoting $\beta_2$ to the power $k$).
\item Compute $(f^{k+1}, g^{k+1})$ as follows
$$
\left(
\begin{array}{c}
f^{k+1}  
\\ \hueco
g^{k+1}
\end{array}
\right)
=
\left(
\begin{array}{c}
f^{k}
\\ \hueco
g^{k}
\end{array}
\right)
- \alpha 
\left(
\ba{c}
\frac{\widetilde m_1^k }{\sqrt{z_1^k + \epsilon}}
\\ \hueco
\frac{\widetilde m_2^k }{\sqrt{z_2^k + \epsilon}}
\ea
\right)\,,
$$
with $\epsilon$ being a small parameter.
\end{enumerate} 
\end{itemize}

\begin{remark}
In the case of a restriction of the type $g\ge 0$, we need to update $g^{k+1}$ taking only its positive part, that is, taking $(g^{k+1})_+$ in Step 9.
\end{remark}

\section{Numerical Simulations}\label{sec:simulations} 

In this section, we present several numerical experiments to illustrate the behavior of the proposed numerical algorithm. All the simulations have been carried out using MATLAB software \cite{Matlab}. First, in Section~\ref{subsec:simbilinear} we investigate the effect of considering a bilinear distributed control in the dynamics of the system. Then, in Section~\ref{subsec:simboundary} we study the effect of introducing a boundary control on the system. 

In previous sections, in order to simplify the notation we omitted to write the physical parameters of the system, although it is easy to follow where each parameter should appear through the text. System \eqref{eq:PDEsystem} including all the relevant physical parameters can be written as
\beq
\left\{\ba{rcl}
u_t + \nabla\cdot (-D_u\nabla u + \chi u \nabla v) 
&=&
0\,,
\\ \hueco
v_t + \nabla\cdot (-D_v\nabla v) + \lambda v - \mu u
&=&
 f\, v\, 1_{\Omega_c}\,,
\ea\right.
\eeq
where $D_u$, $D_v$ are diffusion coefficients, $\chi$ denotes the chemotactic sensitivity, and parameters $\lambda$ and $\mu$  represent the degradation and production rates of variable $v$, respectively.

Unless otherwise mentioned, the physical and discretization parameters considered for all the simulations are presented in Table~\ref{tab:physdiscparameters} and the parameters related with the Adam method are detailed in Table~\ref{tab:ADAMSparameters}.

\begin{table}[H]
\begin{center}
\begin{tabular}{|c|c|c|c|c|c|c|c|c|c|c|}
\hline
$\Omega$ & $h$ & $[0,T]$ & $\Delta t$ & $\alpha_f$ & $\alpha_g$ 
& $D_u$  & $\chi$ & $D_v$ & $\lambda$ & $\mu$
 \\
\hline
$(-1, 1)$ & $0.02$ & $[0,0.05]$ 	& $0.0005$ & 0 &  0 
& $0.1$ & $1.0$ & $0.1$ & $0.1$ & $1.0$ 
 \\
 \hline
\end{tabular}
\caption{Physical and discrete parameters}
\label{tab:physdiscparameters}
\end{center}
\end{table}

\begin{table}[H]
\begin{center}
\begin{tabular}{|c|c|c|c|c|c|}
\hline
$\alpha$ & \texttt{tol} & \texttt{MaxNumIter}& $\beta_1$ & $\beta_2$ & $\epsilon$	
 \\
\hline
$0.1$ & $10^{-4}$ & $10^5$ & 0.9 & 0.999 & $10^{-8}$
\\
\hline
\end{tabular}
\caption{Parameters for Adam methods}
\label{tab:ADAMSparameters}
\end{center}
\end{table}

In particular, we consider the most challenging situation without regularization terms in the cost functional ($\alpha_f=\alpha_g=0$). Then, the cost functional reduces to
$$
J(f,g)=\frac{1}{2} \frac{1}{T\, |\Omega_o|}\int_0^T\!\!\int_{-L}^L (u-u_d)^2 1_{\Omega_o} dx dt.
$$

\subsection{Approximation of a manufactured solution}\label{subsec:manufacsol}

In this section we study how the algorithm behaves when a control $f$ to generate the desired state $u$ is known a priori. To this end we consider initially a smooth control
$$
f(x,t)
\,=\,
\cos(3\pi x)\cos(20\pi t)\,,
$$
and the initial condition
$$
u_0(\x)
\,=\,
1+\cos(\pi x)
\quad\mbox{ and }\quad
v_0(\x)
\,=\,
3+\cos(\pi x)\,,
$$
and we use them to generate the associated dynamics of variables $u$ and $v$, which are presented in Figure~\ref{fig:Manufactured_Exact}. Then, we use the generated $u$ as the desired state $u_d$, that is, we run the simulations using $u_d=u$ and an initial control $f(x,t)=0$. The results are presented in Figure~\ref{fig:Manufactured_Results}, where we observe in the first row the dynamics of the three variables. As expected the dynamics of $u$ achieve the ones in $u_d$ although the dynamics of $f$ and $v$ are different from the manufactured ones. This is not surprising, as there might be many controls $f$ that make the system to be very close to the desired state $u_d$. Moreover, we present in the second row of Figure~\ref{fig:Manufactured_Results} how the functional $J(f)$ and its gradient evolves with respect to the iterations, and how by perturbing the obtained state as $J(u\pm \delta_n)$ we conclude that the achieved state represents a minimum (at least looking at the constant direction).

\begin{figure}[H]
\begin{center}
\includegraphics[width=0.328\textwidth]{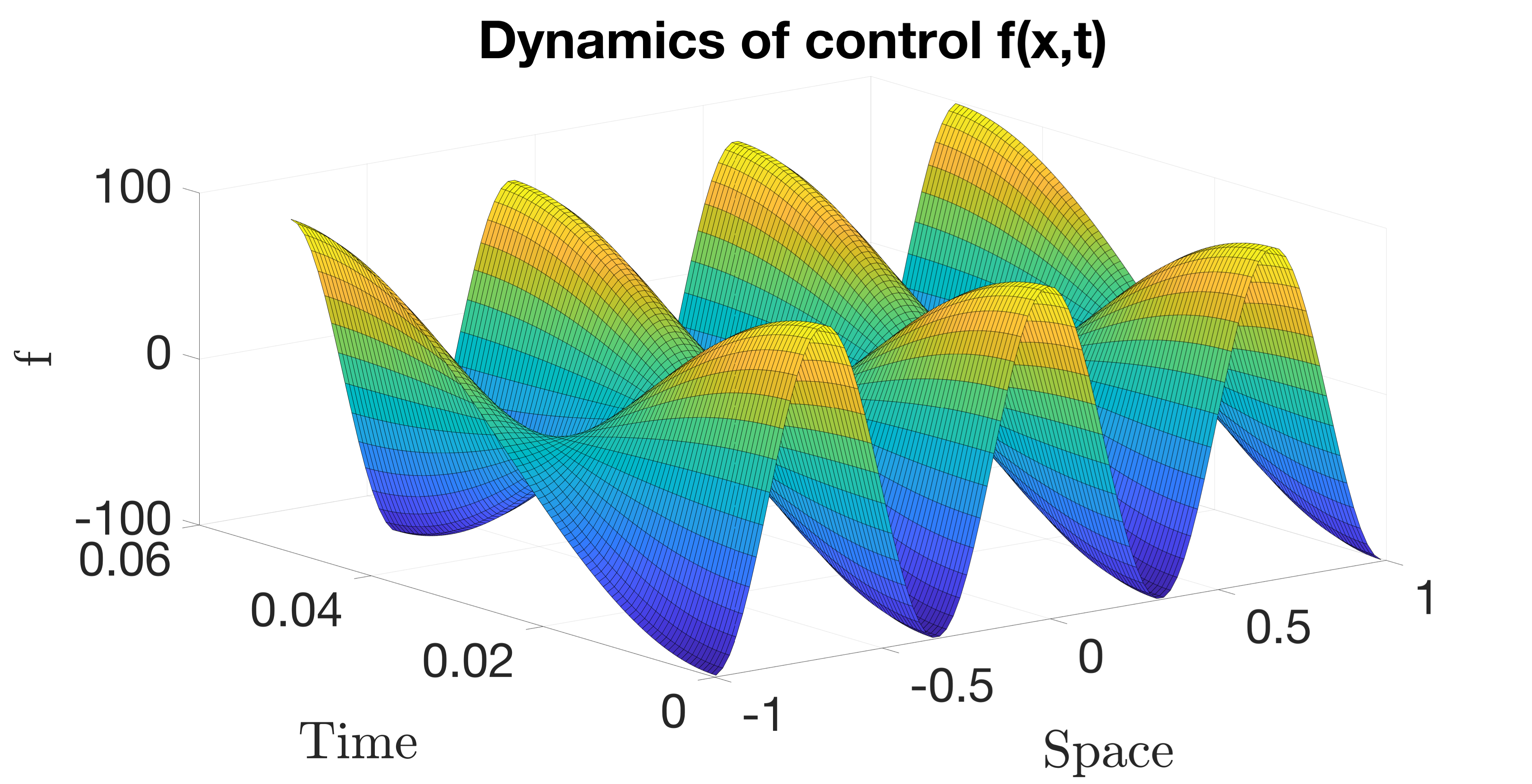}
\includegraphics[width=0.328\textwidth]{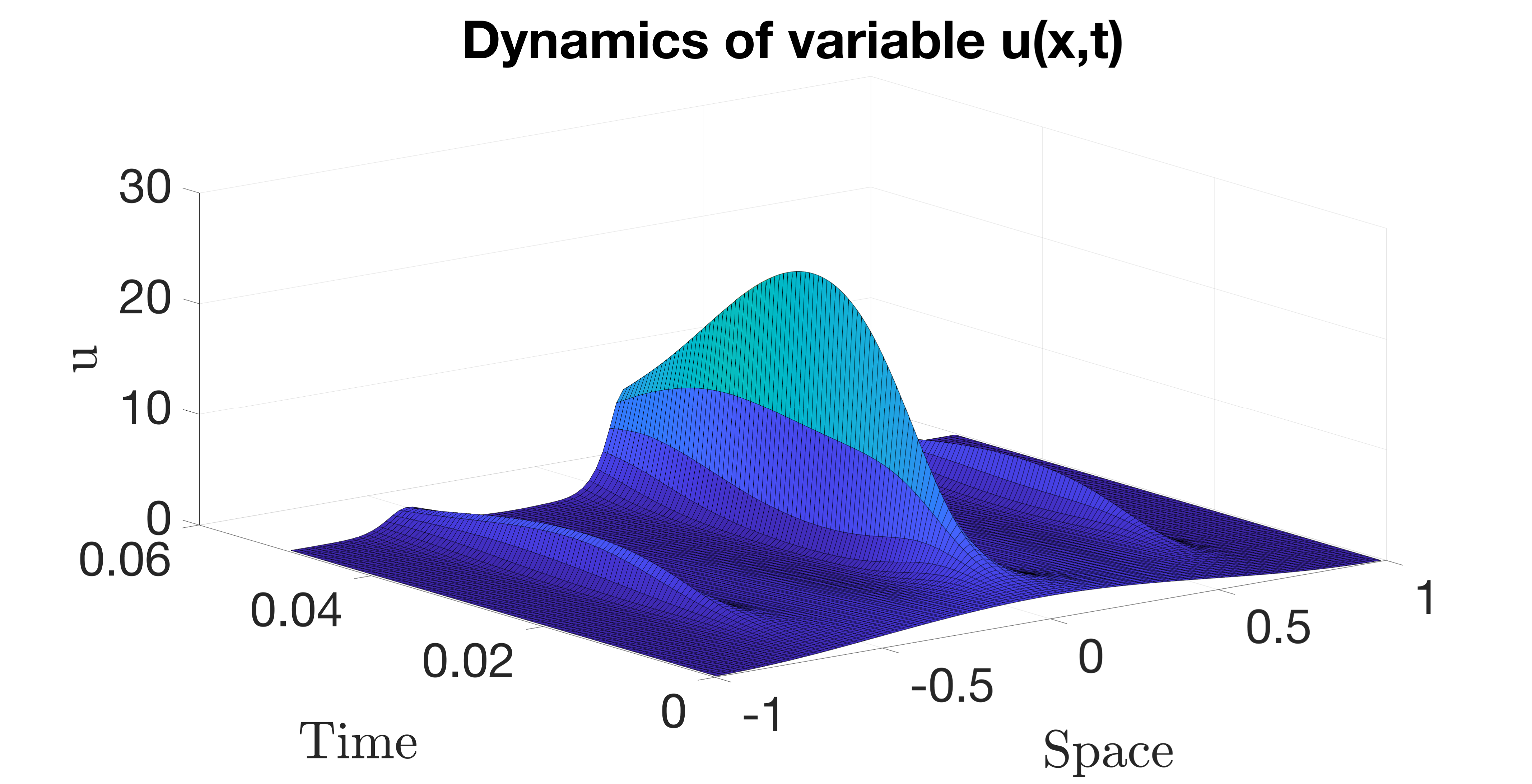}
\includegraphics[width=0.328\textwidth]{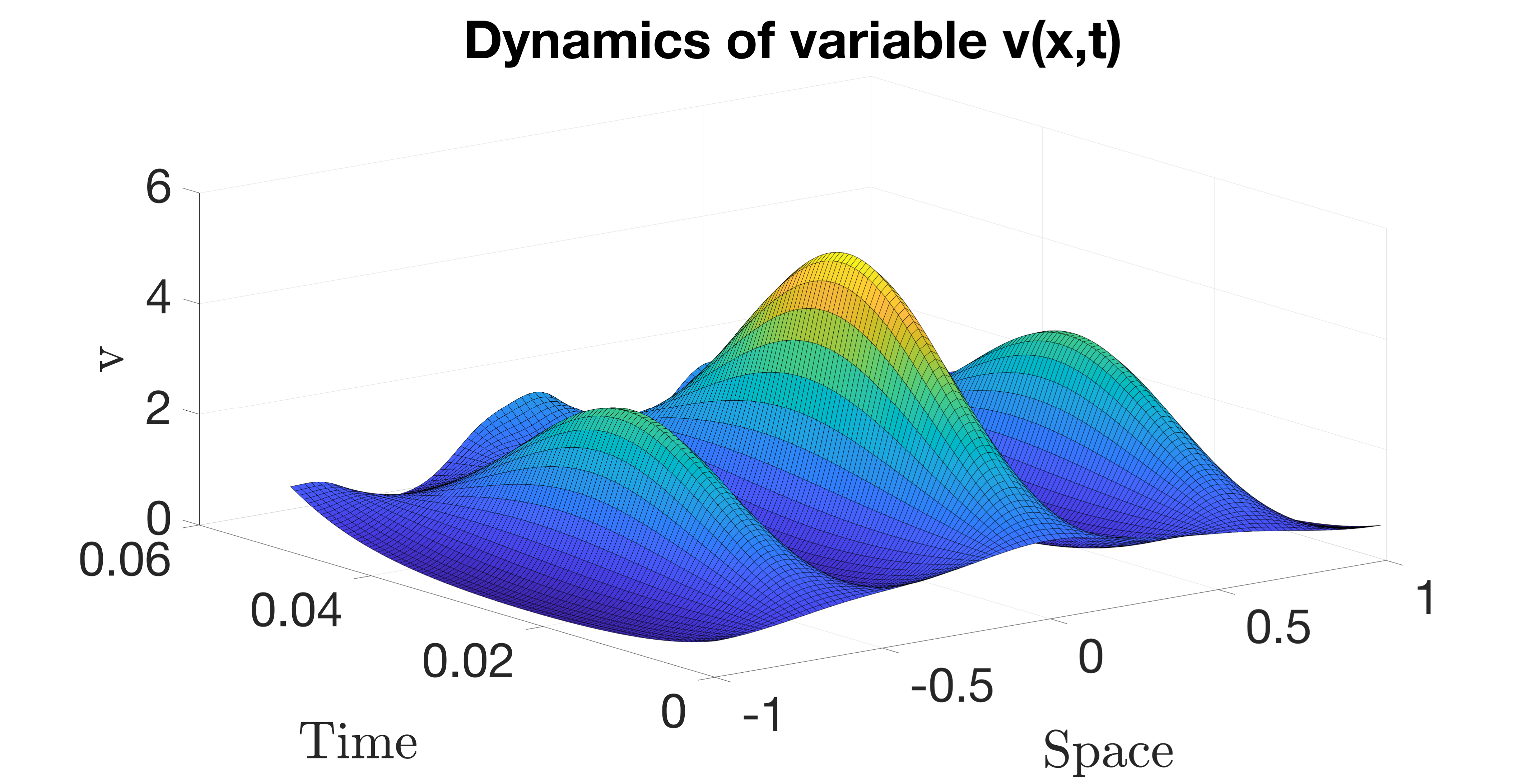}
\caption{Dynamics of the manufactured control $f$ (left) and the associated variables $u$ (center) and variable $v$ (right).}\label{fig:Manufactured_Exact}
\end{center}
\end{figure}

\begin{figure}[H]
\begin{center}
\includegraphics[width=0.328\textwidth]{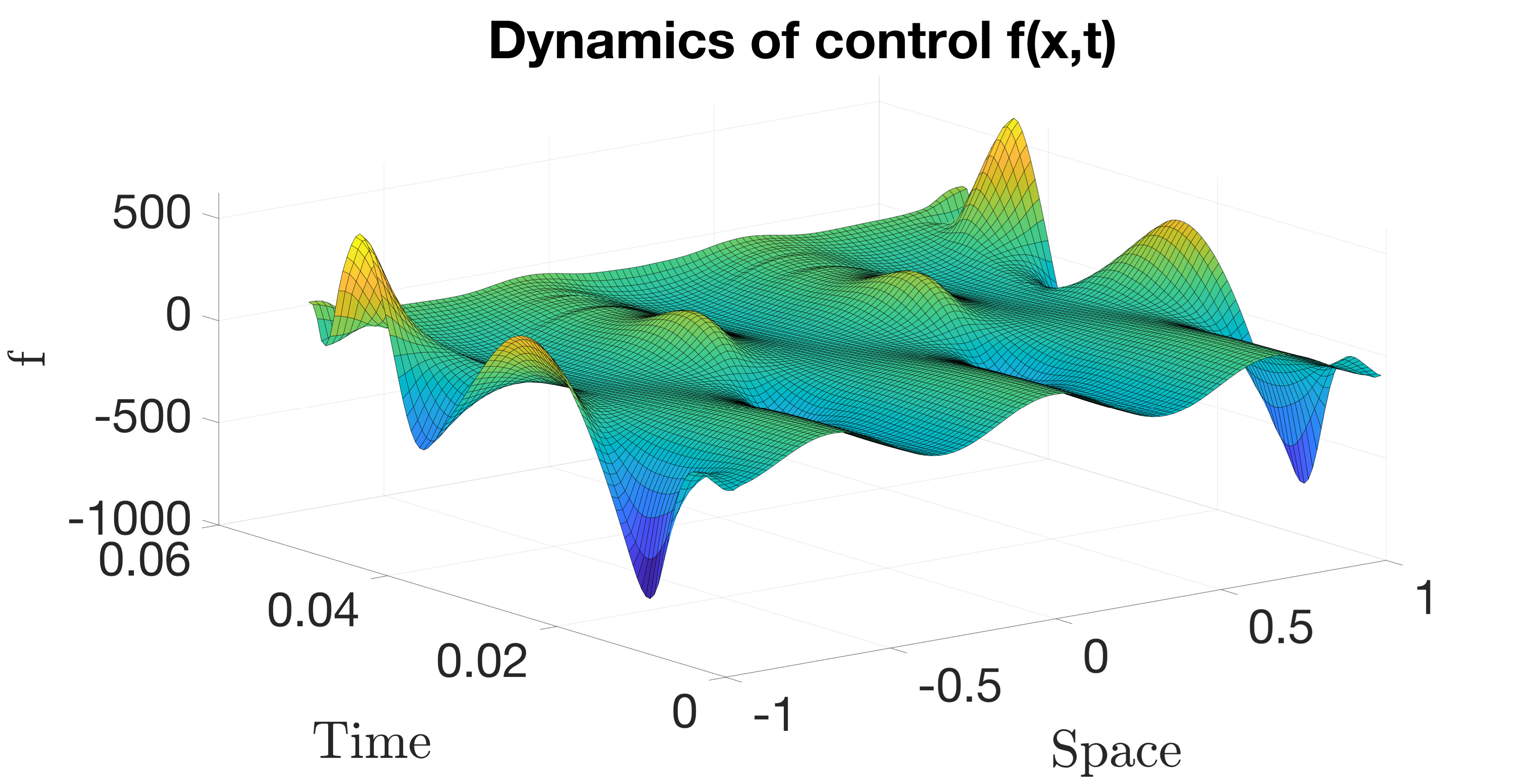}
\includegraphics[width=0.328\textwidth]{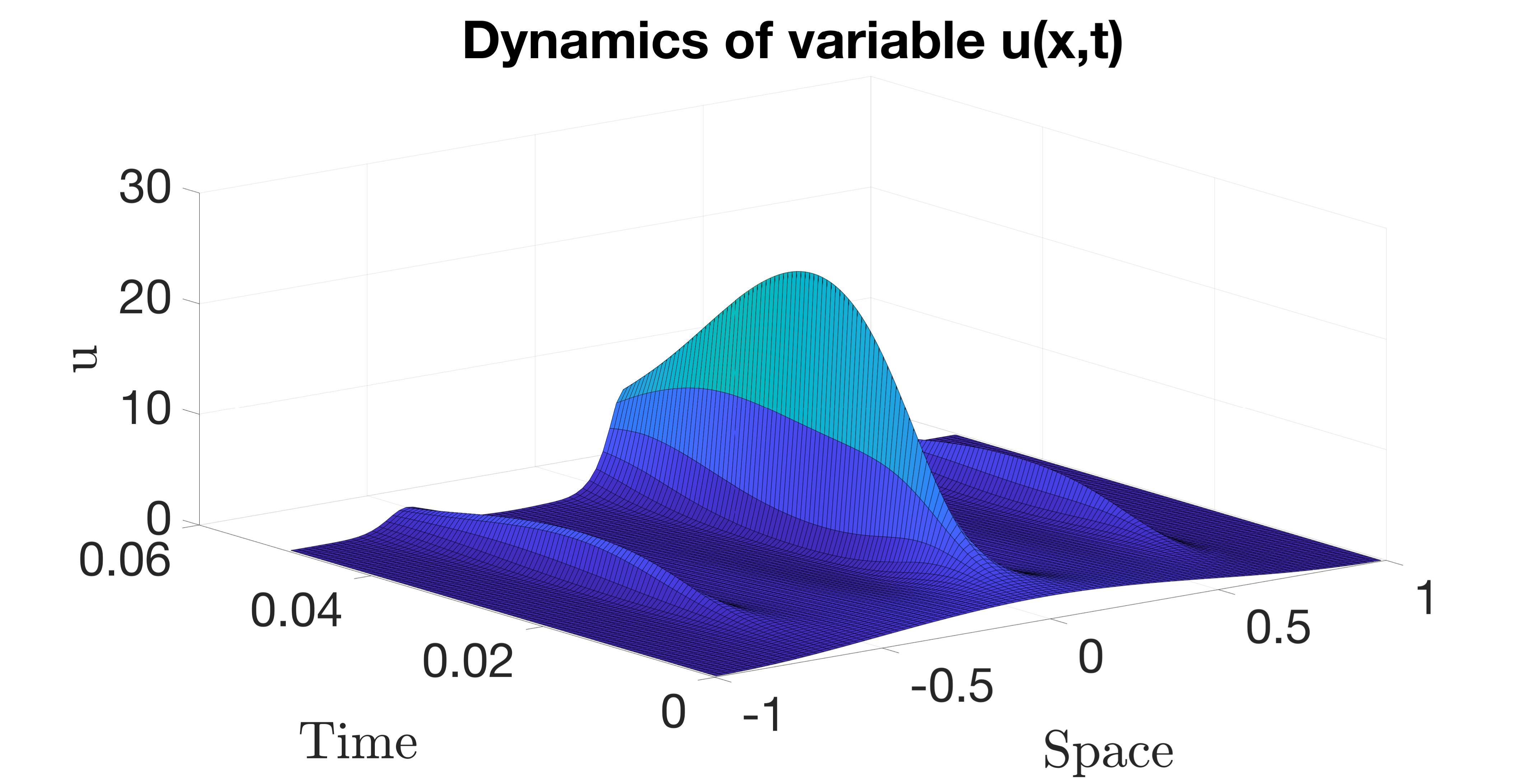}
\includegraphics[width=0.328\textwidth]{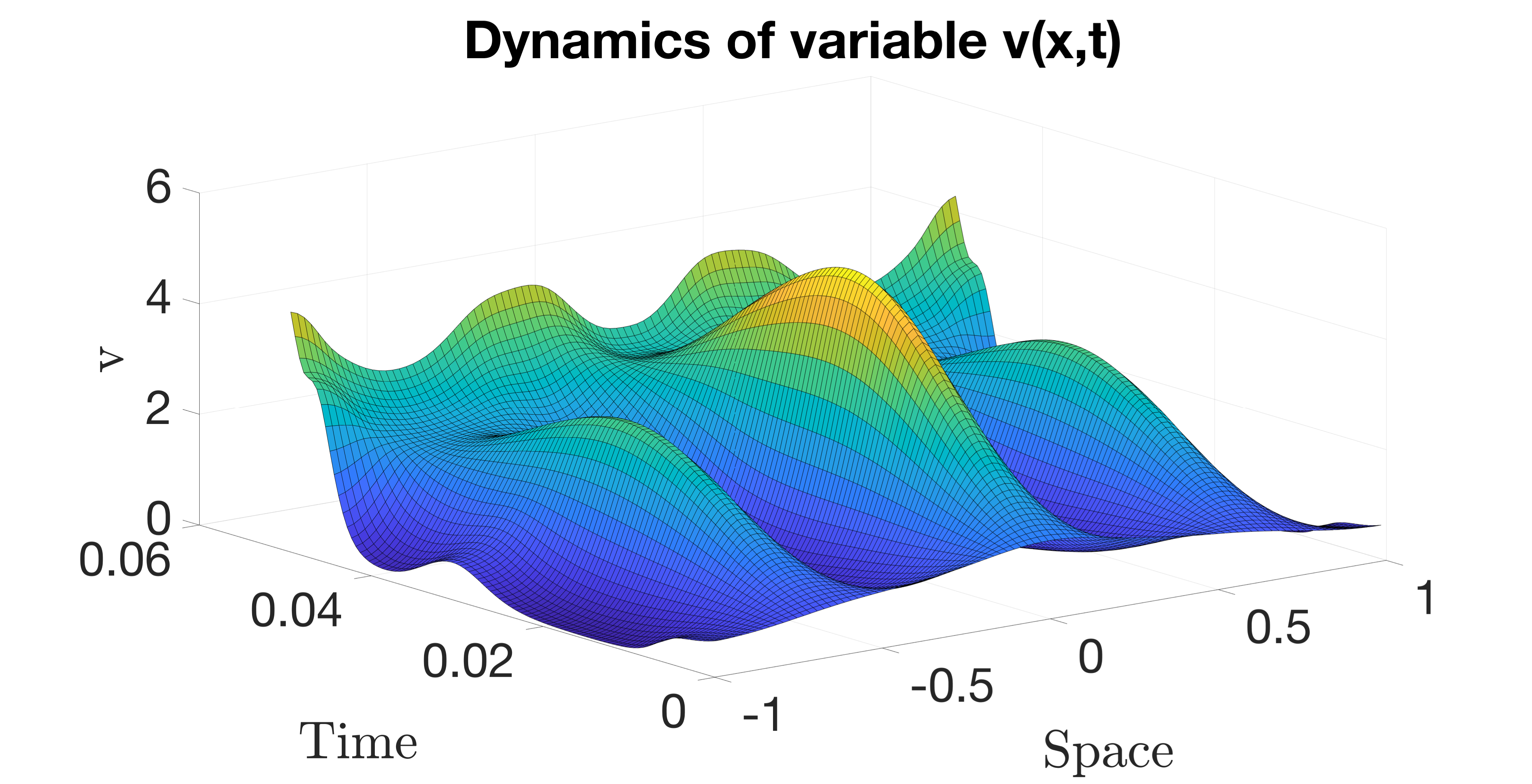}
\\
\includegraphics[width=0.328\textwidth]{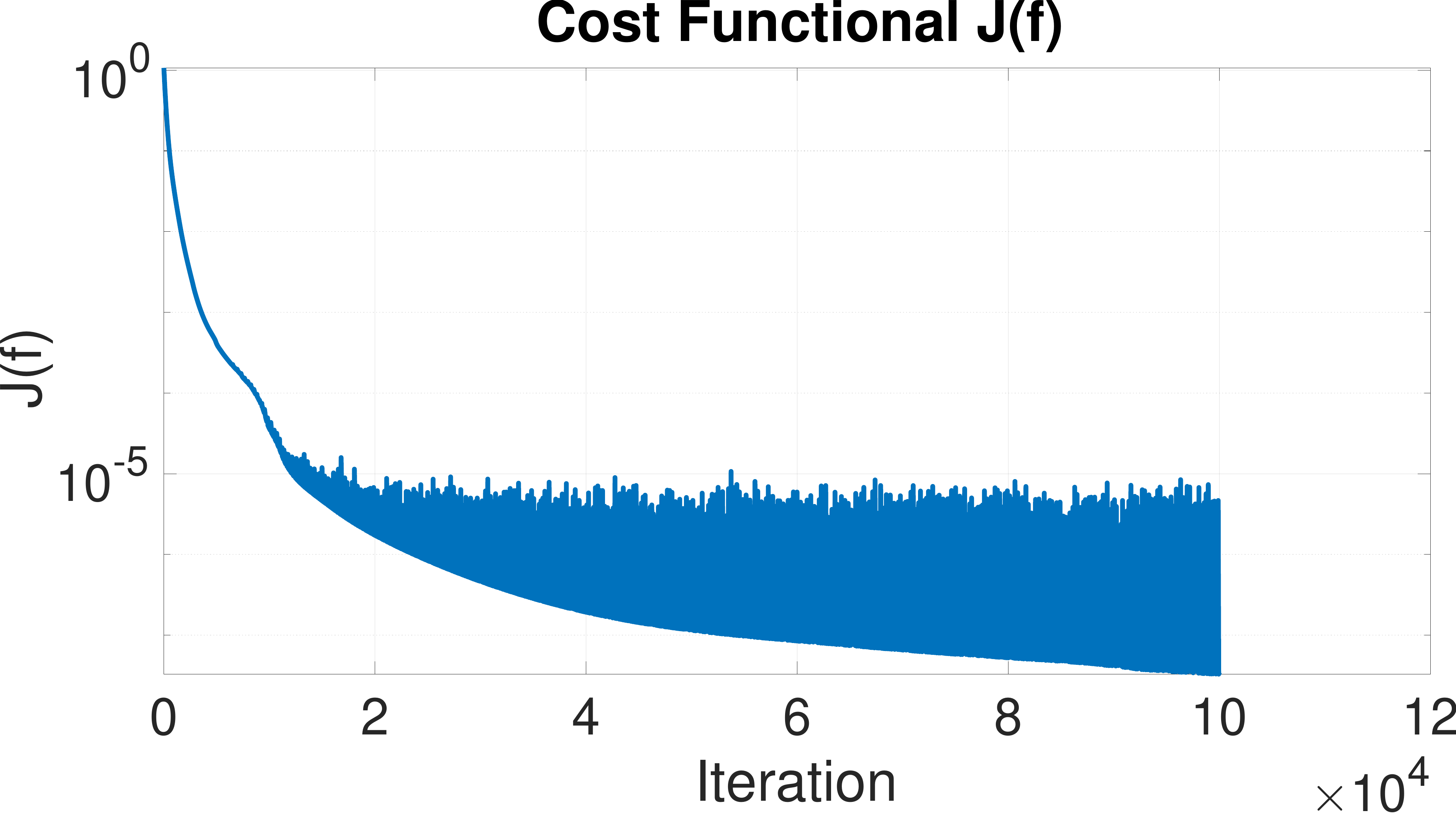}
\includegraphics[width=0.328\textwidth]{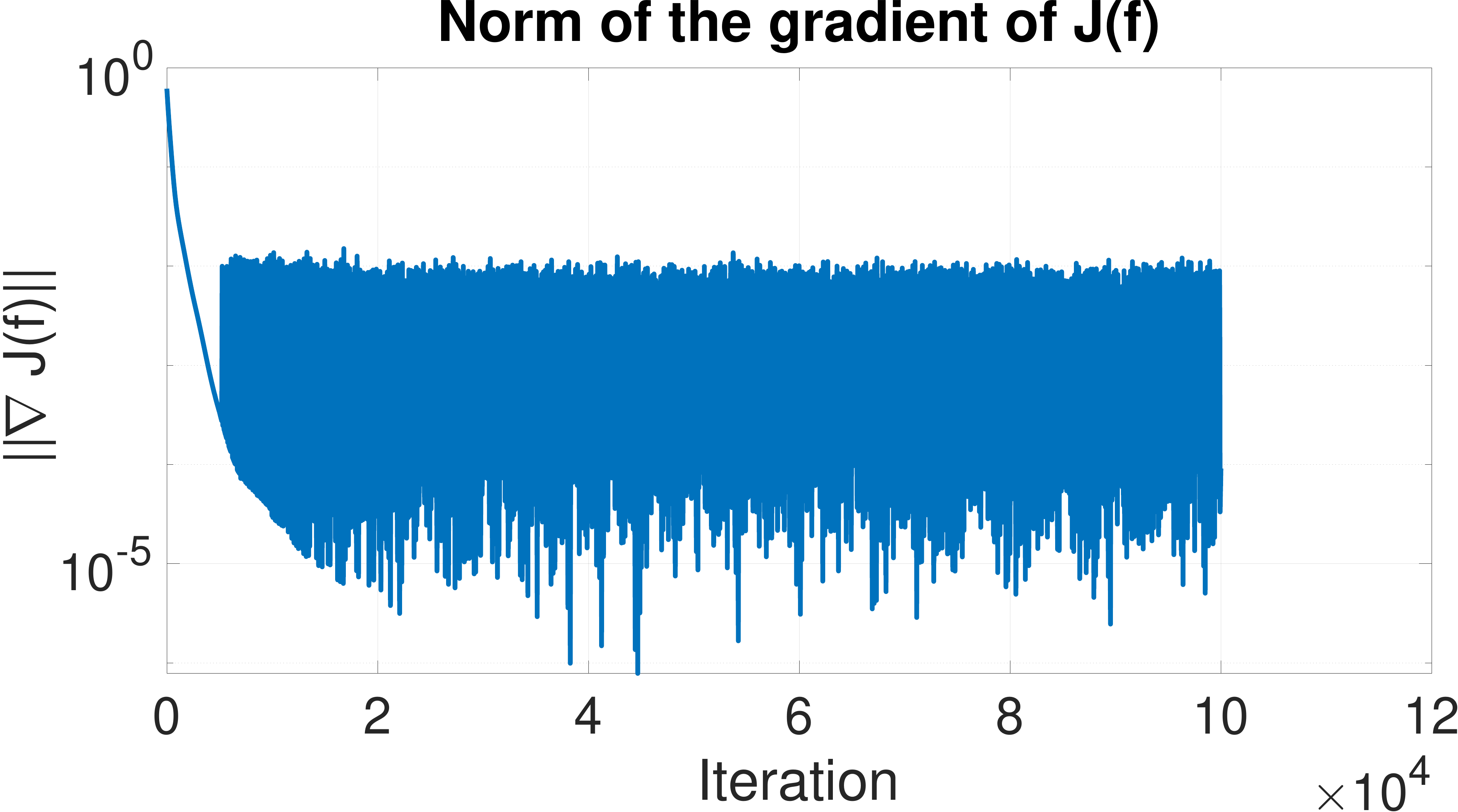}
\includegraphics[width=0.328\textwidth]{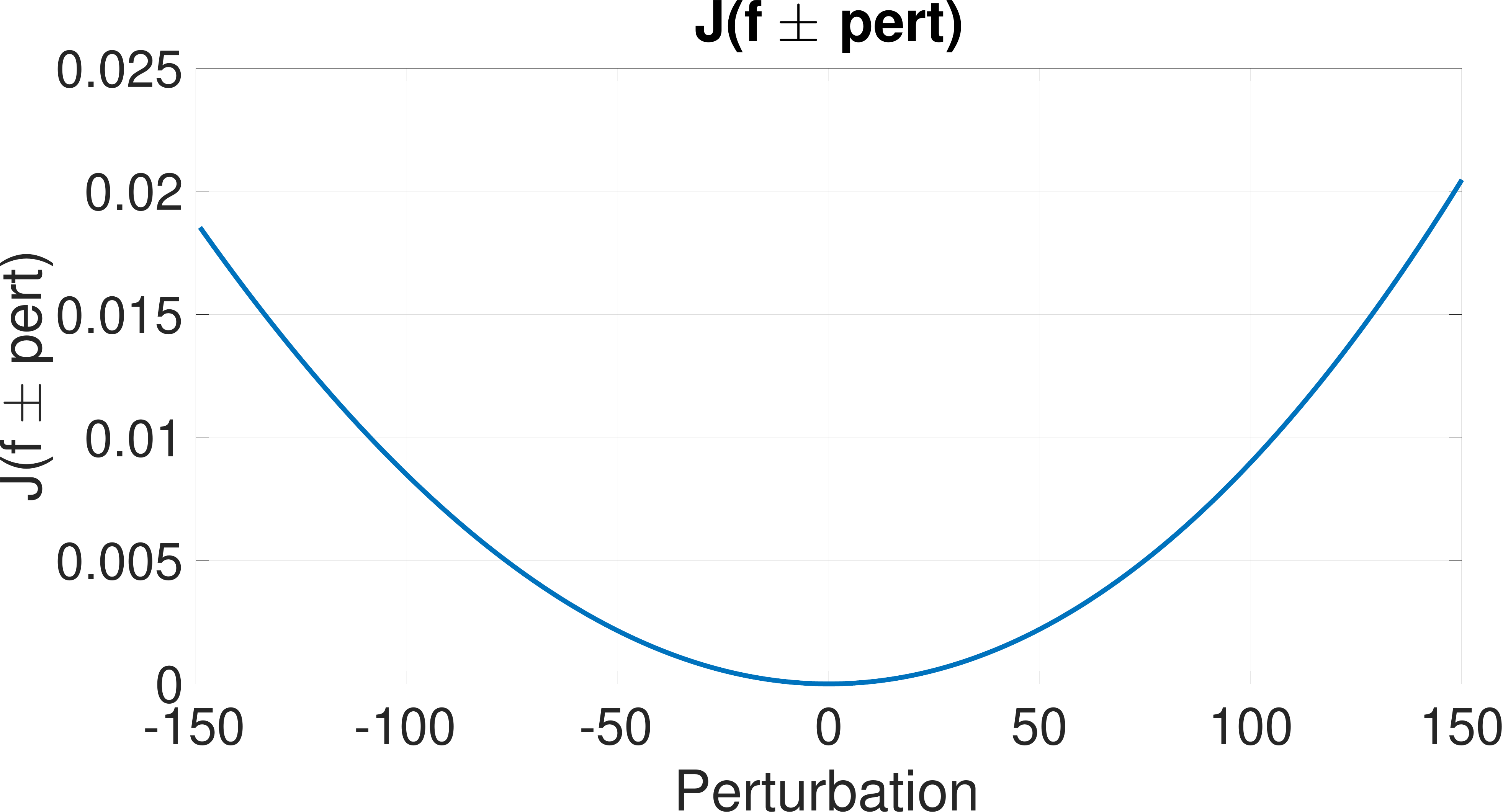}
\caption{Approximation a manufactured solution. Top row: Dynamics of the obtained control $f$ (left) and the associated variables $u$ (center) and variable $v$ (right). Bottom row: Evolution of the cost functional $J(f)$ (left), evolution of $\|\nabla J(f)\|$ (right) and the influence of perturbing the obtained control.}\label{fig:Manufactured_Results}
\end{center}
\end{figure}

\subsection{Bilinear Distributed Control}\label{subsec:simbilinear} 

In this section we study the effect of considering bilinear controls $f(x,t)$. In particular we study how the system behaves when we consider different control and observation domains. In all the cases the initial conditions are
$$
u_0(\x)
\,=\,
1+\cos(\pi x)
\quad\mbox{ and }\quad
v_0(\x)
\,=\,
3+\cos(\pi x)\,.
$$
Moreover, we take as   target state  the constant value $u_d=\oint_\Omega u_0dx=1$ and we consider $f(x,t)=0$ as the initial control in the minimization process. In order to illustrate that the bilinear control has a strong effect in the behavior of the system, we present in Figure~\ref{fig:Bilinearcase0} the dynamics of variables $u$ and $v$ when we compute the solution without imposing any control acting on the system.

\begin{figure}[H]
\begin{center}
\includegraphics[width=0.49\textwidth]{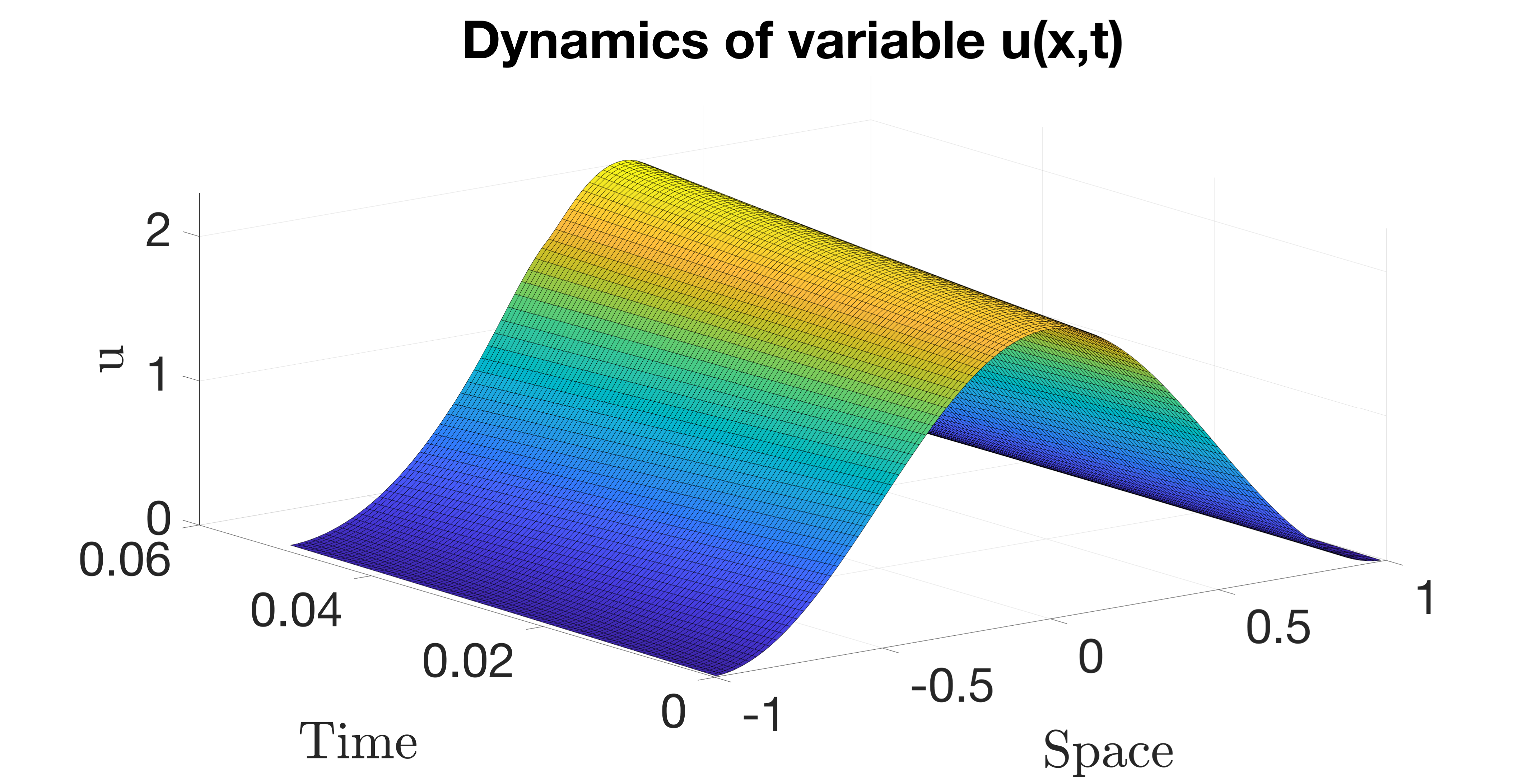}
\includegraphics[width=0.49\textwidth]{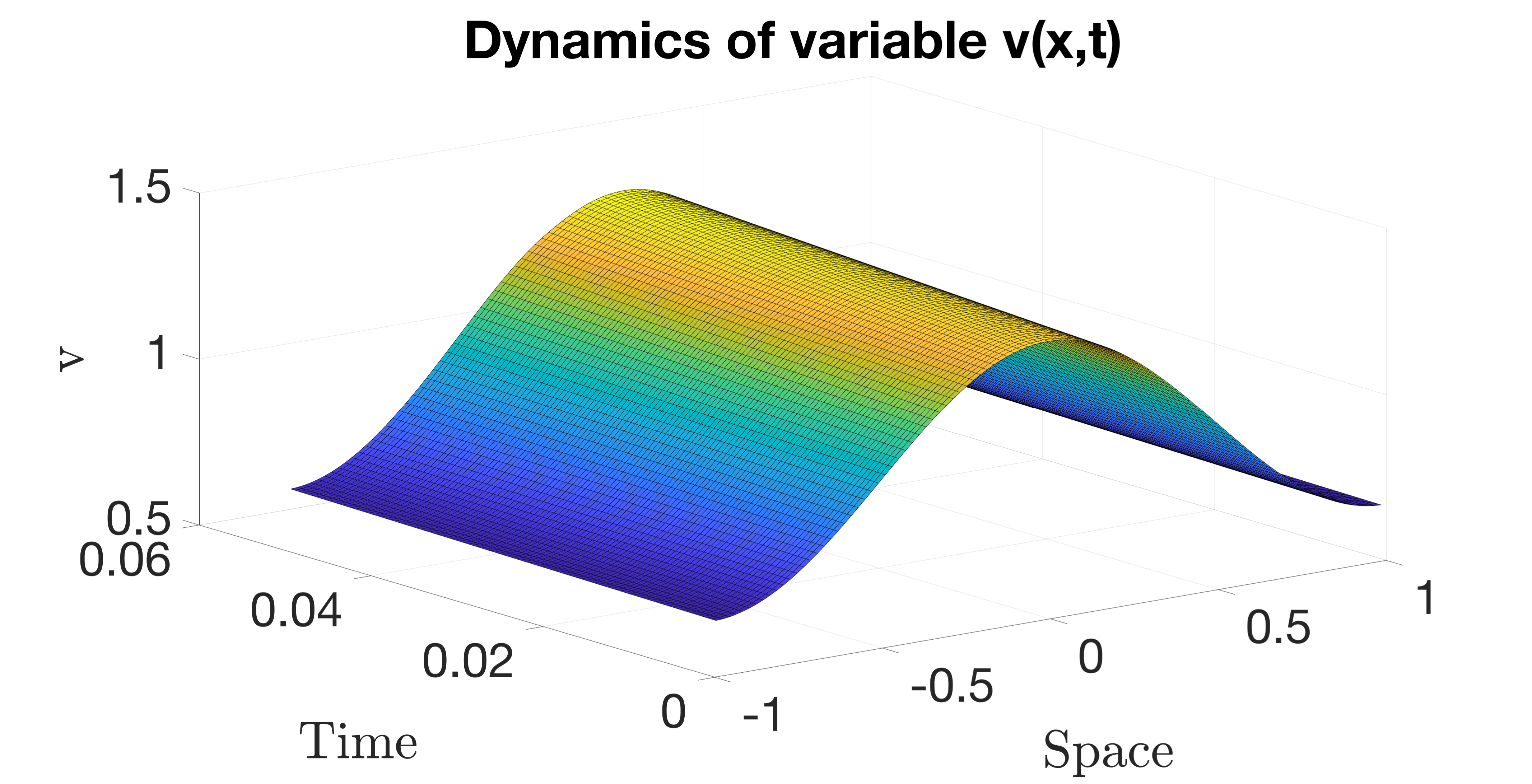}
\caption{Dynamics of variables $u$ (left) and $v$ (right) when no control is acting on the system.}\label{fig:Bilinearcase0}
\end{center}
\end{figure}

\subsubsection{Case 1. $\Omega_c=\Omega_o=\Omega$}\label{subsec:BDCcase1}

In this case, we consider both subdomains, of control and observation, to be equal to the whole spatial domain $\Omega=(-1,1)$. The results are presented in Figure~\ref{fig:case1}, where we can observe how the control acts mainly at the beginning of the time,  producing a symmetric behavior of the variable $v$, with big influence close to the endpoints of the spatial interval, which leads to some spikes close to the endpoints for variable $u$ (which are not numerical instabilities, in fact we have run the same simulations with lower discretization parameters, both in time and space, and we obtain the same dynamics). As a result, the obtained control is able to make variable $u$ really close to $u_d$  from the time $t=0.025$ to the final time. Moreover we see how the functional $J(f)$ is always decreasing with respect to the Adam's iterations and how the norm of the gradient, although oscillating, has a decreasing trend, achieving the desired tolerance before reaching the maximum number of iterations. Finally, we also present a perturbation of the obtained control to evidence that it corresponds with a local minimum of the functional.

\begin{figure}[H]
\begin{center}
\includegraphics[width=0.328\textwidth]{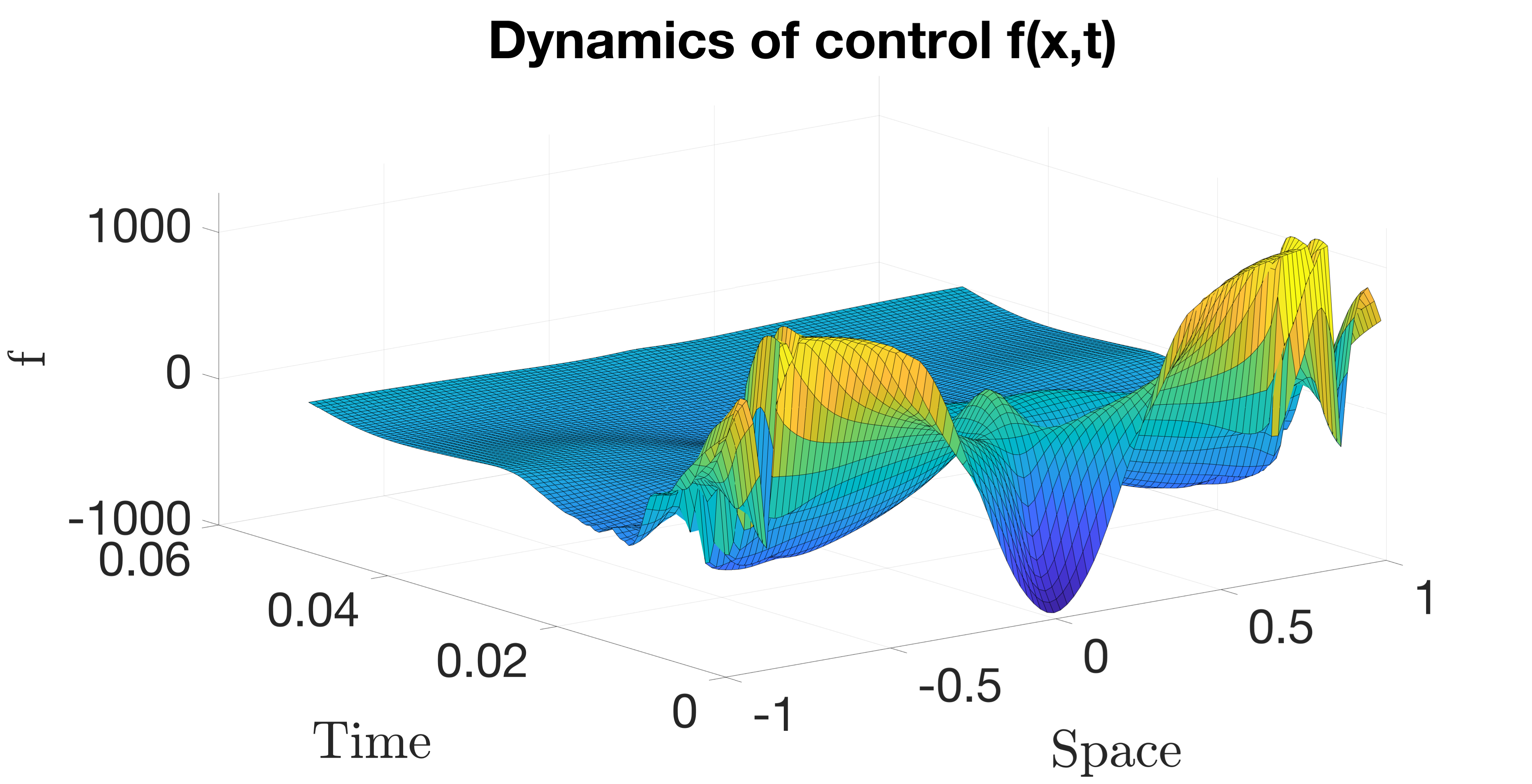}
\includegraphics[width=0.328\textwidth]{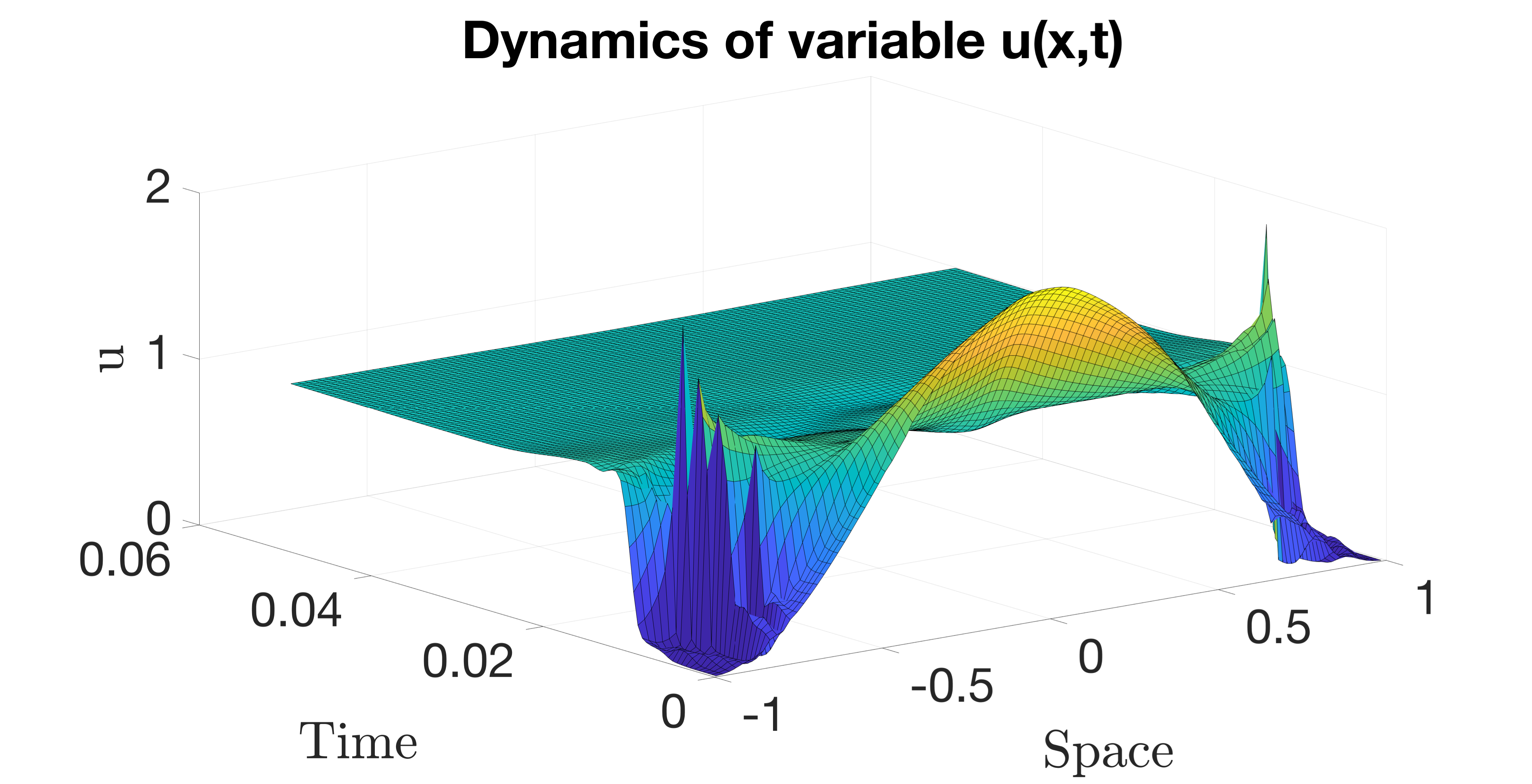}
\includegraphics[width=0.328\textwidth]{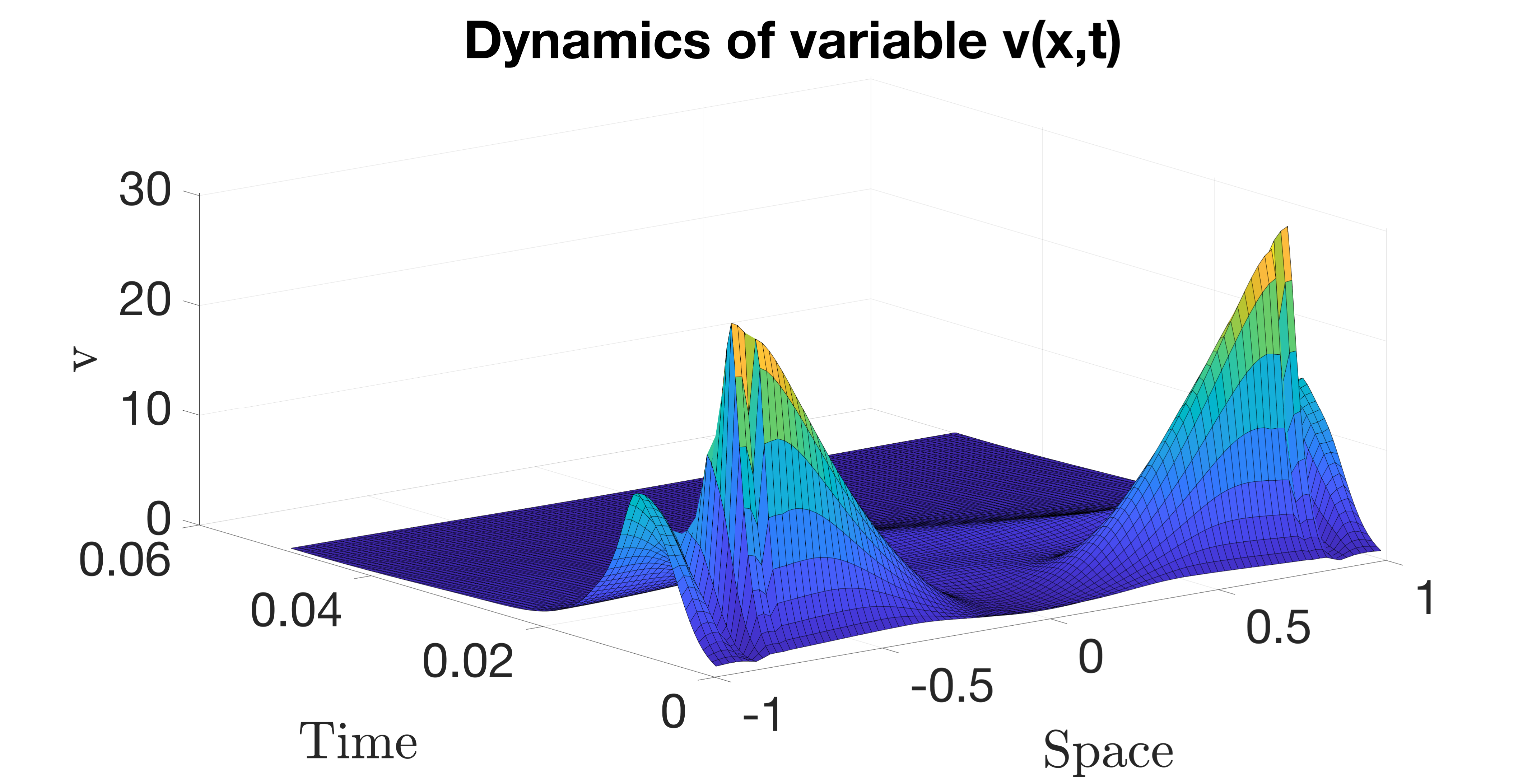}
\\
\includegraphics[width=0.328\textwidth]{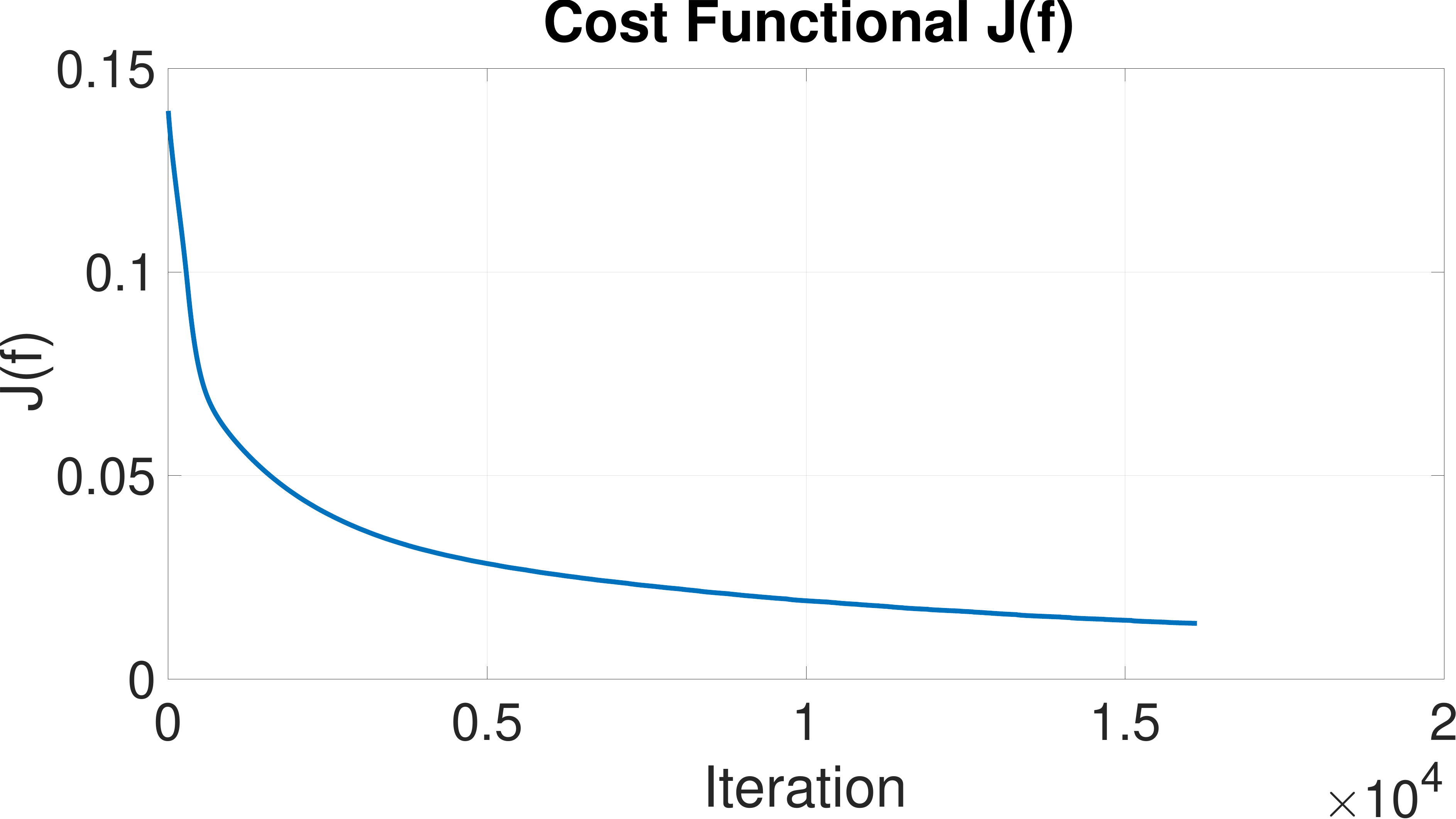}
\includegraphics[width=0.328\textwidth]{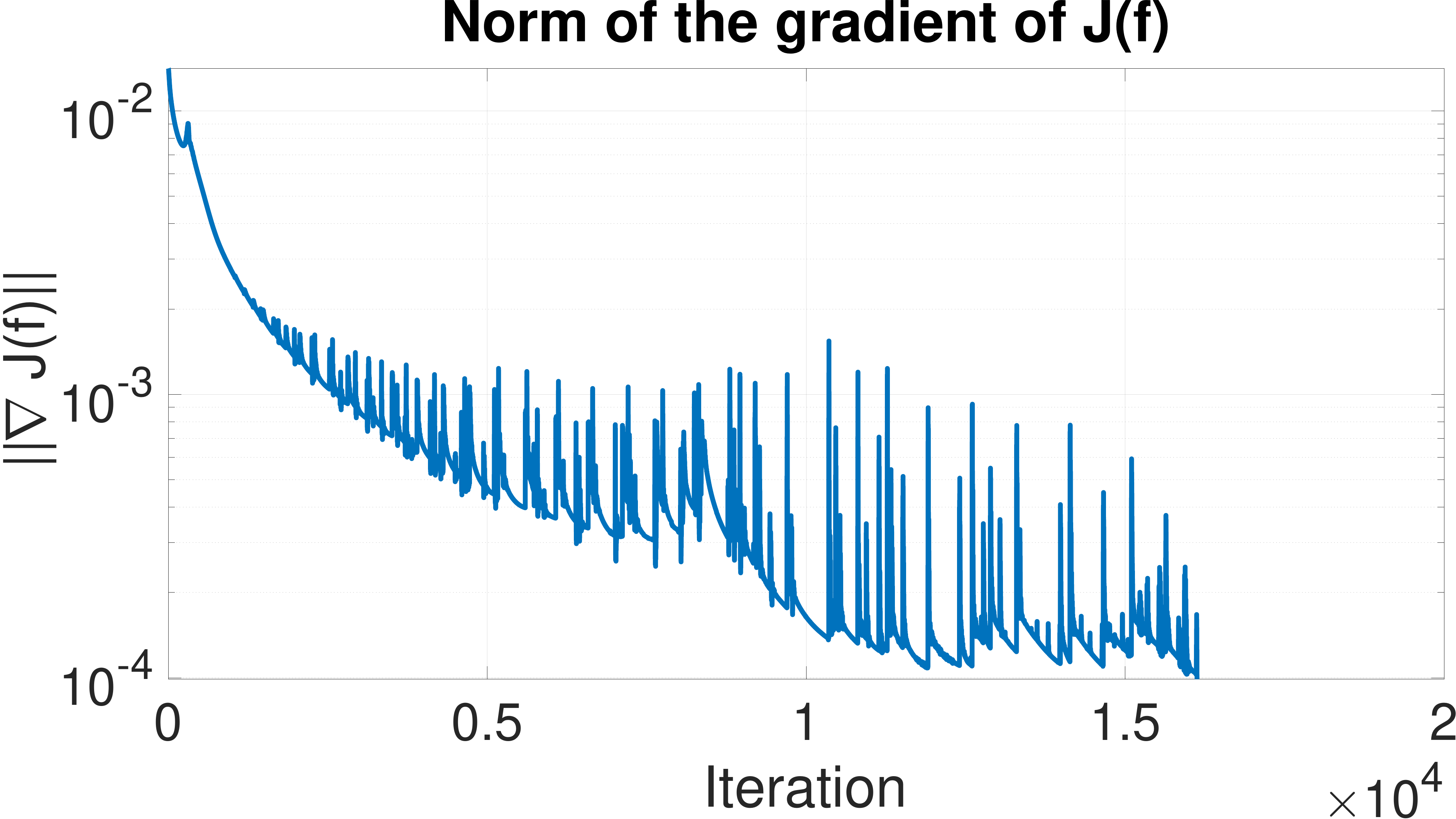}
\includegraphics[width=0.328\textwidth]{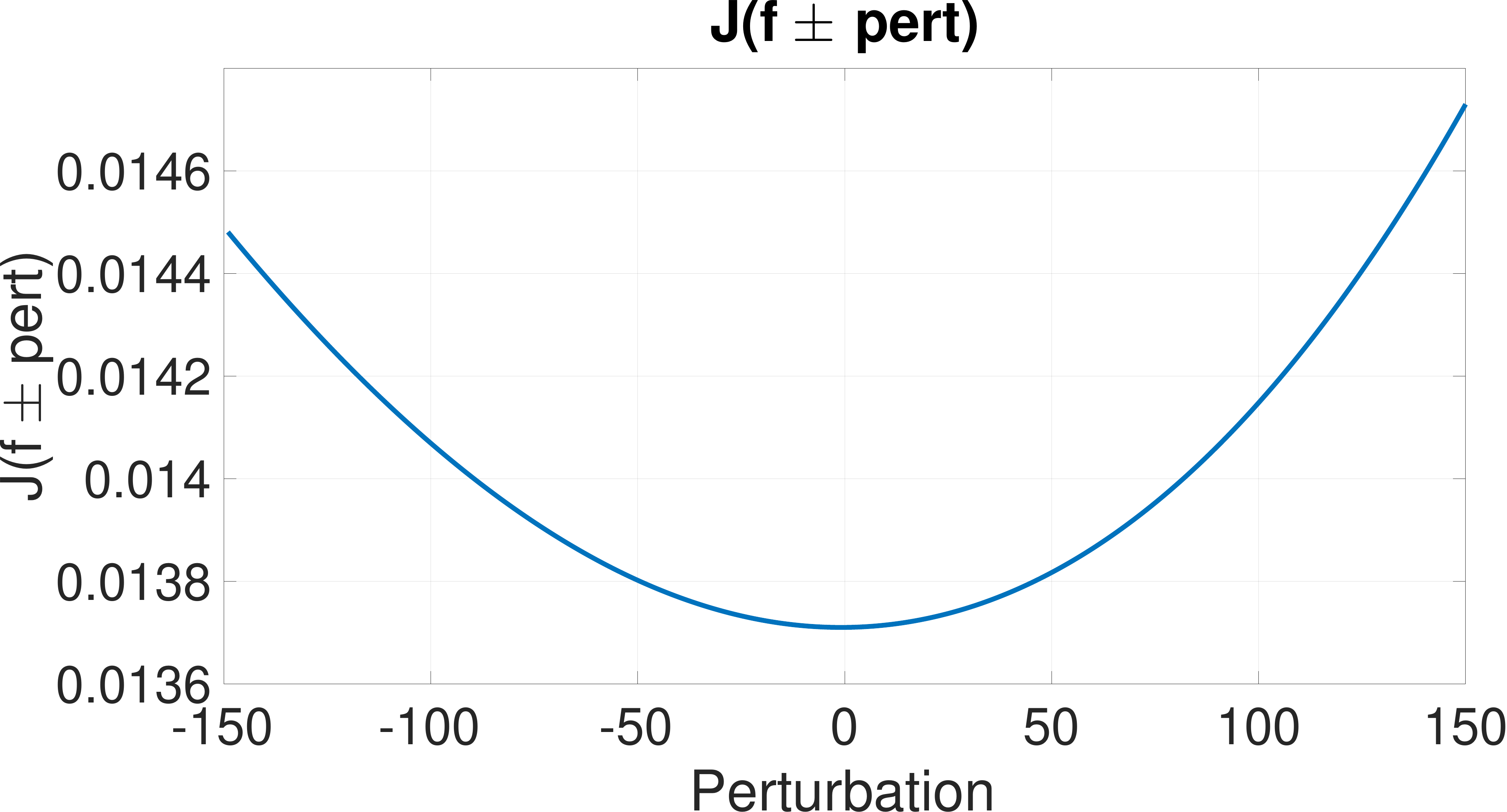}
\caption{Results for $\Omega_c=[-1,1]$ and $\Omega_o=[-1,1]$. Top row: Dynamics of control $f$ (left) and the associated variables $u$ (center) and variable $v$ (right). Bottom row: Evolution of the cost functional $J(f)$ (left), evolution of $\|\nabla J(f)\|$ (center) and the influence of perturbing the obtained control (right).}\label{fig:case1}
\end{center}
\end{figure}

\subsubsection{Case 2. $\Omega_c=[-0.5,0.5]$ and $\Omega_o=[-1,1]$}

In this case, the observation domain is the whole spatial domain but we consider the control domain as the central part of the spatial domain. In Figure~\ref{fig:case2} we observe how the control acts again mainly at the beginning, producing a symmetric behavior of the variable $v$, with big influence close to the endpoints of the control interval, which leads to an increase of $u$ close to the endpoints of $\Omega_c$. In this case, the obtained control is not able to make variable $u$ really close to $u_d$ in the whole observation domain, but it is able to make it very close in $\Omega_c$ in almost the whole time interval (in fact, at time $t=0.05$, variable $u$ is really close to $u_d$ in a domain slightly larger than $\Omega_d$). Additionally, although functional $J(f)$ is always decreasing with respect to the iterations, the norm of the gradient doesn't have a decreasing trend, reaching the maximum number of iterations before achieving the desired tolerance. Also, by perturbing the obtained control we deduce that it corresponds with a local minimum of the functional $J$.

\begin{figure}[H]
\begin{center}
\includegraphics[width=0.328\textwidth]{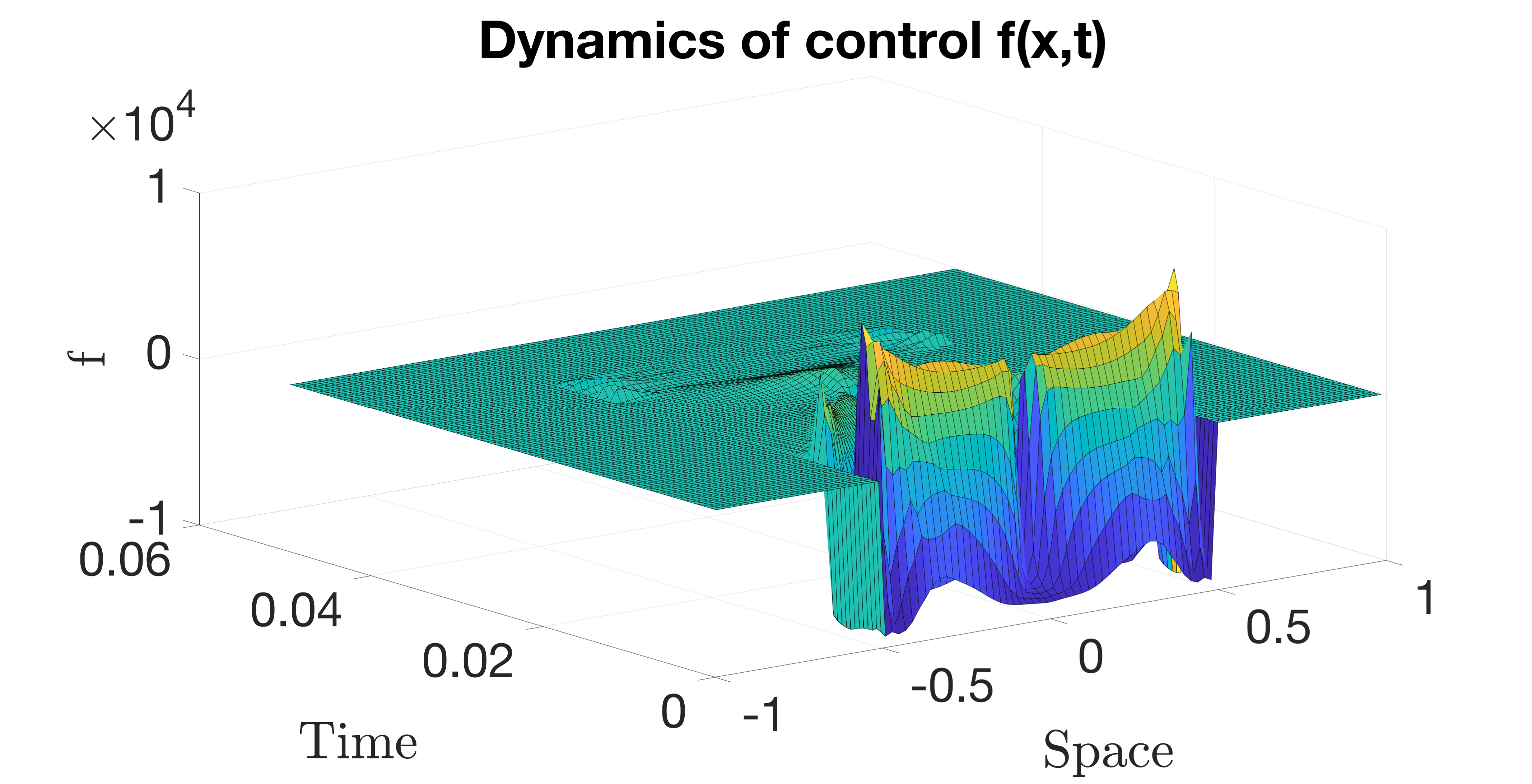}
\includegraphics[width=0.328\textwidth]{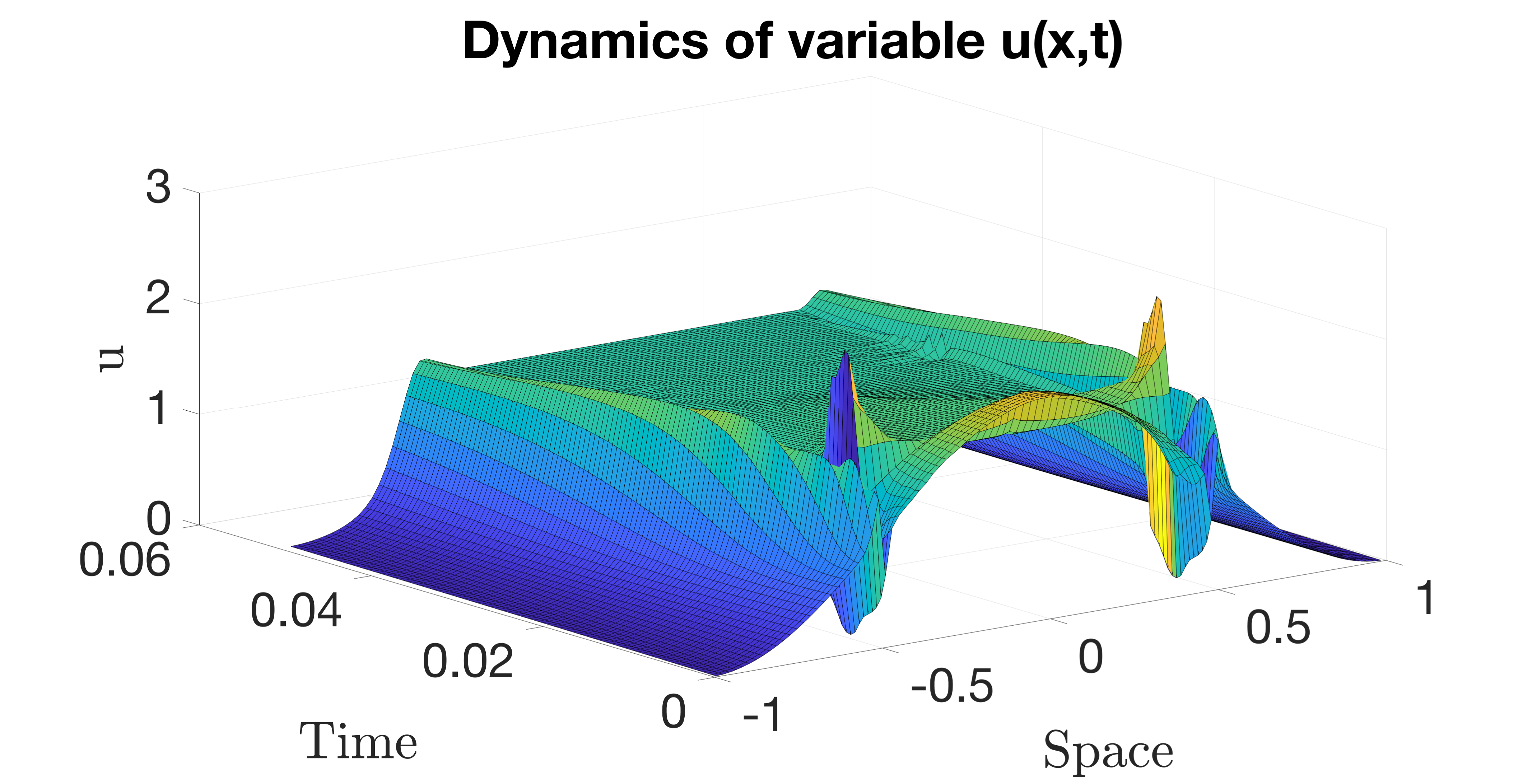}
\includegraphics[width=0.328\textwidth]{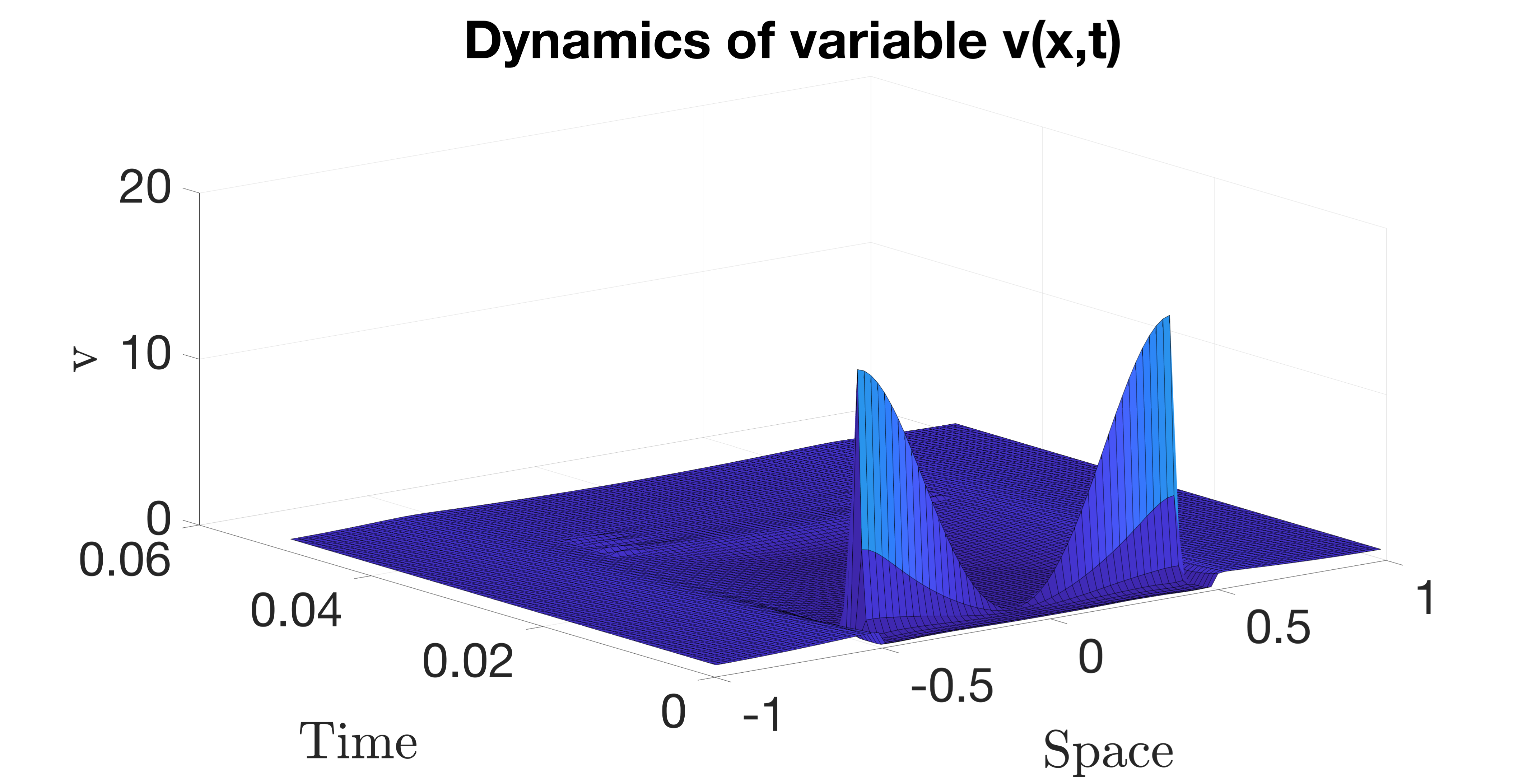}
\\
\includegraphics[width=0.328\textwidth]{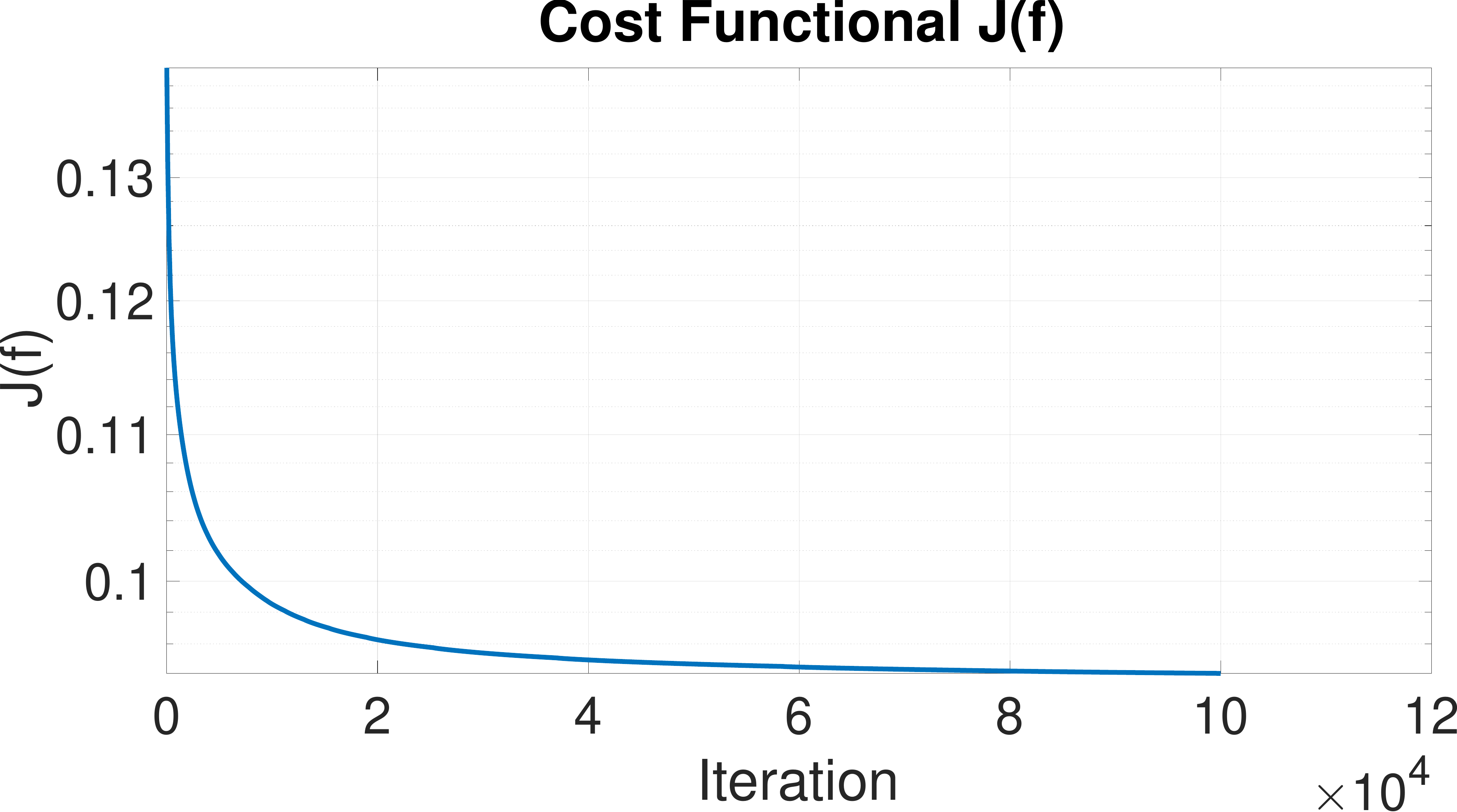}
\includegraphics[width=0.328\textwidth]{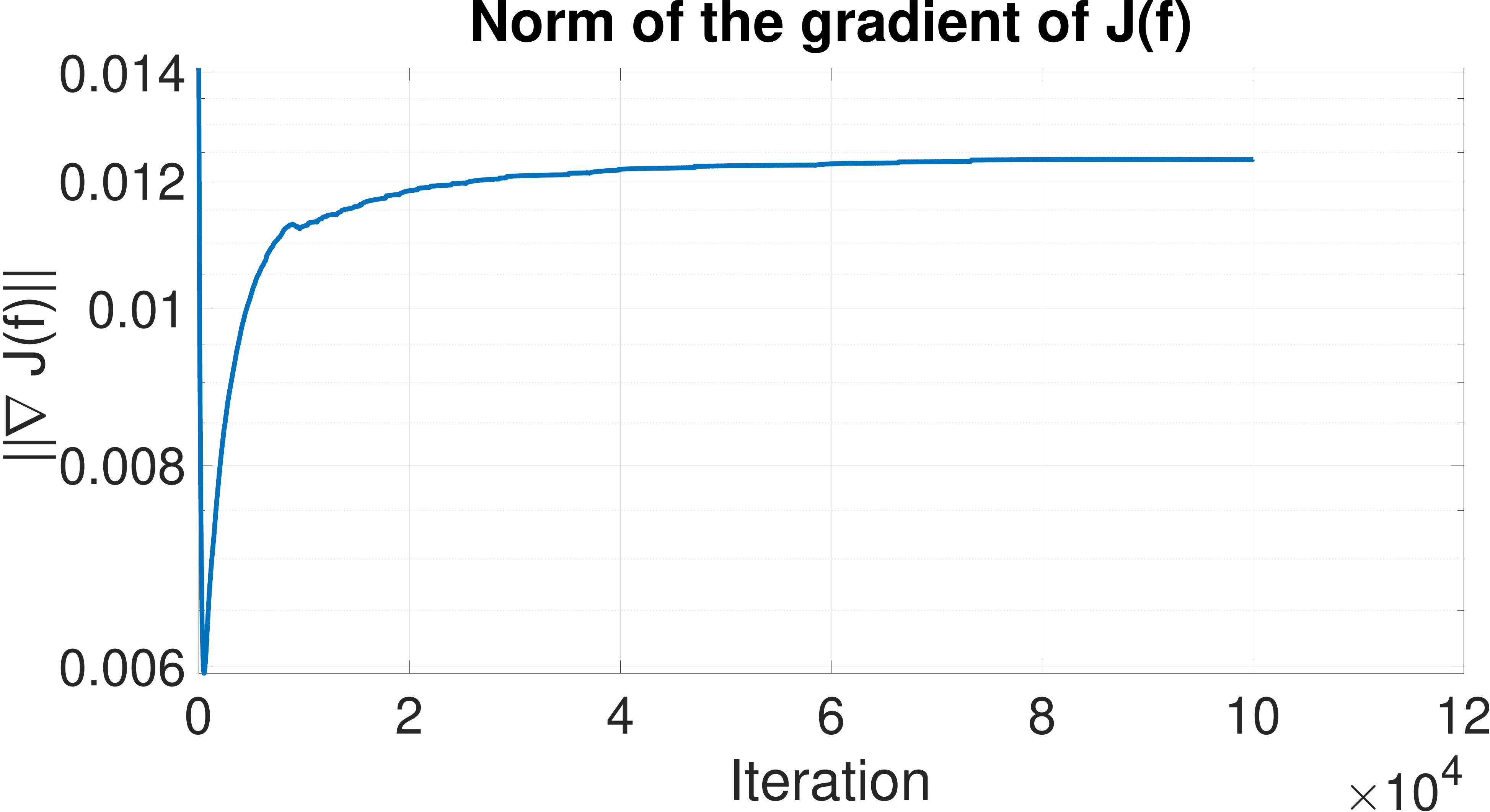}
\includegraphics[width=0.328\textwidth]{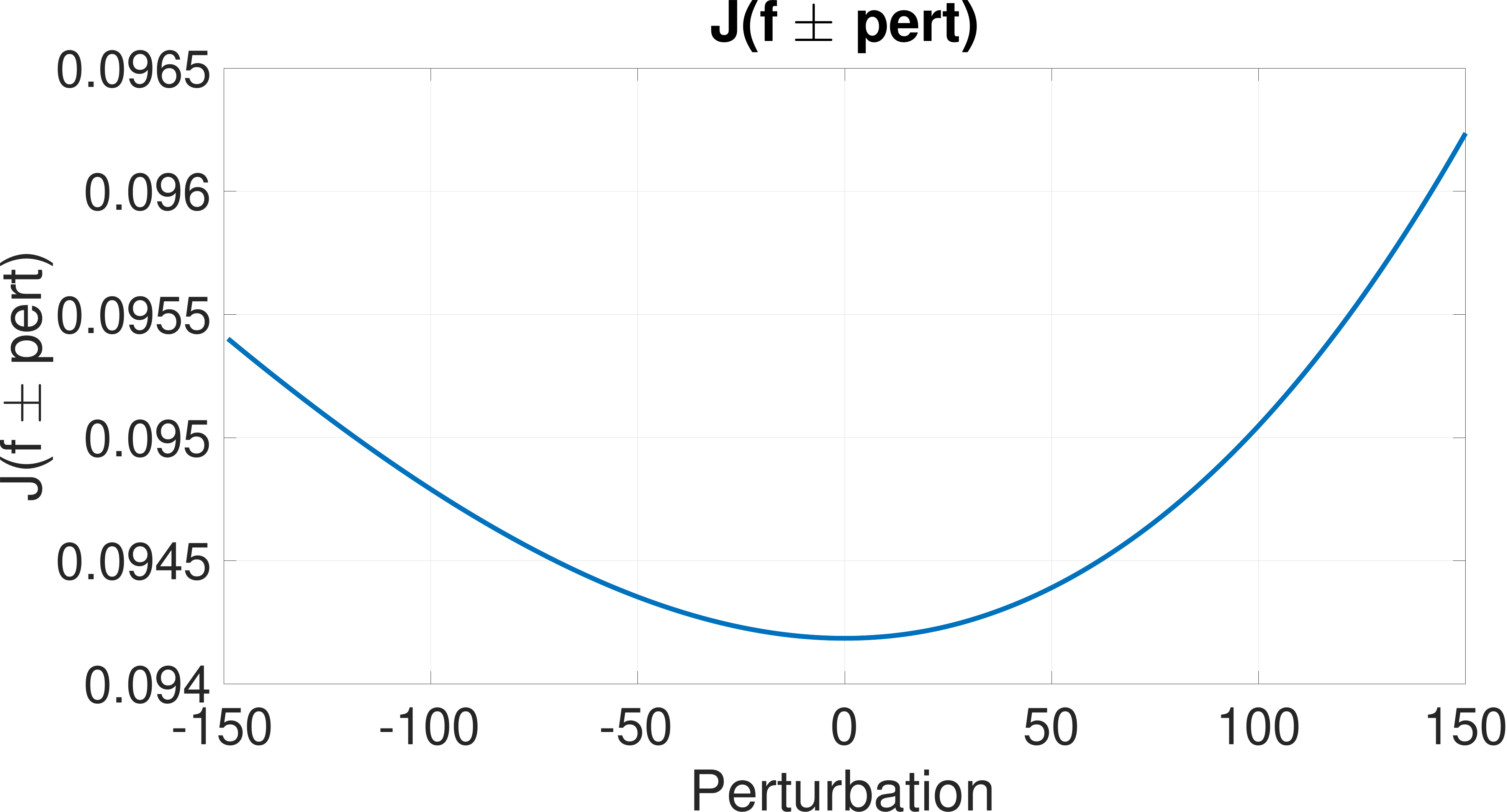}
\caption{Results for $\Omega_c=[-0.5,0.5]$ and $\Omega_o=[-1,1]$. Top row: Dynamics of control $f$ (left) and the associated variables $u$ (center) and variable $v$ (right). Bottom row: Evolution of the cost functional $J(f)$ (left), evolution of $\|\nabla J(f)\|$ (center) and the influence of perturbing the obtained control (right).}\label{fig:case2}
\end{center}
\end{figure}

\subsubsection{Case 3. $\Omega_c=[-1,1]$ and $\Omega_o=[-0.5,0.5]$}

We focus now on the case where the control domain is the whole spatial domain but the observation domain is the central part of the spatial domain. We present the results in Figure~\ref{fig:case3}, where we observe how the control activity is concentrated at the beginning of the time and in the spatial region around the observation domain, producing a symmetric behavior of the variable $v$, with big influence close to the endpoints of the observation interval, which leads to $u$  very close to $u_d$ in $\Omega_o$, in almost the whole time interval. In fact, this is possible at the 'expense' of significantly increasing the value of $u$ on the regions next but outside of the observation domain. Functional $J(f)$ and the norm of its gradient are decreasing  with respect to the iterations (to be precise, $\|\nabla J (f)\|$ exhibits some small oscillations), reaching the desired tolerance before the maximum number of iterations is satisfied. Moreover, by studying the value of $J$ for perturbations of the obtained  control we can deduce that it corresponds with a local minimum of the functional $J$.

\begin{figure}[H]
\begin{center}
\includegraphics[width=0.328\textwidth]{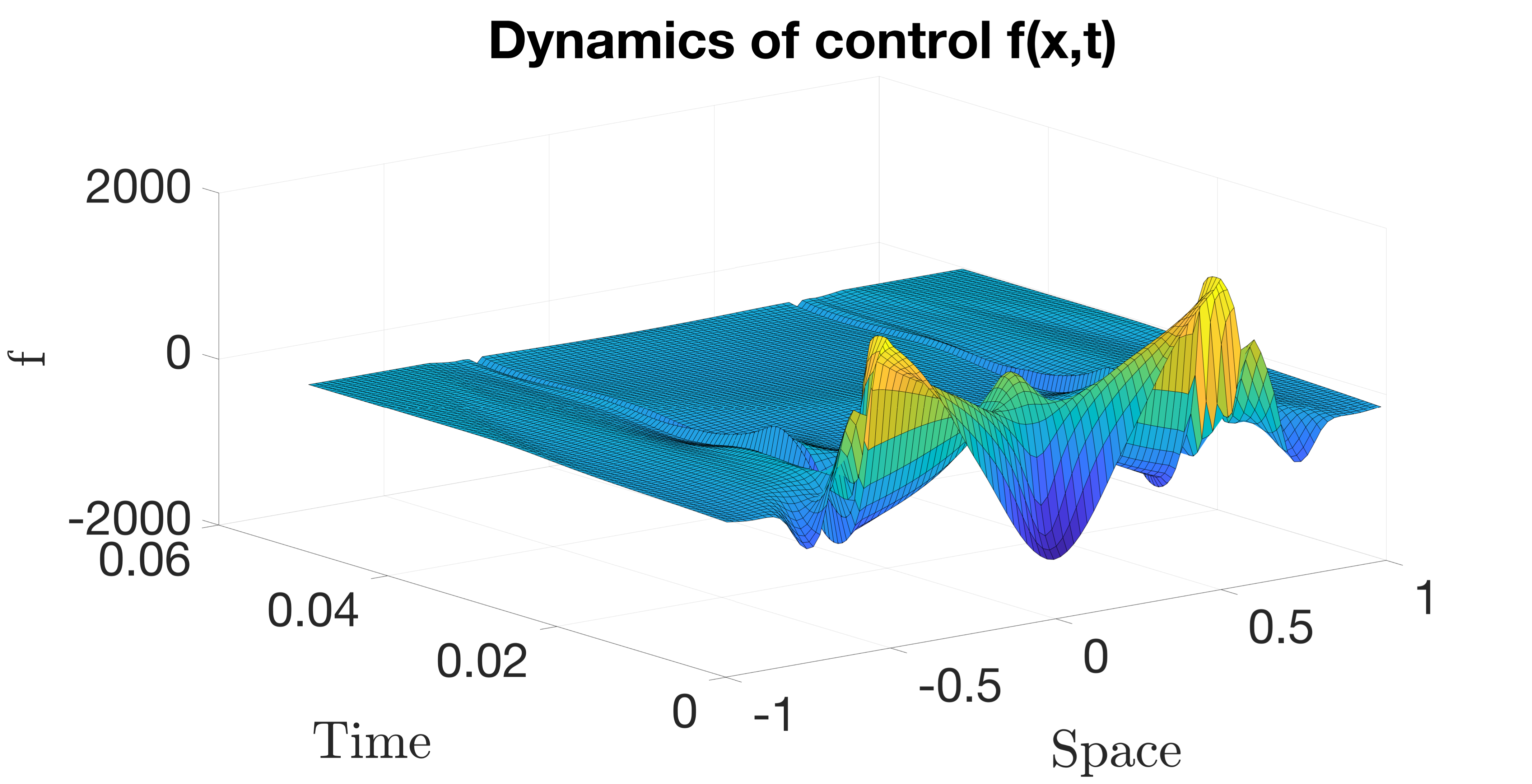}
\includegraphics[width=0.328\textwidth]{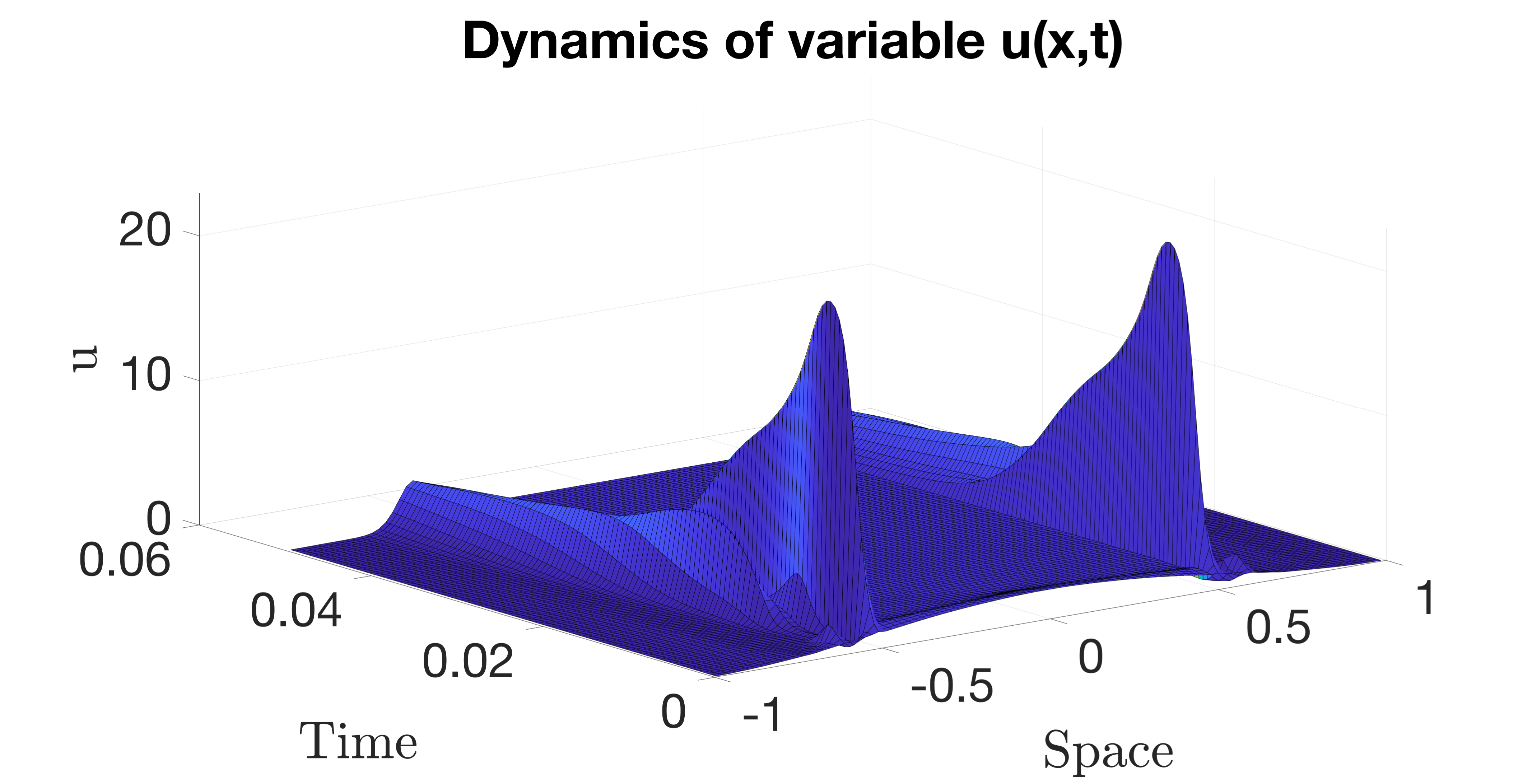}
\includegraphics[width=0.328\textwidth]{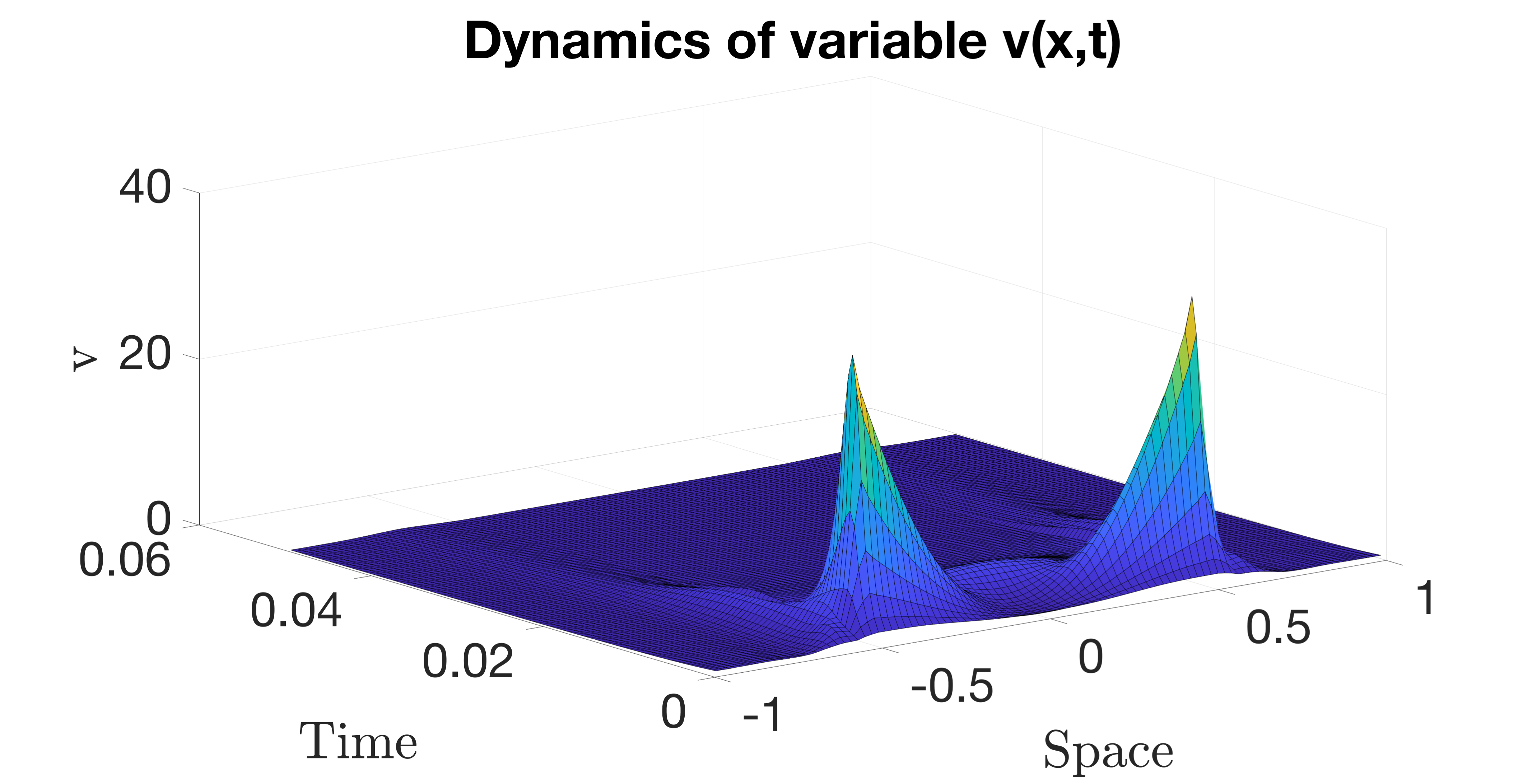}
\\
\includegraphics[width=0.328\textwidth]{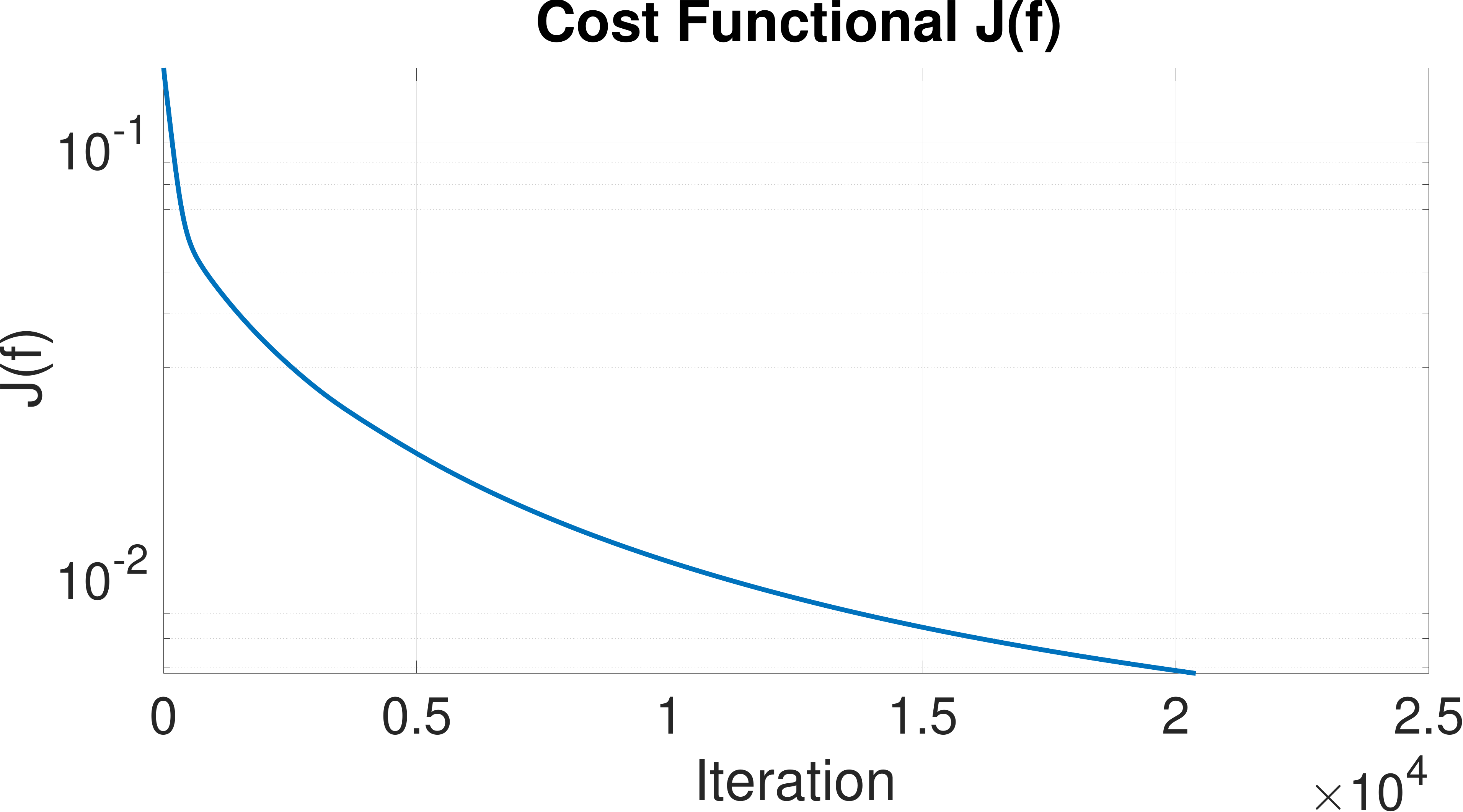}
\includegraphics[width=0.328\textwidth]{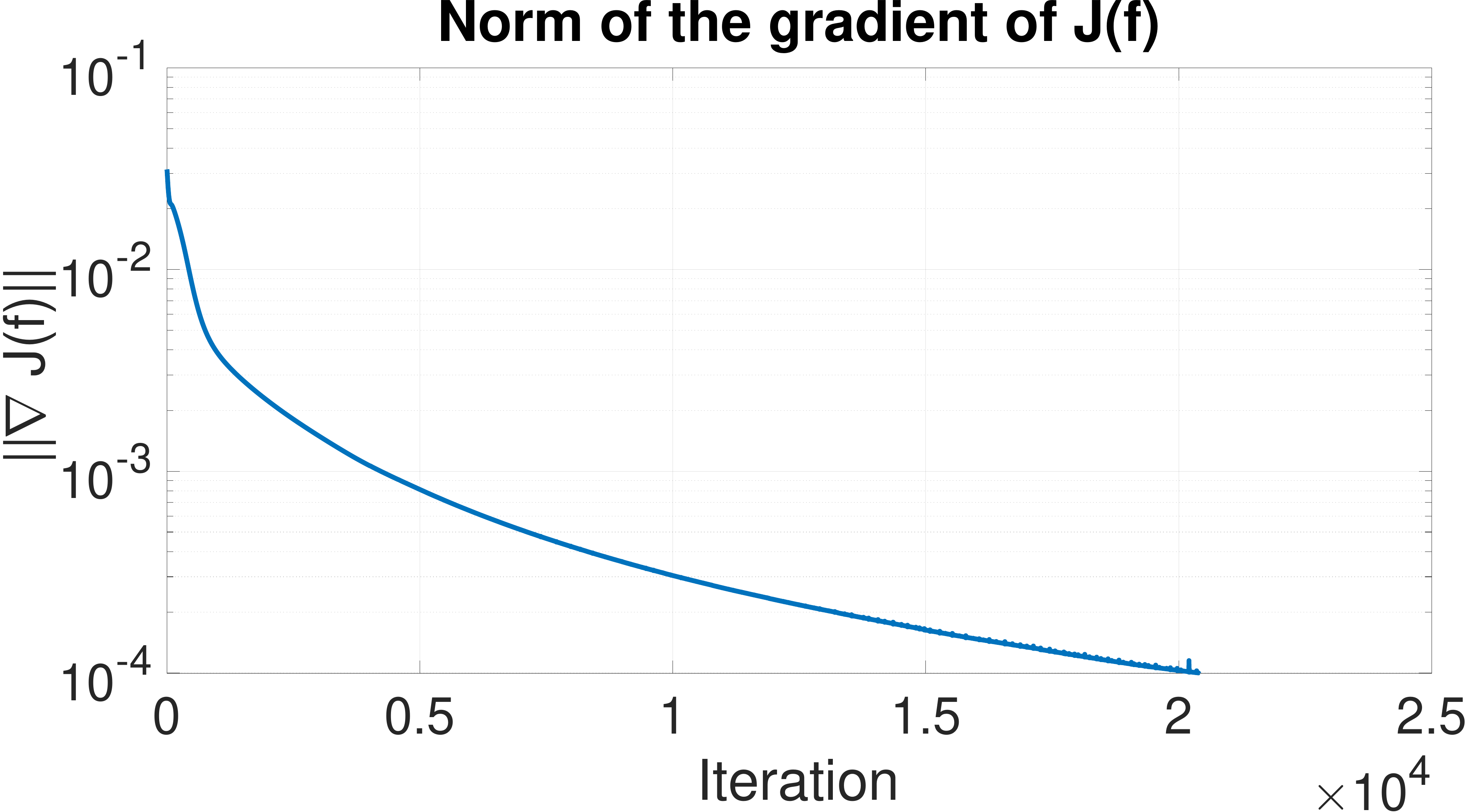}
\includegraphics[width=0.328\textwidth]{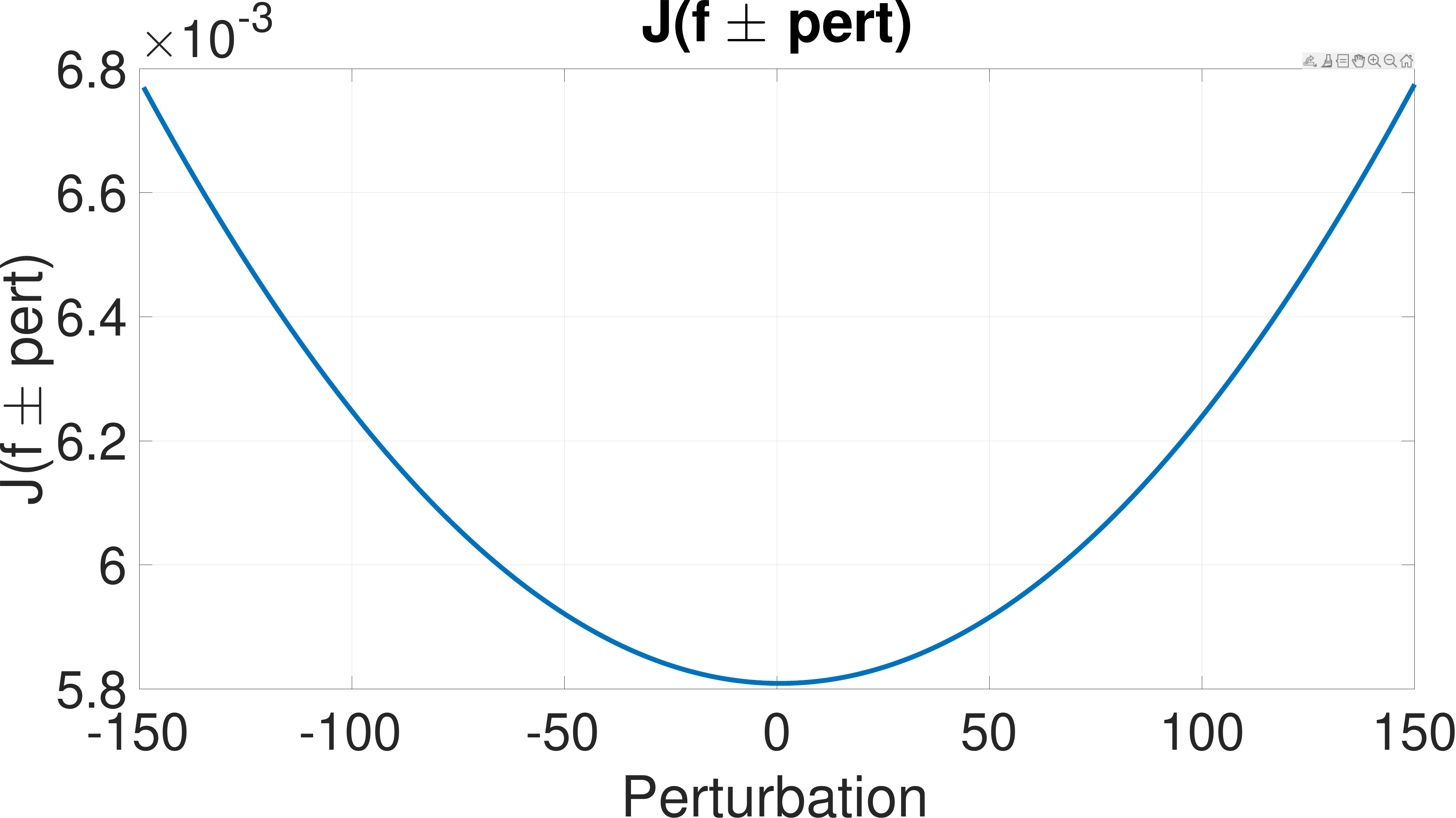}
\caption{Results for $\Omega_c=[-1,1]$ and $\Omega_o=[-0.5,0.5]$. Top row: Dynamics of control $f$ (left) and the associated variables $u$ (center) and variable $v$ (right). Bottom row: Evolution of the cost functional $J(f)$ (left), evolution of $\|\nabla J(f)\|$ (center) and the influence of perturbing the obtained control (right).}\label{fig:case3}
\end{center}
\end{figure}

\subsubsection{Case 4. $\Omega_c=[-1,0.2]$ and $\Omega_o=[-0.2,1]$}

In Figure~\ref{fig:case4} we present the results on the case where the control domain and the observation domain only overlap in the central part of the spatial domain (interval $[-0.2,0.2]$). In this case the obtained control plays a role only at the beginning of the time interval and around the half of $\Omega_c$ that is next to $\Omega_o$, producing a non-symmetric behavior of the variable $v$, with the highest influence on the left endpoint of the observation domain, which leads the variable $u$ to be somehow close to $u_d$ in $\Omega_o$, at the 'expense' of being really away of the desired configuration outside of the observation domain. Moreover, functional $J(f)$ is decreasing  with respect to the iterations and the norm of its gradient has a steep decrease at the beginning of the simulations, then a mild increase and then it decreases again, 
although it doesn't reach the desired tolerance before the maximum number of iterations is satisfied (at the final iteration $\|\nabla J(f)\|=0.0077$). Finally, we observe that the obtained control corresponds with a local minimum of the functional $J$.

\begin{figure}[H]
\begin{center}
\includegraphics[width=0.328\textwidth]{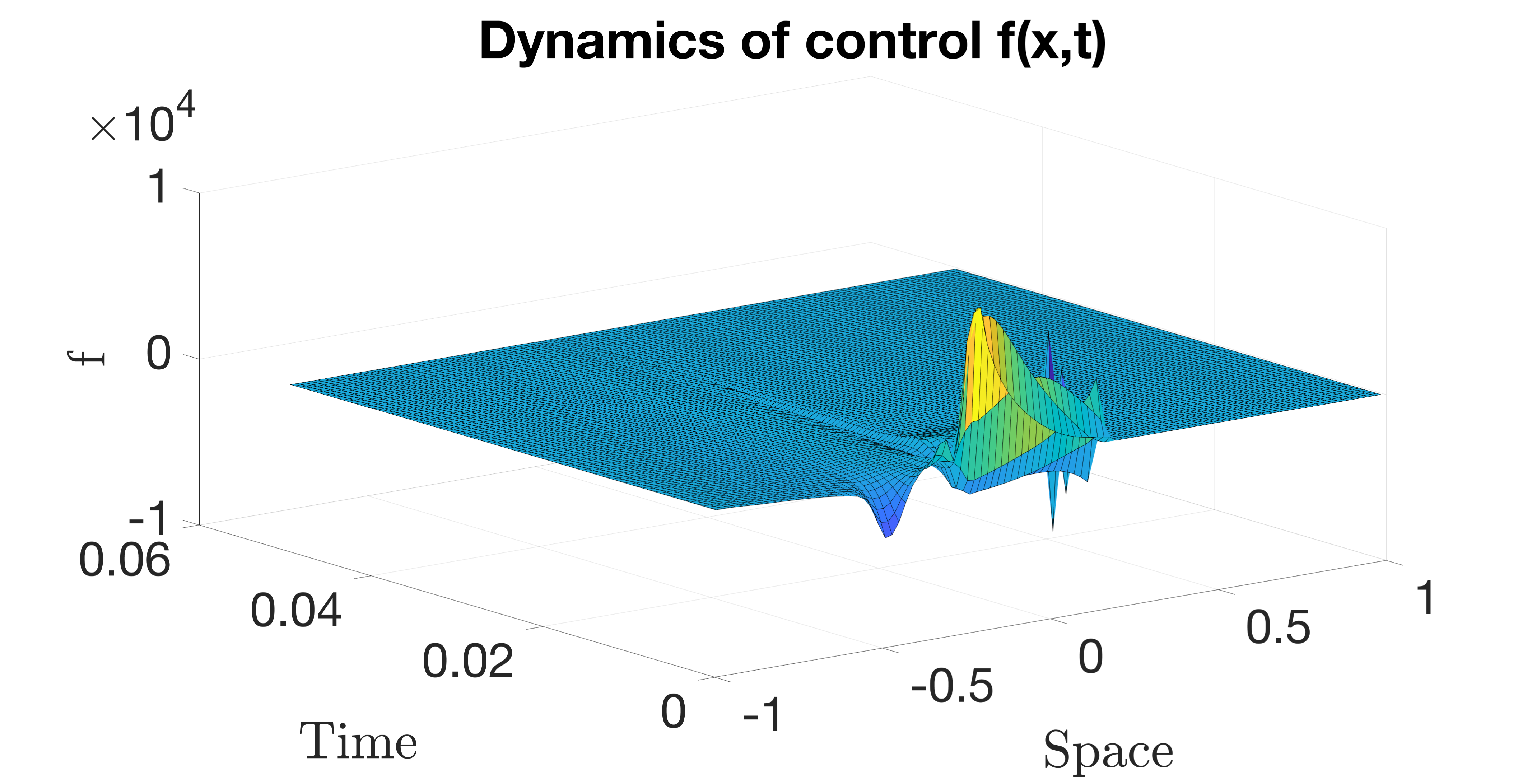}
\includegraphics[width=0.328\textwidth]{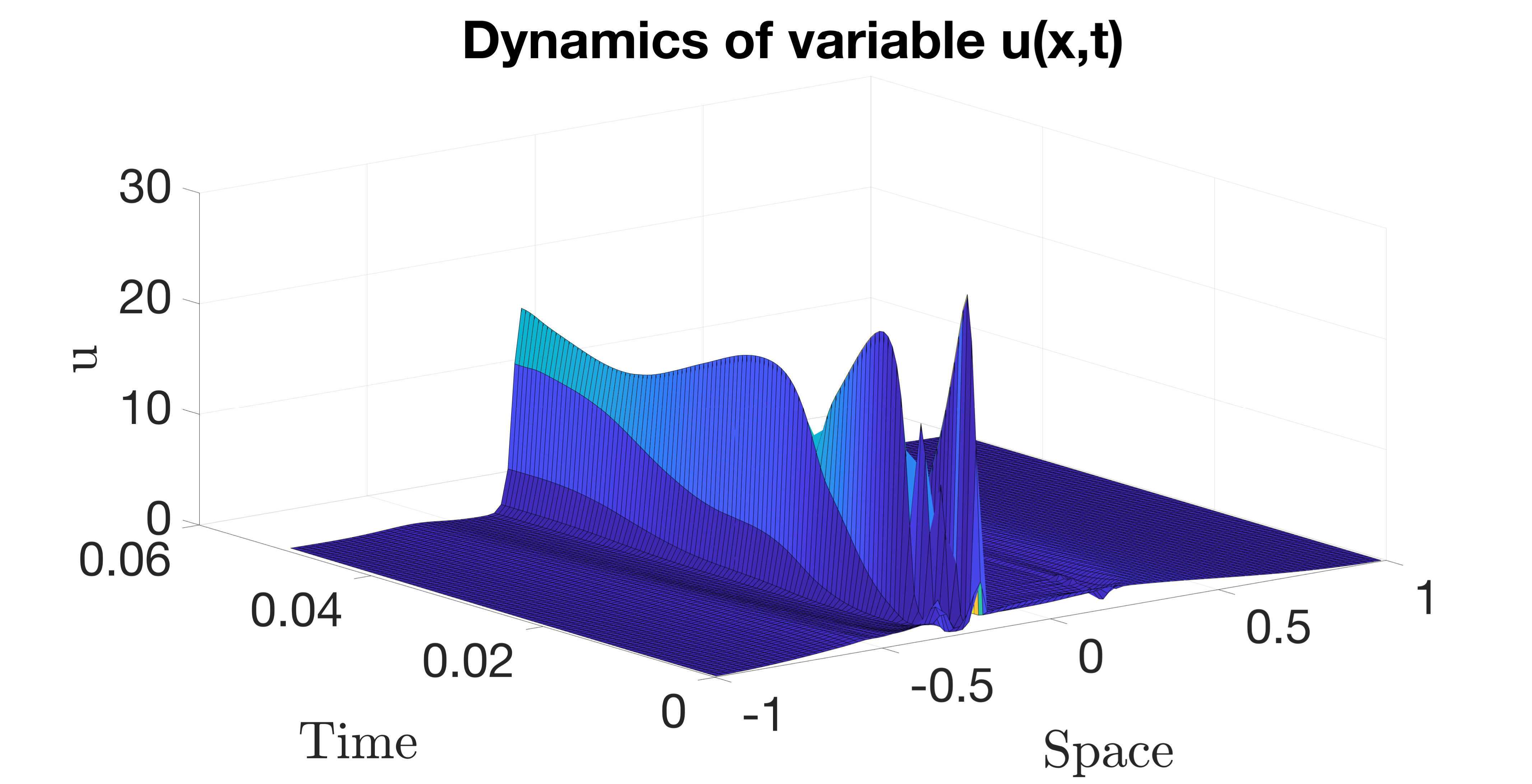}
\includegraphics[width=0.328\textwidth]{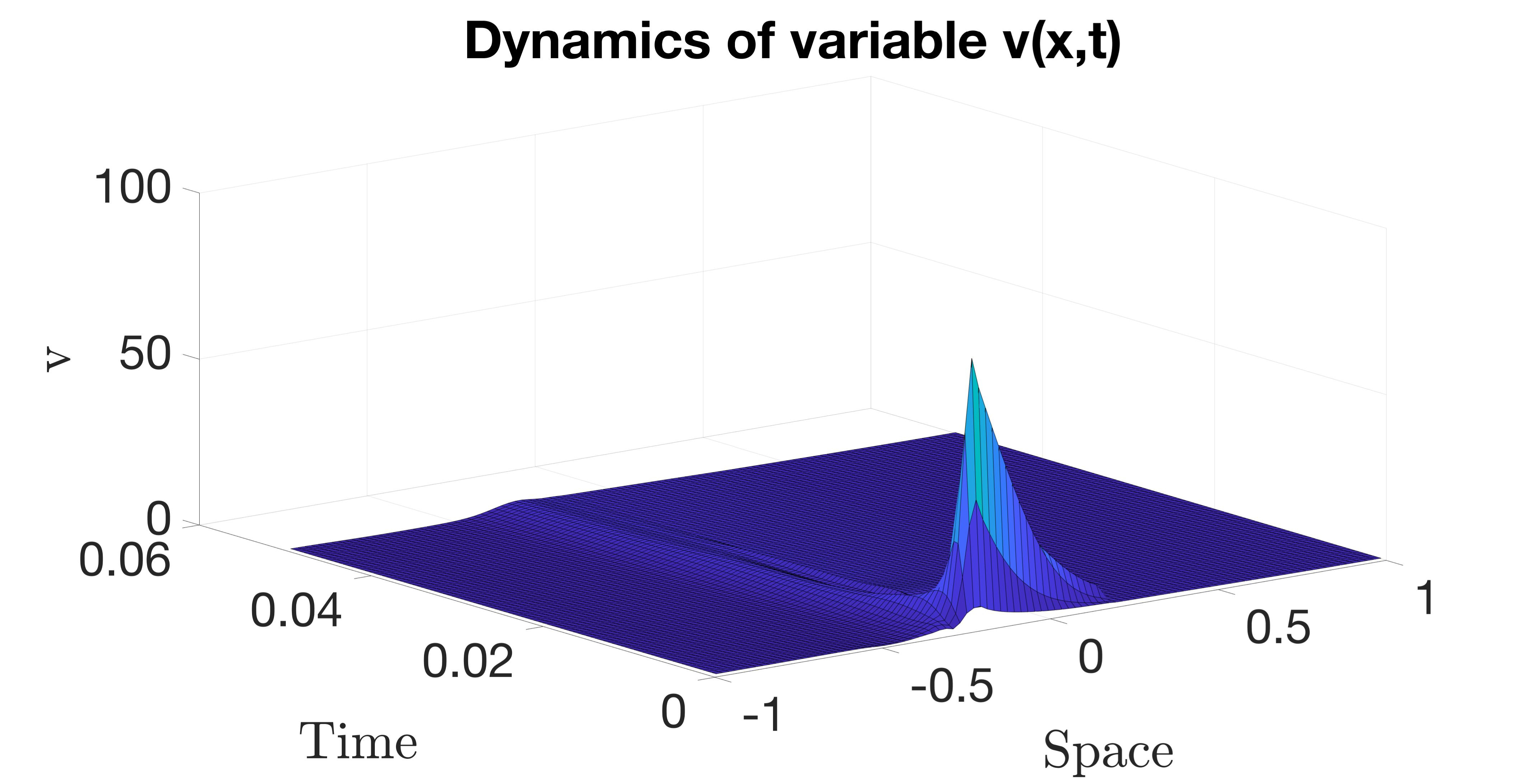}
\\
\includegraphics[width=0.328\textwidth]{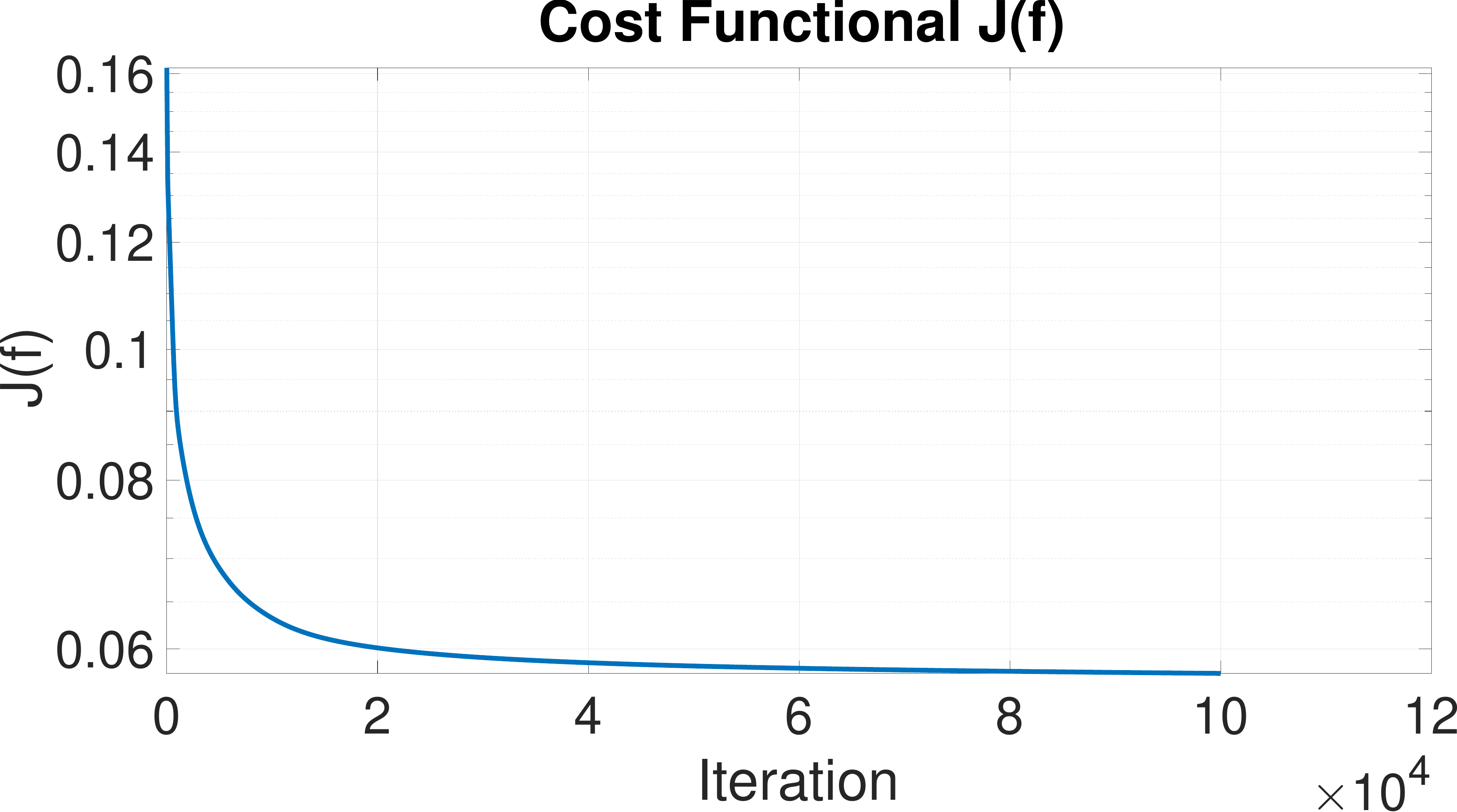}
\includegraphics[width=0.328\textwidth]{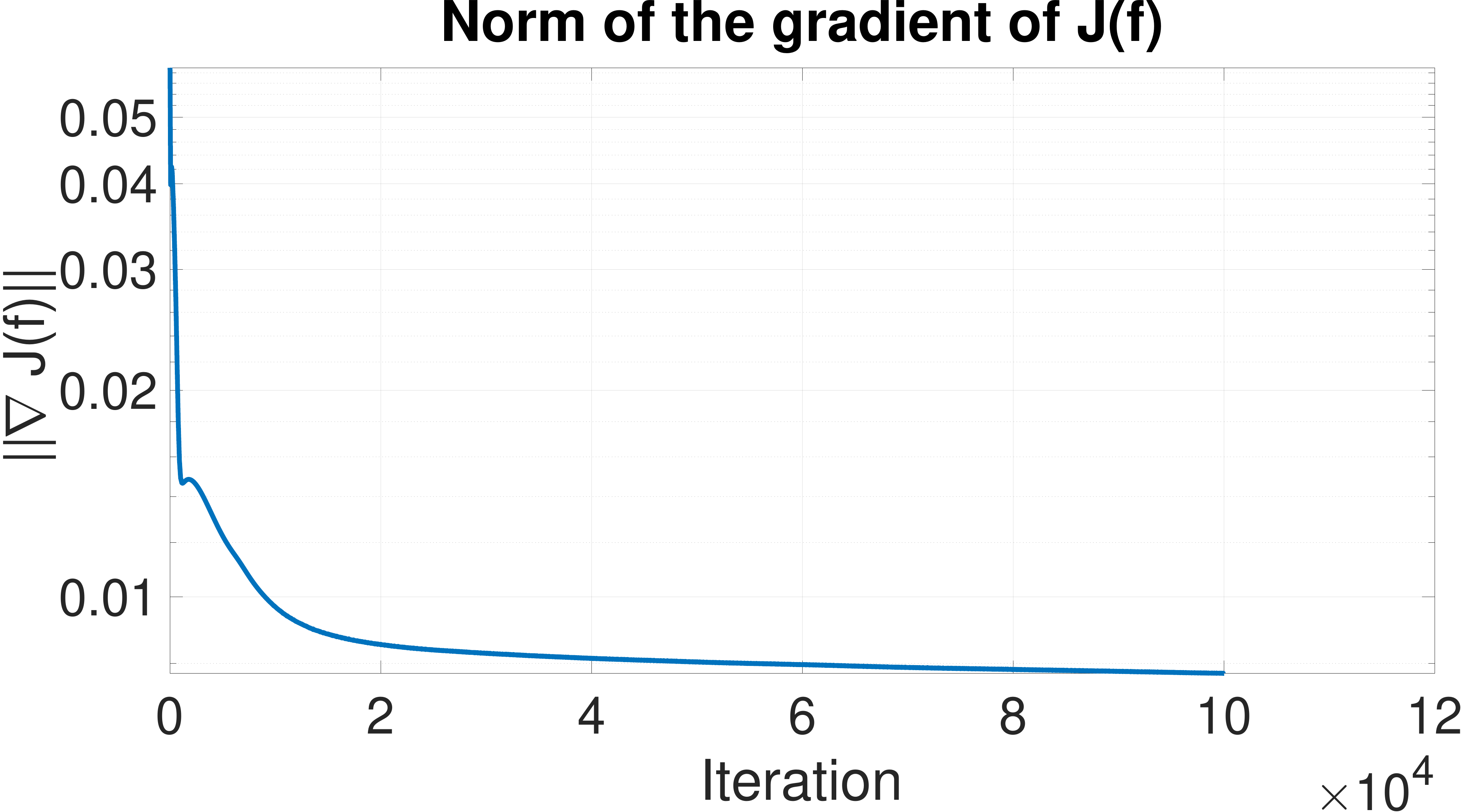}
\includegraphics[width=0.328\textwidth]{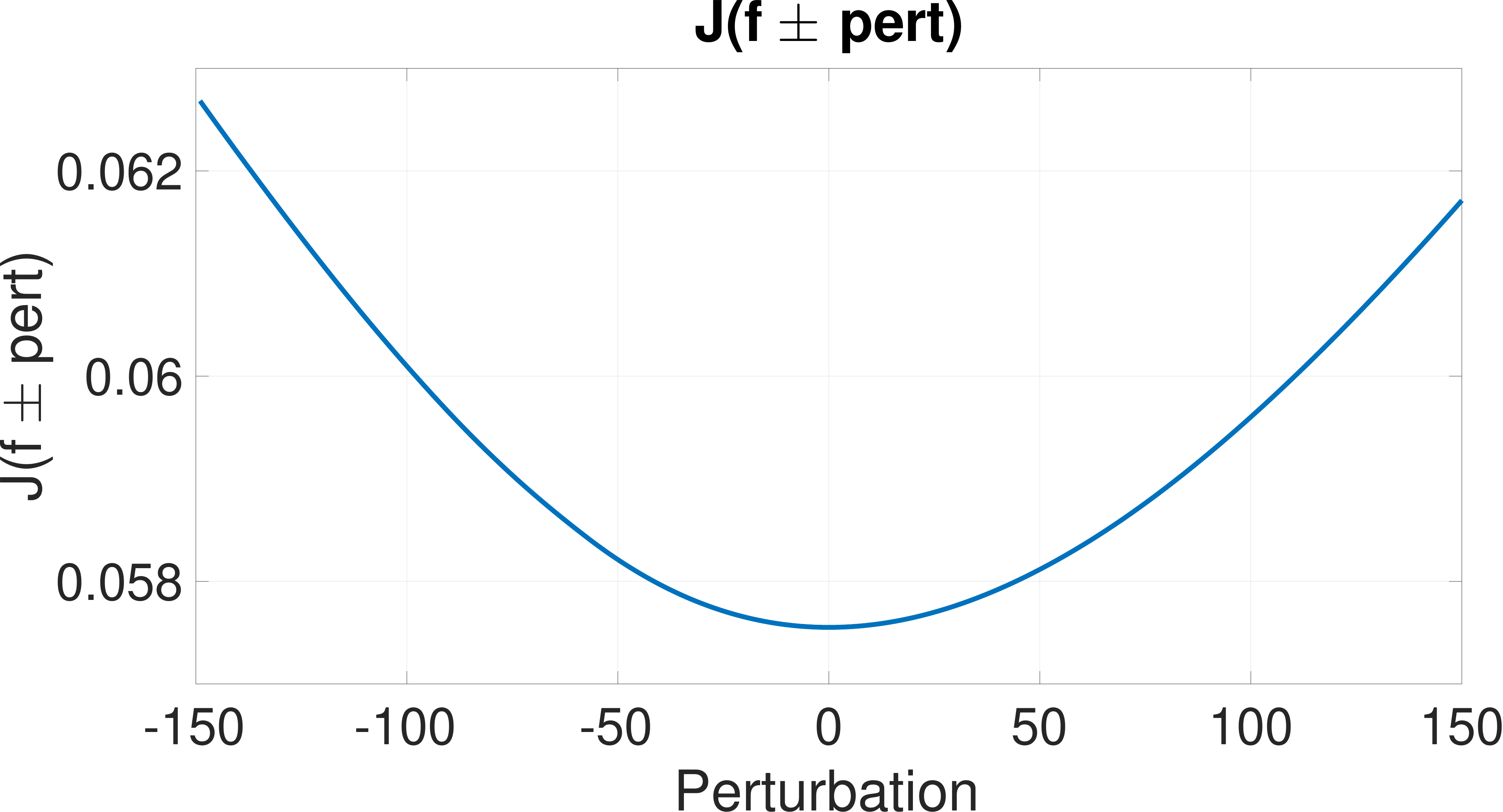}
\caption{Results for $\Omega_c=[-1,0.2]$ and $\Omega_o=[-0.2,1]$. Top row: Dynamics of control $f$ (left) and the associated variables $u$ (center) and variable $v$ (right). Bottom row: Evolution of the cost functional $J(f)$ (left), evolution of $\|\nabla J(f)\|$ (center) and the influence of perturbing the obtained control (right).}\label{fig:case4}
\end{center}
\end{figure}

\subsubsection{Case 5. $\Omega_c=[-1,-0.2]$ and $\Omega_o=[0.2,1]$}

In this case we study the most challenging of the possible situations, where the control domain (located at the left of the spatial interval) and the observation domain (located at the right of the spatial interval) don't overlap 
with each other. The results are presented in Figure~\ref{fig:case5} where we can observe how the control is acting mainly in the region close to the right boundary of the control domain ($x=-0.2$), and as a difference with previous cases, it acts during most of the time interval. This produce a large spike on the configuration of variable $v$ close to $x=-0.2$ (which is later on diminished on time due to the fact that the control is still acting) and this produce $u$ to get close to $u_d$ in $\Omega_o$ but very far away of that configuration outside that domain. Both functional $J(f)$ and the norm of its gradient is always decreasing with respect to the iterations but it doesn't reach the desired tolerance before the maximum number of iterations is satisfied (at the final iteration $\|\nabla J(f)\|=0.0367$). Interestingly this is the only case where the obtained control is not a local minimum for functional $J$ when we perturb the control, meaning that if we keep running the simulations we should be able to obtain a better control. Unfortunately, the decrease on functional $J$ is so slow, that even increasing the number of iterations does not produce a better result, to illustrate this fact we present in Figure~\ref{fig:case5b} the results if the maximum number of iterations is set to $200000$, where now we obtain $\|\nabla J(f)\|=0.035$.

\begin{figure}[H]
\begin{center}
\includegraphics[width=0.328\textwidth]{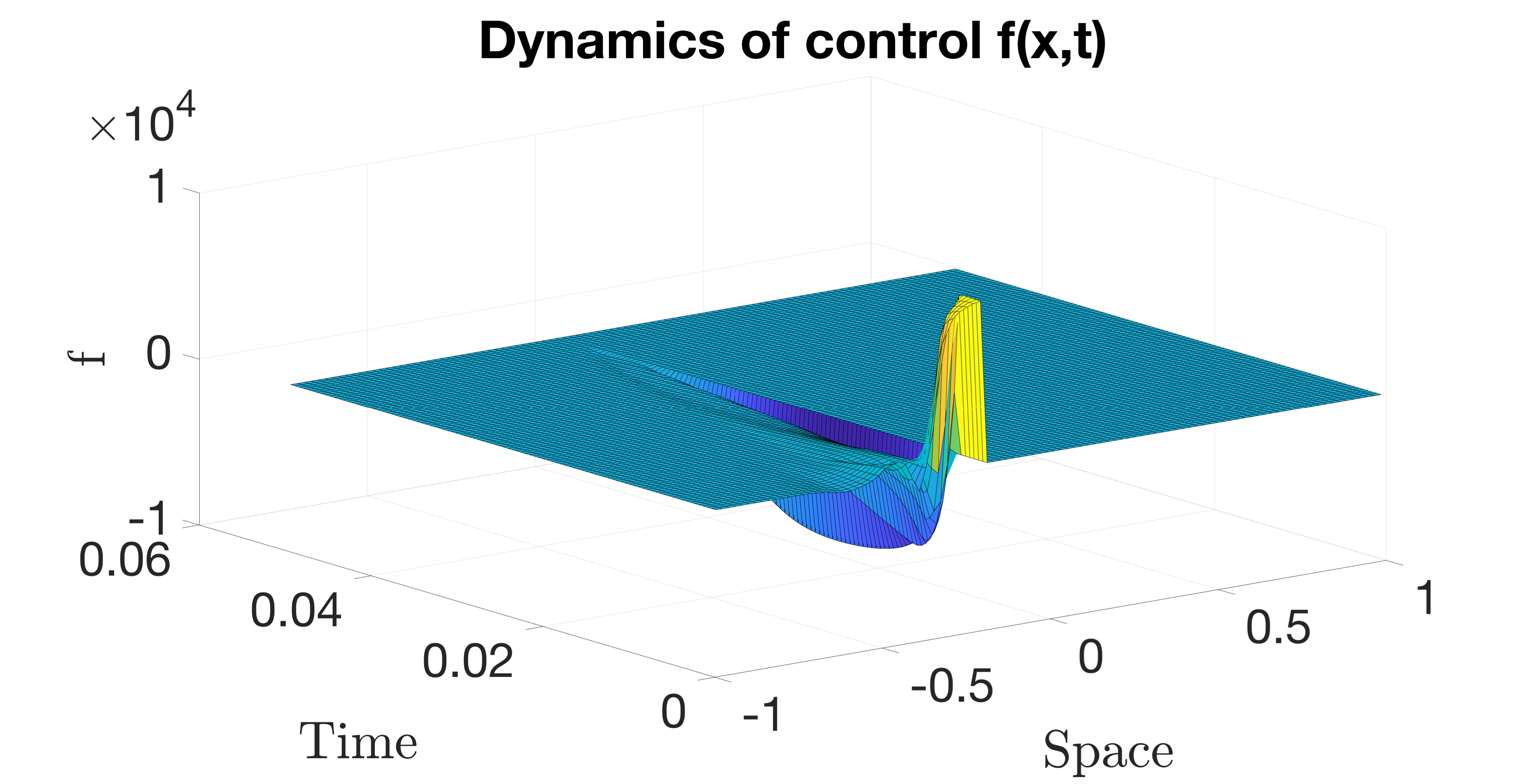}
\includegraphics[width=0.328\textwidth]{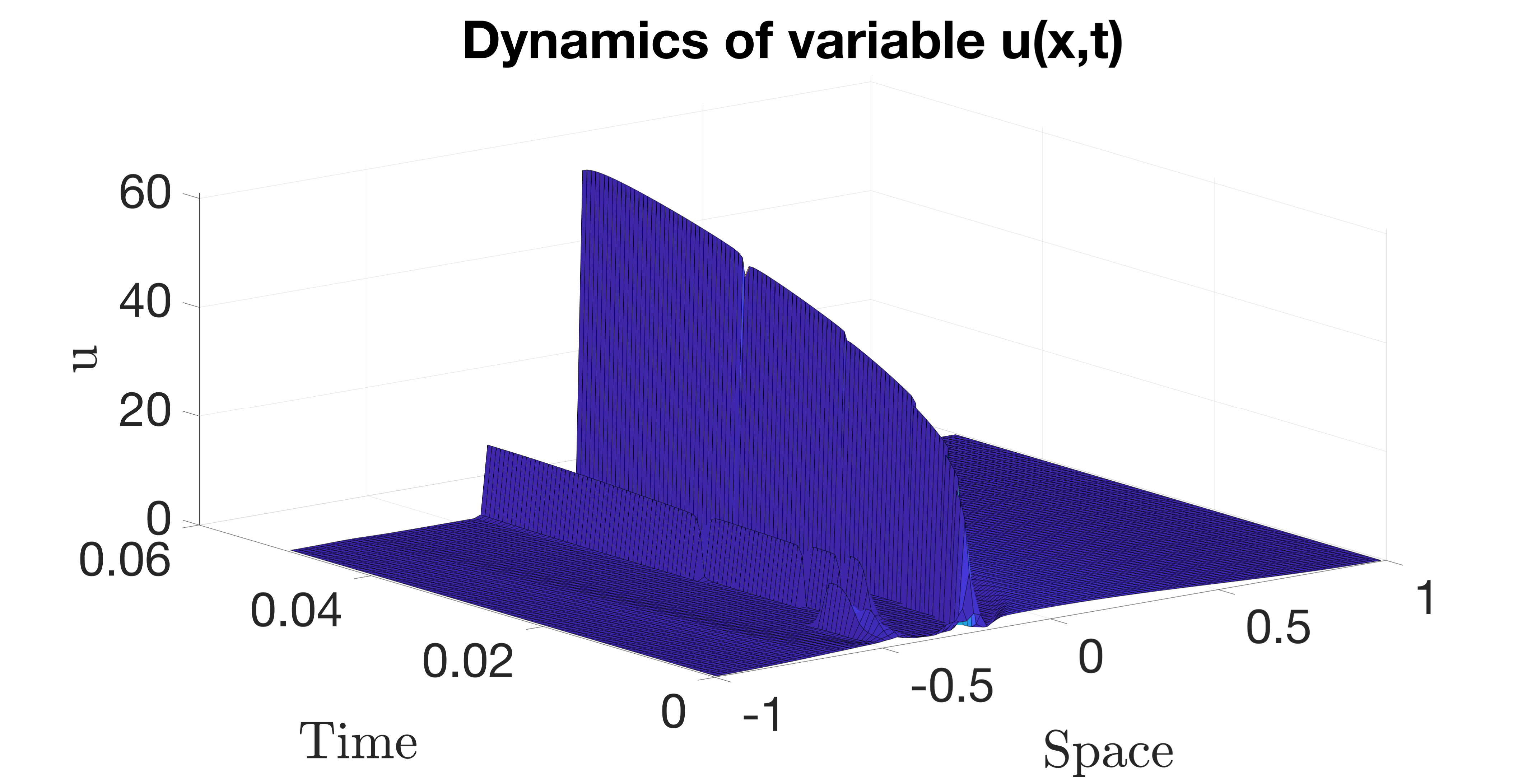}
\includegraphics[width=0.328\textwidth]{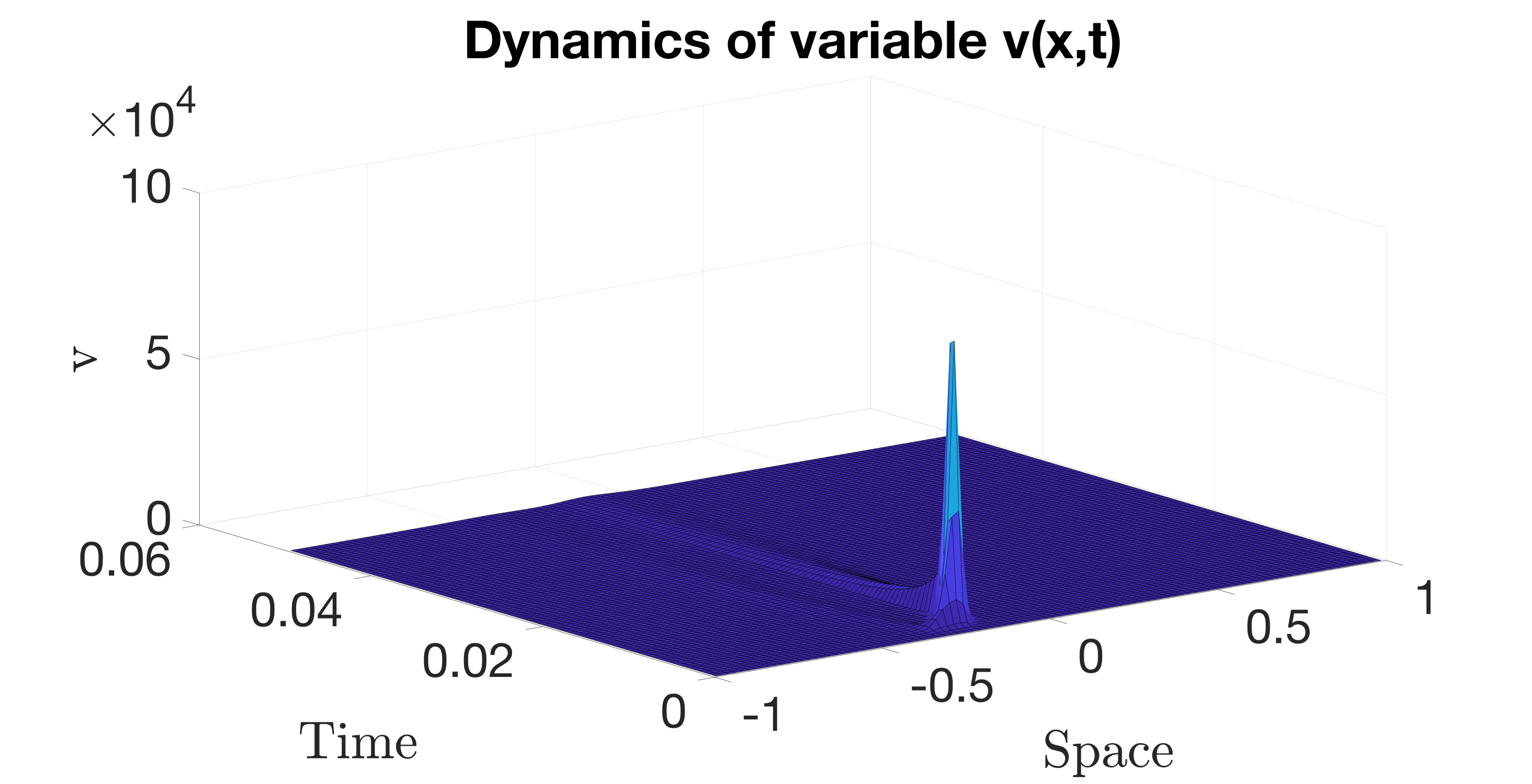}
\\
\includegraphics[width=0.328\textwidth]{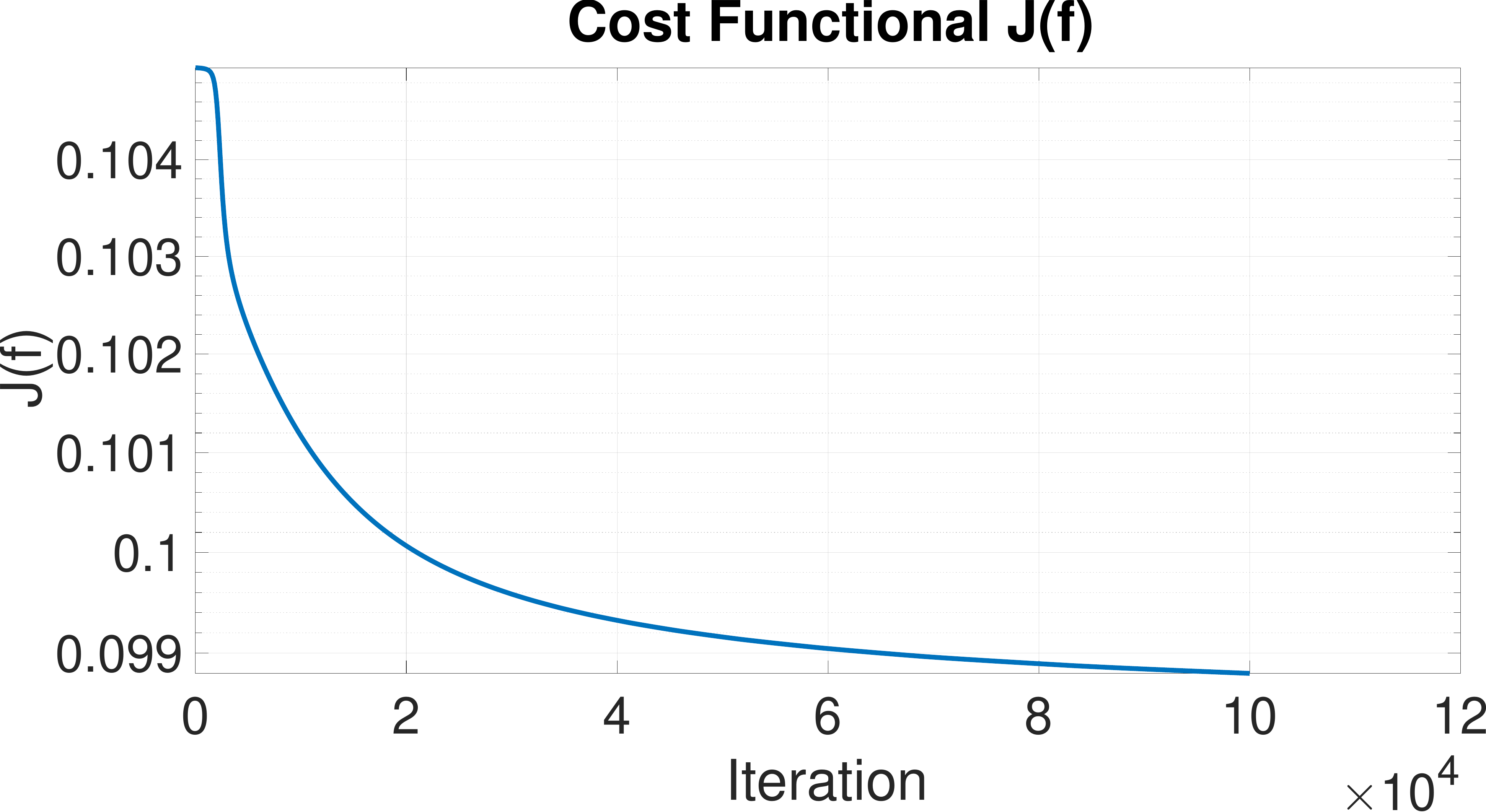}
\includegraphics[width=0.328\textwidth]{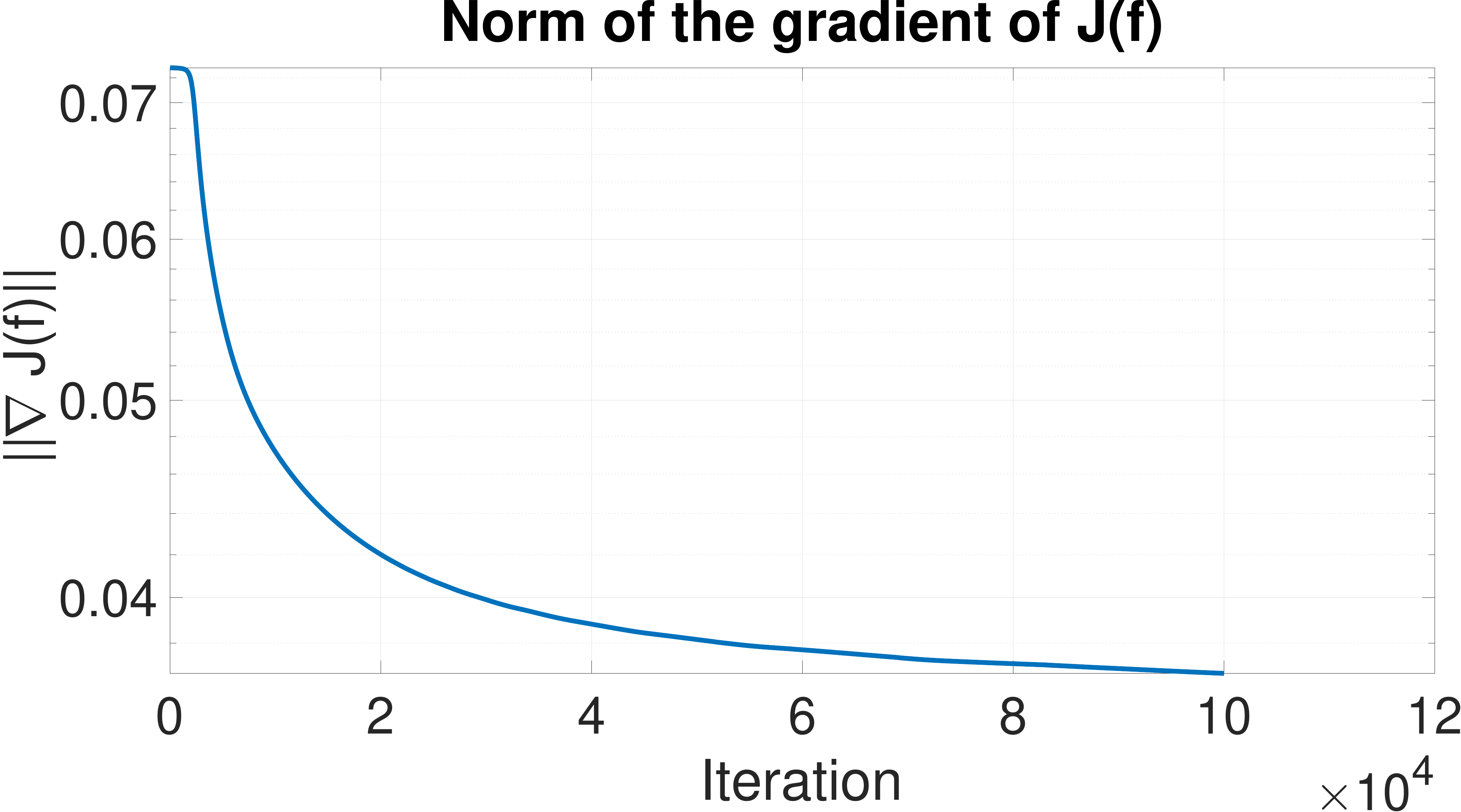}
\includegraphics[width=0.328\textwidth]{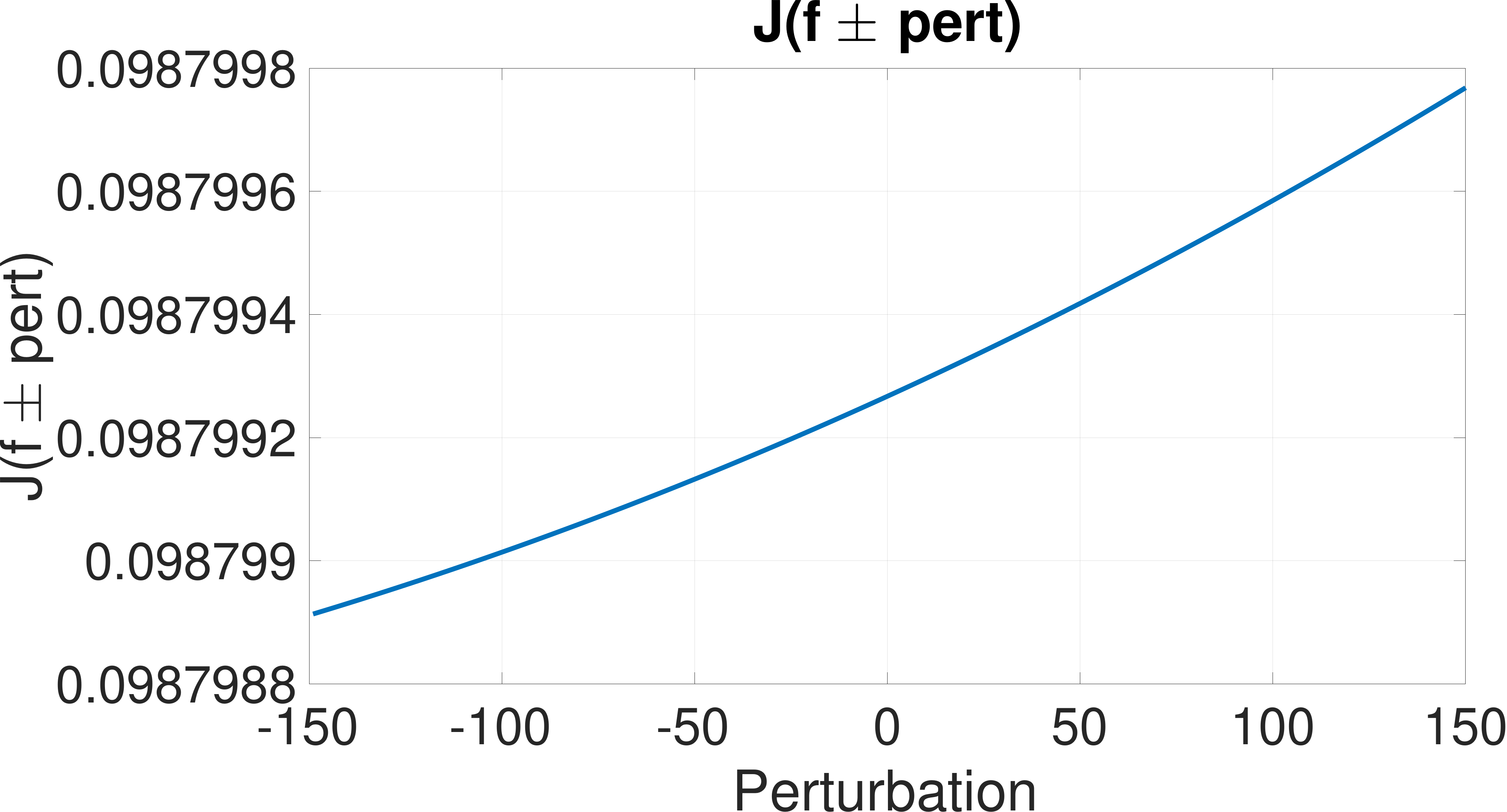}
\caption{Results for $\Omega_c=[-1,-0.2]$ and $\Omega_o=[0.2,1]$. Top row: Dynamics of control $f$ (left) and the associated variables $u$ (center) and variable $v$ (right). Bottom row: Evolution of the cost functional $J(f)$ (left), evolution of $\|\nabla J(f)\|$ (center) and the influence of perturbing the obtained control (right).}\label{fig:case5}
\end{center}
\end{figure}

\begin{figure}[H]
\begin{center}
\includegraphics[width=0.328\textwidth]{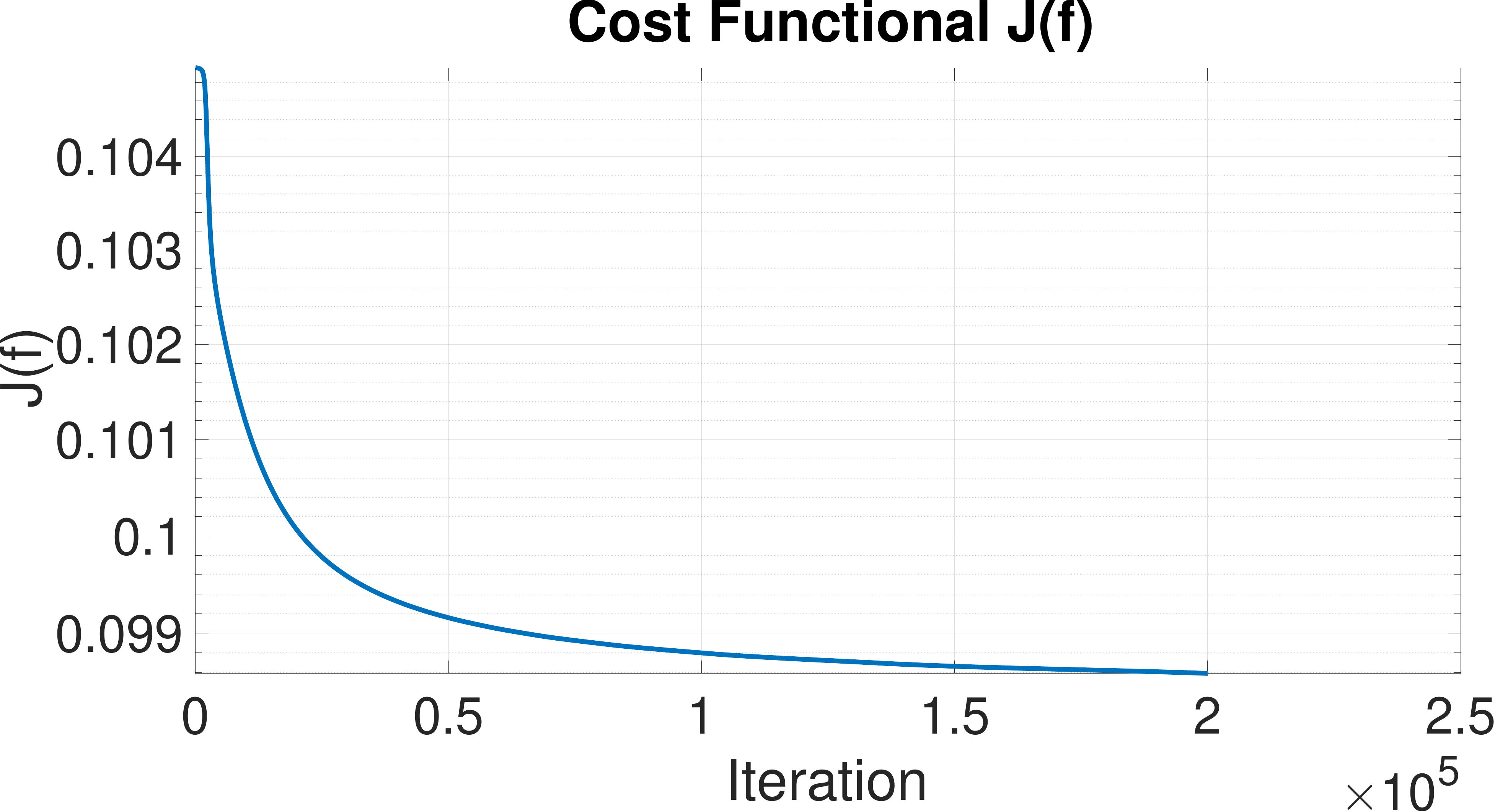}
\includegraphics[width=0.328\textwidth]{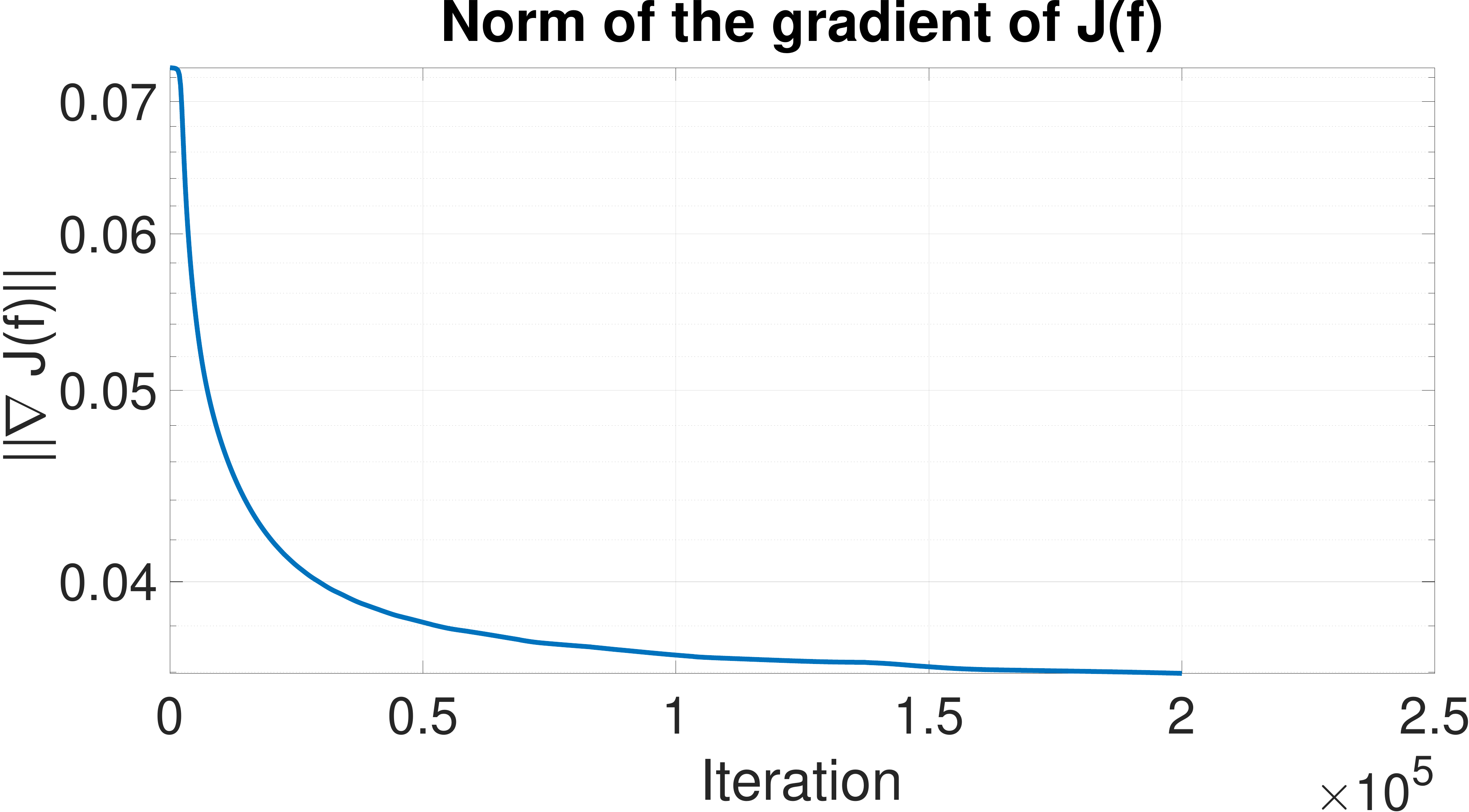}
\includegraphics[width=0.328\textwidth]{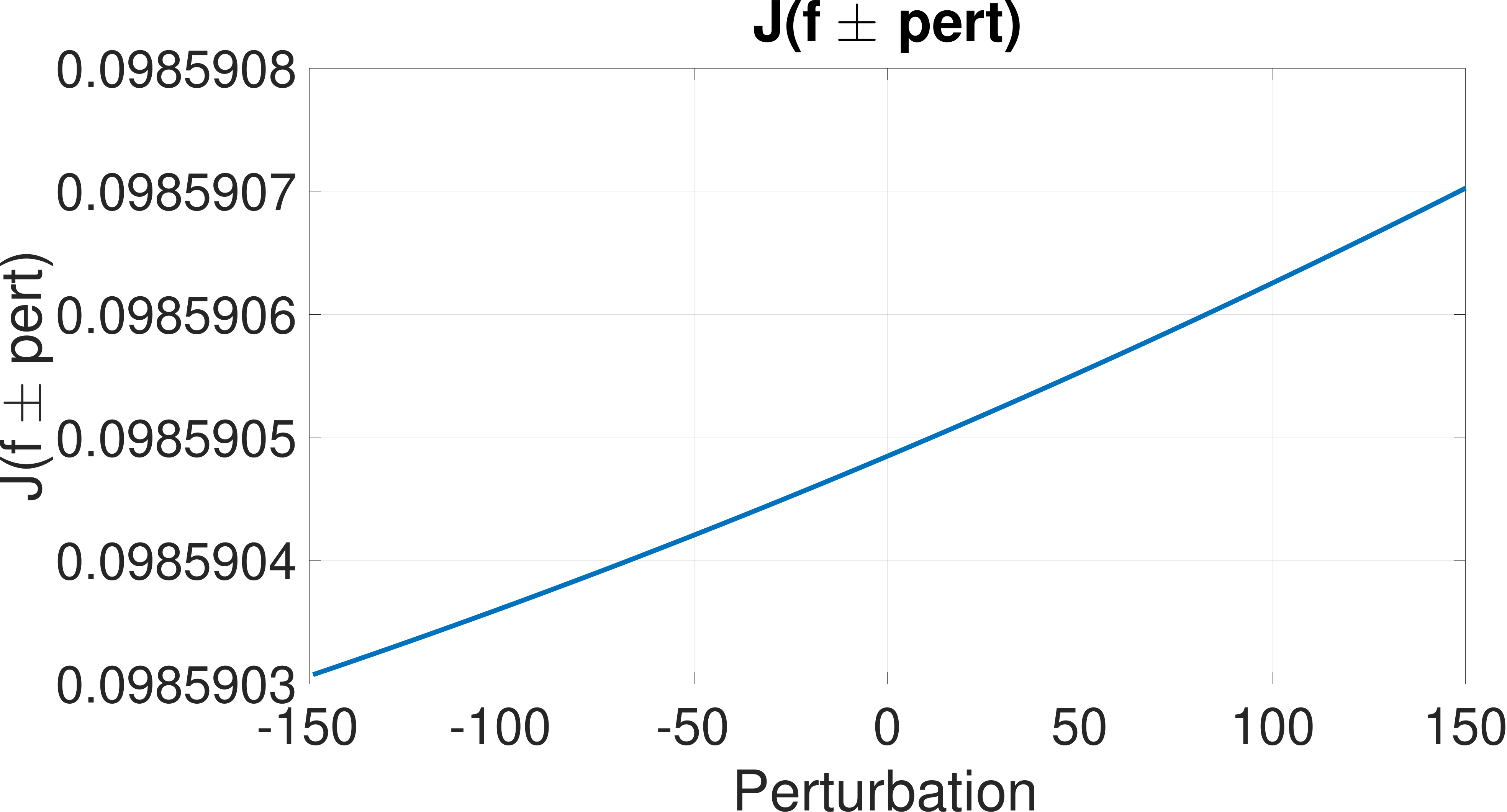}
\caption{Results for $\Omega_c=[-1,-0.2]$ and $\Omega_o=[0.2,1]$ with number of maximum of iterations set to $200000$. Evolution of the cost functional $J(f)$ (left), evolution of $\|\nabla J(f)\|$ (center) and the influence of perturbing the obtained control (right).}\label{fig:case5b}
\end{center}
\end{figure}

\subsection{Boundary Control}\label{subsec:simboundary}

In this section we study the effect of considering boundary controls $g(t)$ at both endpoints of the spatial interval ($x=-1$ and $x=1$) using two different type of boundary conditions, namely Robin and bilinear ones, as presented in \eqref{eq:typesBoundaryControls}. In particular we study how the system behaves when we consider different observation domains. In all the cases the initial conditions are
$$
u_0(\x)
\,=\,
1-\cos(\pi x)
\quad\mbox{ and }\quad
v_0(\x)
\,=\,
1-\cos(\pi x)\,.
$$
The desired state is set to be the constant value $u_d=\int_\Omega u_0dx=1$ and we consider $g(t)=0$ in both end points as the initial control in the minimization process. In order to point out  the type of influence  of the boundary control in the behavior of the system, we present in Figure~\ref{fig:Boundarycase0} the dynamics of variables $u$ and $v$ when we compute the dynamics without imposing any control acting on the system.


\begin{figure}[H]
\begin{center}
\includegraphics[width=0.49\textwidth]{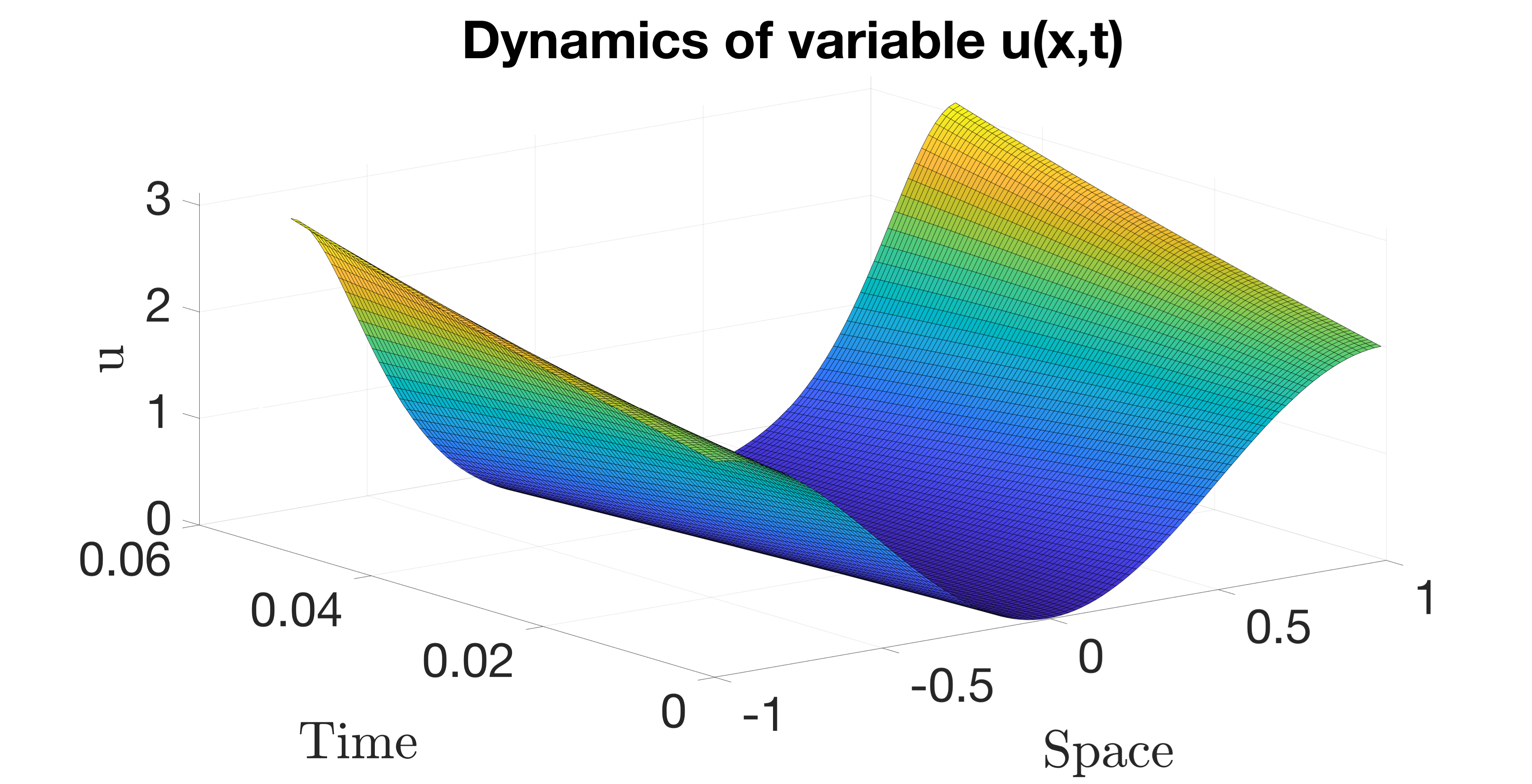}
\includegraphics[width=0.49\textwidth]{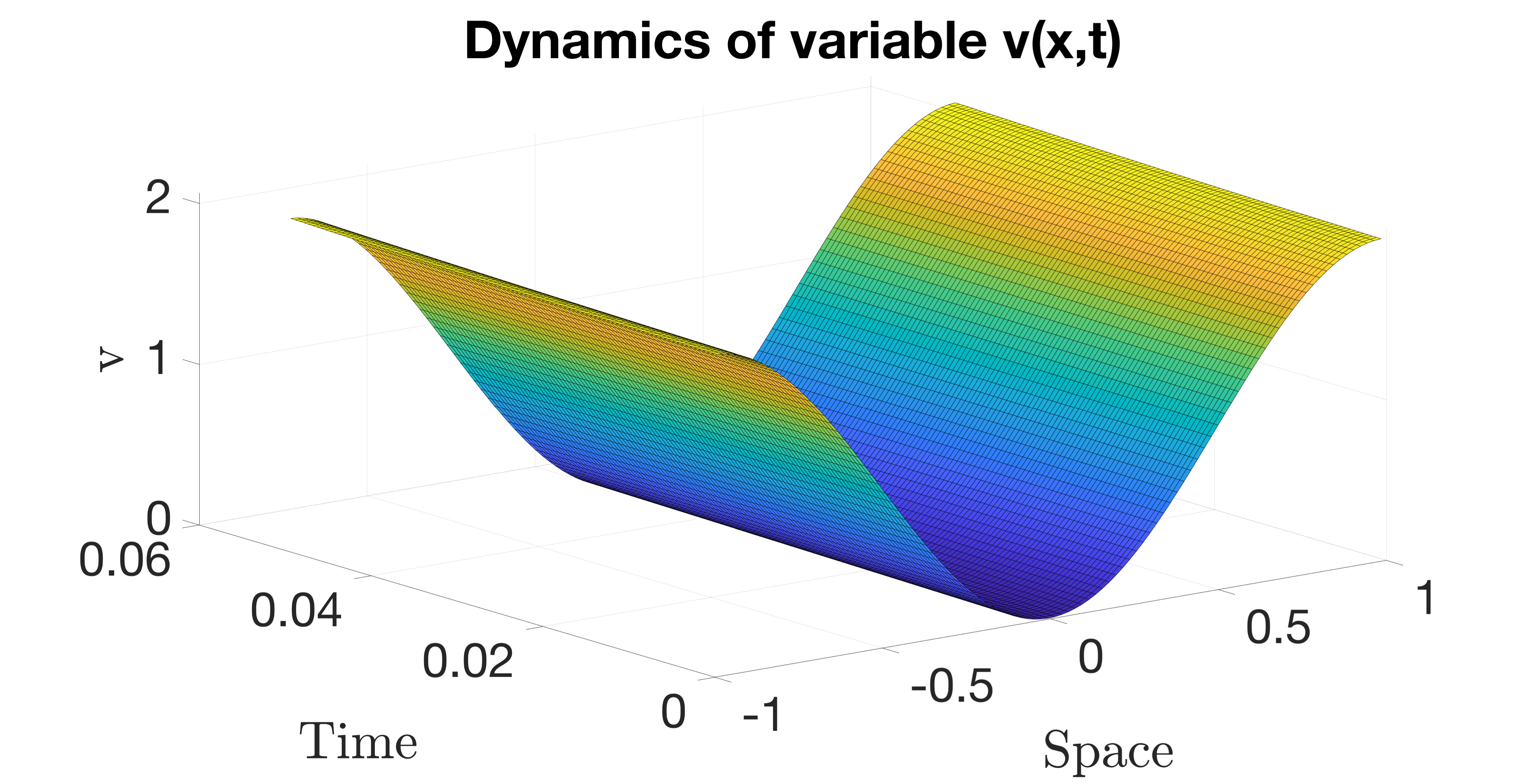}
\caption{Dynamics of variables $u$ (left) and $v$ (right) when no boundary control is acting on the system.}\label{fig:Boundarycase0}
\end{center}
\end{figure}

\subsubsection{Bilinear boundary control}

We start focusing on bilinear boundary conditions as presented in \eqref{eq:typesBoundaryControls} and we compare the differences in the dynamics of setting the observation domain to be the whole spatial domain or a subinterval located in the interior. In Figure~\ref{fig:BBCcase1} we present the results for $\Omega_o=[-1,1]$, where we can observe how the control on both sides acts mainly at the beginning of the simulations by adding and removing variable $v$, although their influence in the dynamics seems to be limited, and at the end it mainly does not change the dynamics of variables $u$ and $v$ if we compare them with the dynamics without using a control  in Figure~\ref{fig:Boundarycase0}. Moreover we see how the functional $J(g)$ has a slightly decreasing trend despite being oscillating, but we observe how the norm of its gradient is not really decreasing, just oscillates (at the final iteration $\|\nabla J(g)\|=0.041$). Finally, when perturbing simultaneously on both endpoints the obtained control, we realize that it's a local minimum and that adding control produces to go away of the minimum much faster than removing control.

\

On the other hand, we present in Figure~\ref{fig:BBCcase2} the results for $\Omega_o=[-0.5,0.5]$. We see how in this case the control starts 'aggressively' taking out variable $v$, which induces a much more severe behavior on variable $u$ compared with the previous case. Additionally, the functional $J(f)$  and the norm of its gradient are always decreasing with respect to the iterations, but the value of  the norm of the gradient is still far from the desired tolerance when the algorithm achieves the maximum number of iterations (at the final iteration $\|\nabla J(g)\|=0.0097$). We also observe that the obtained control is not exactly a local minimum although the values of $J$ are basically constant (the differences are of the order of $10^{-8}$).

\begin{figure}[H]
\begin{center}
\includegraphics[width=0.328\textwidth]{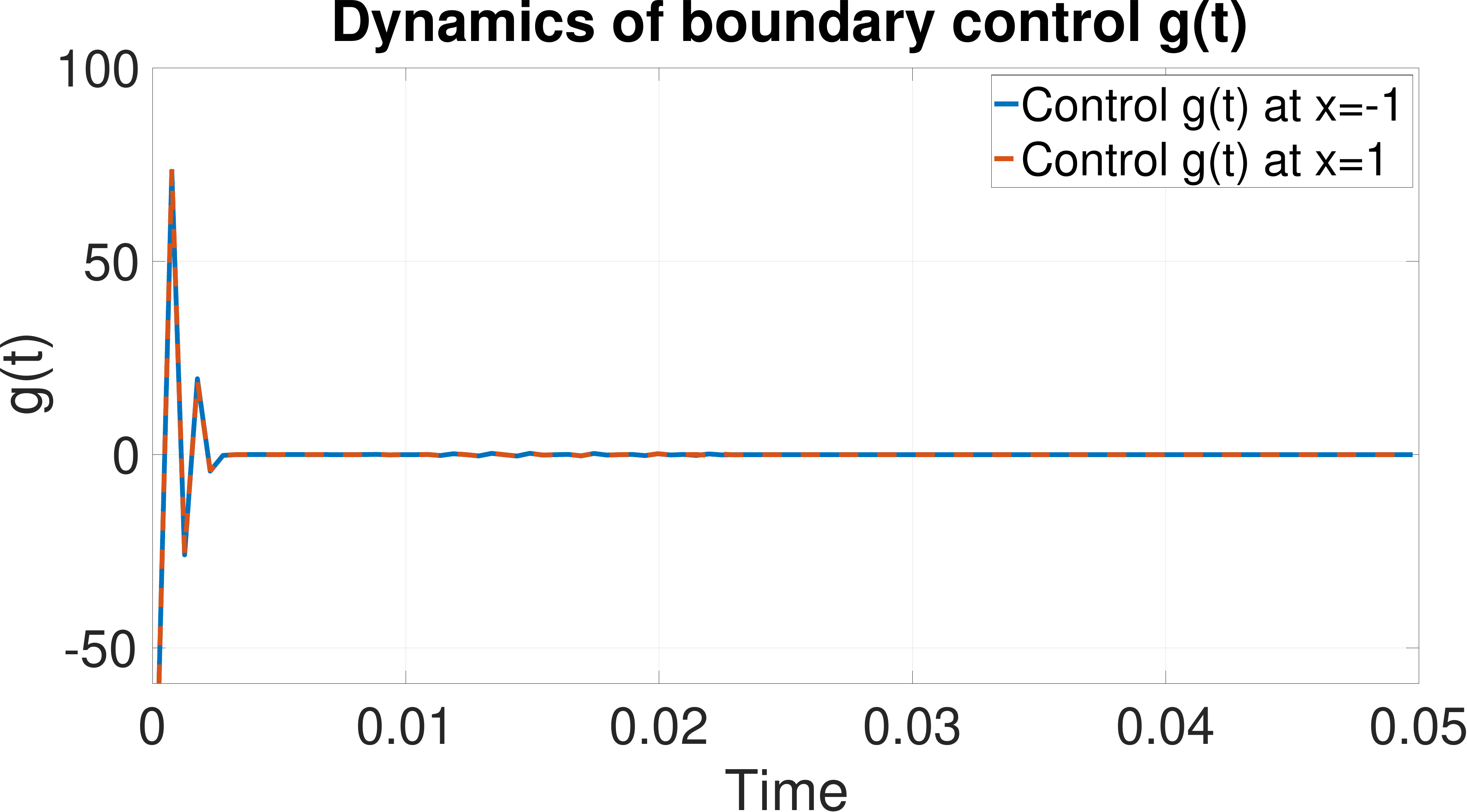}
\includegraphics[width=0.328\textwidth]{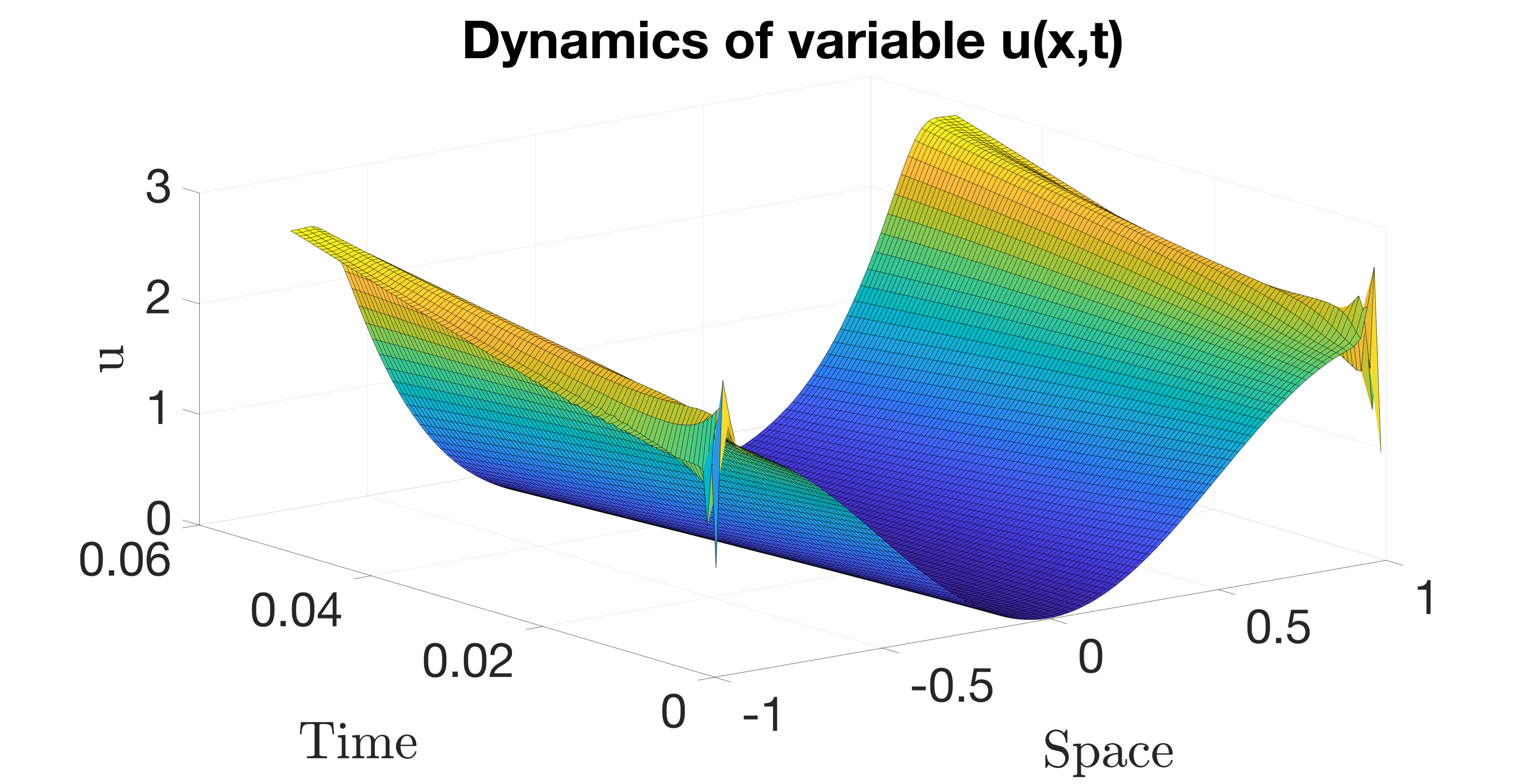}
\includegraphics[width=0.328\textwidth]{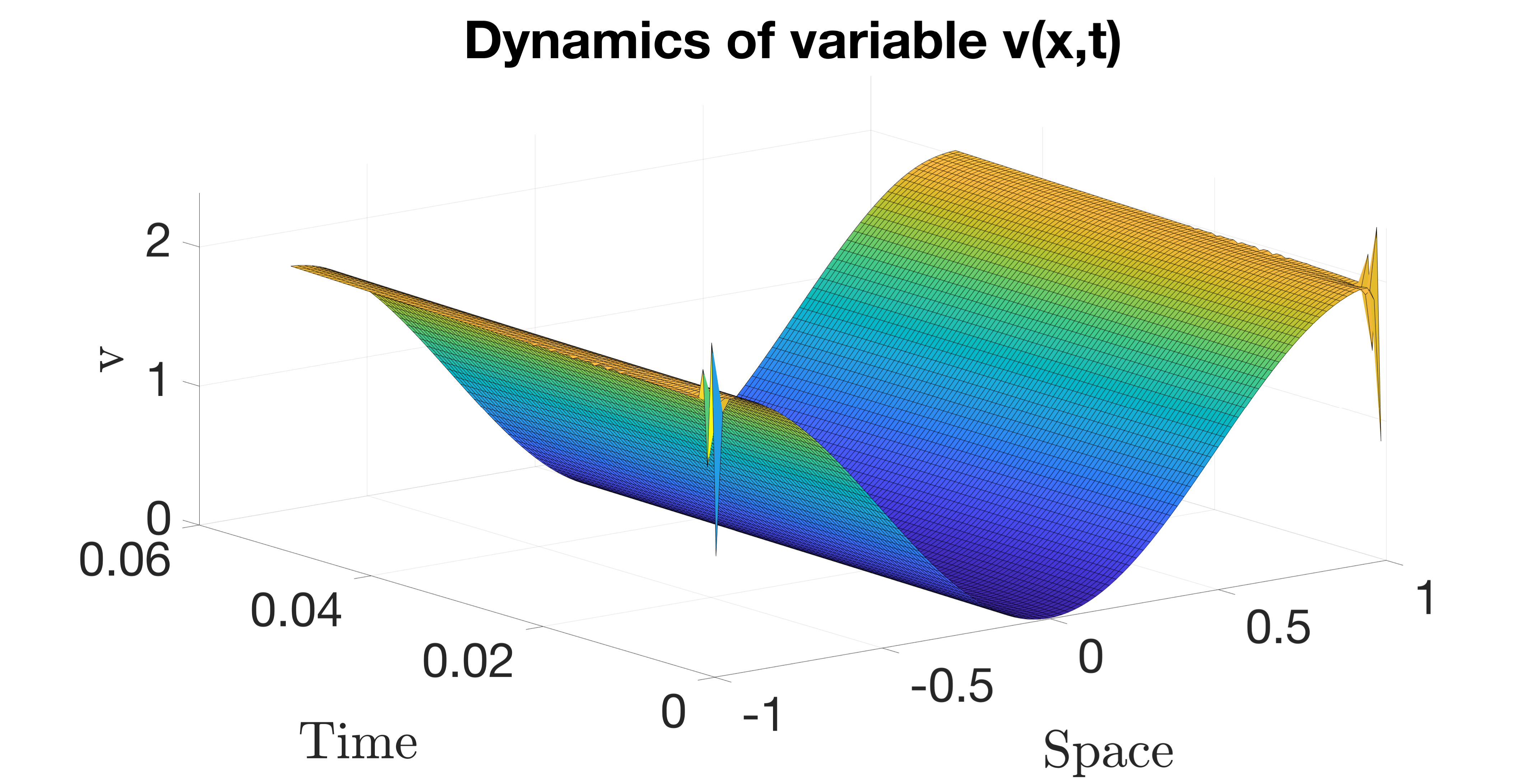}
\\
\includegraphics[width=0.328\textwidth]{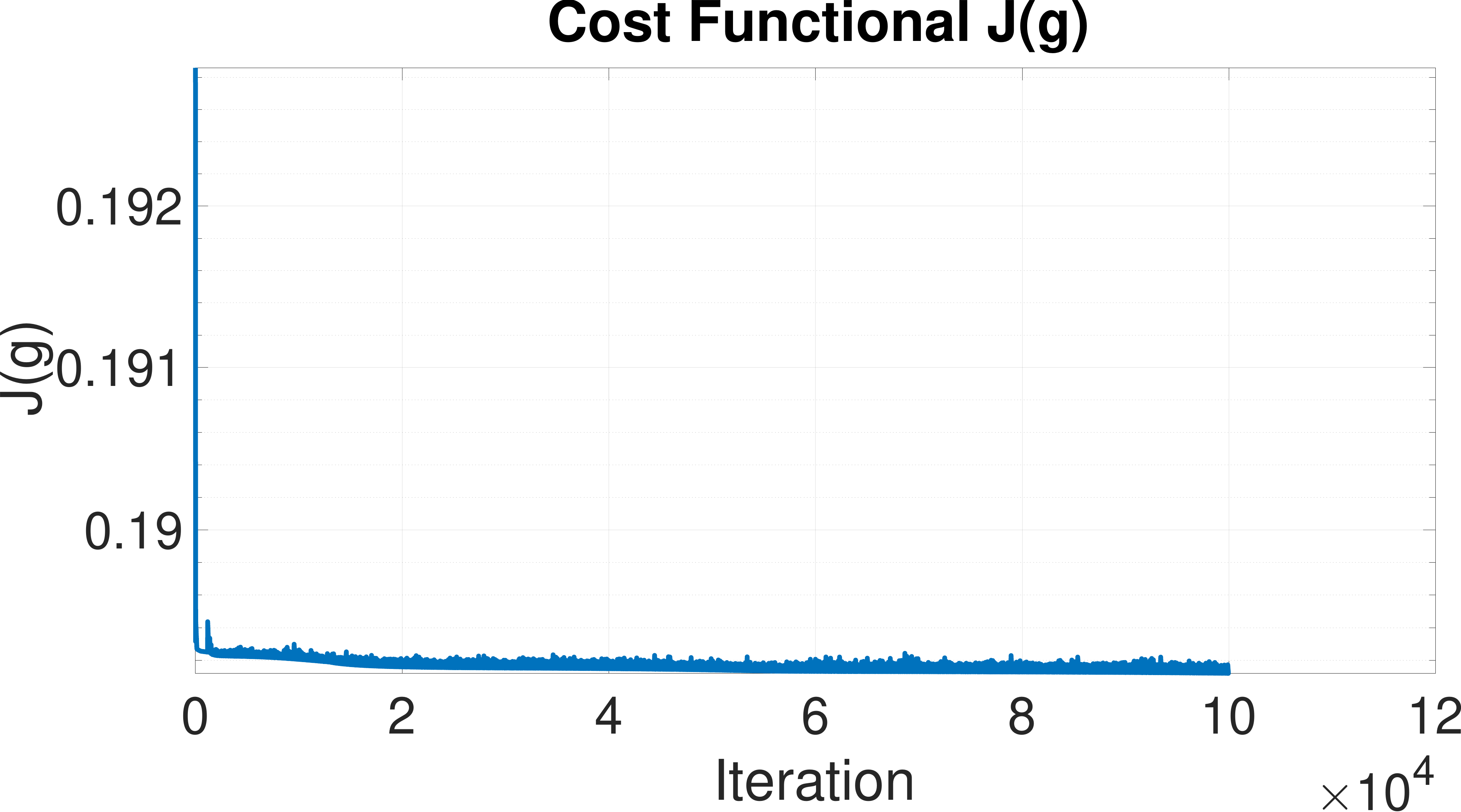}
\includegraphics[width=0.328\textwidth]{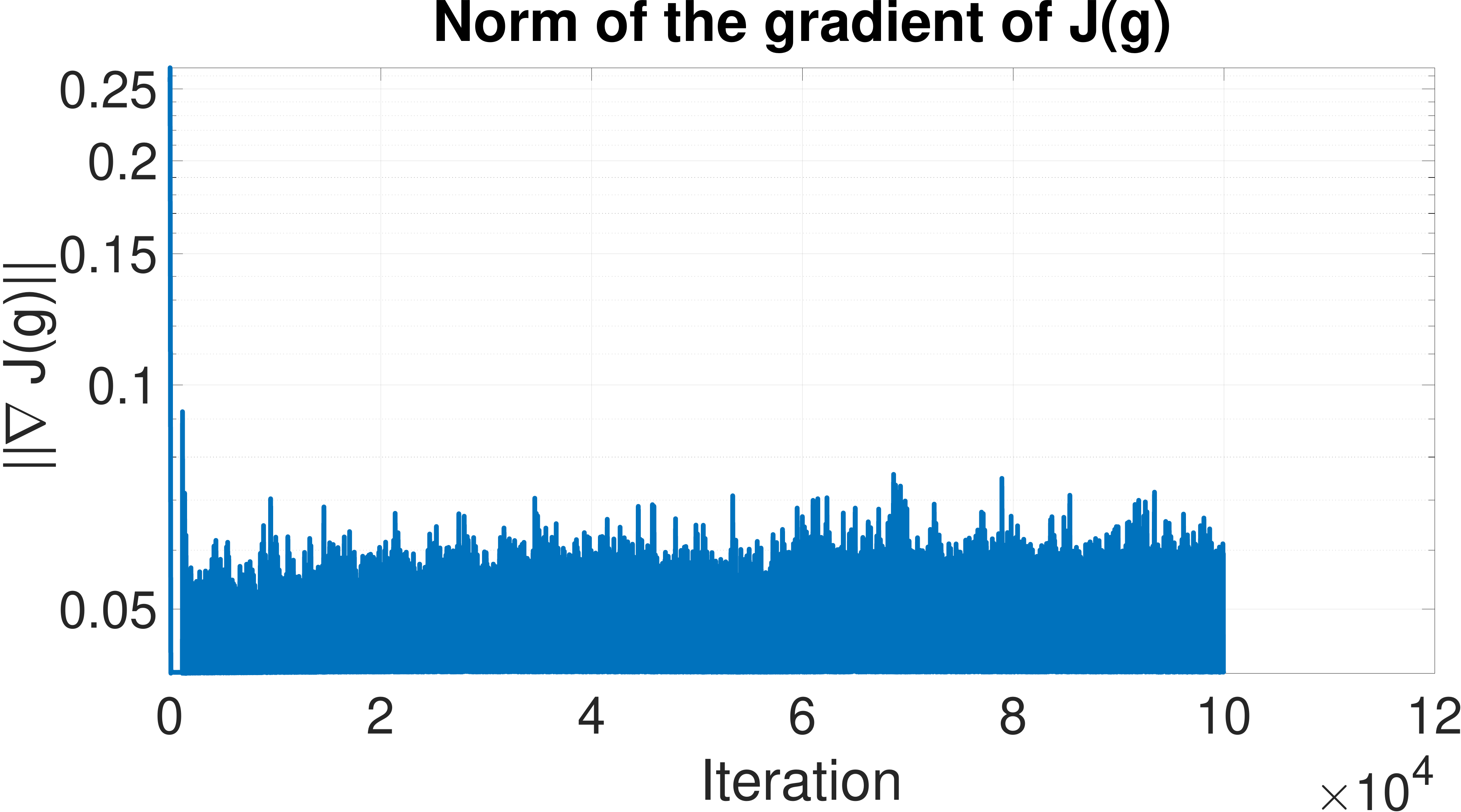}
\includegraphics[width=0.328\textwidth]{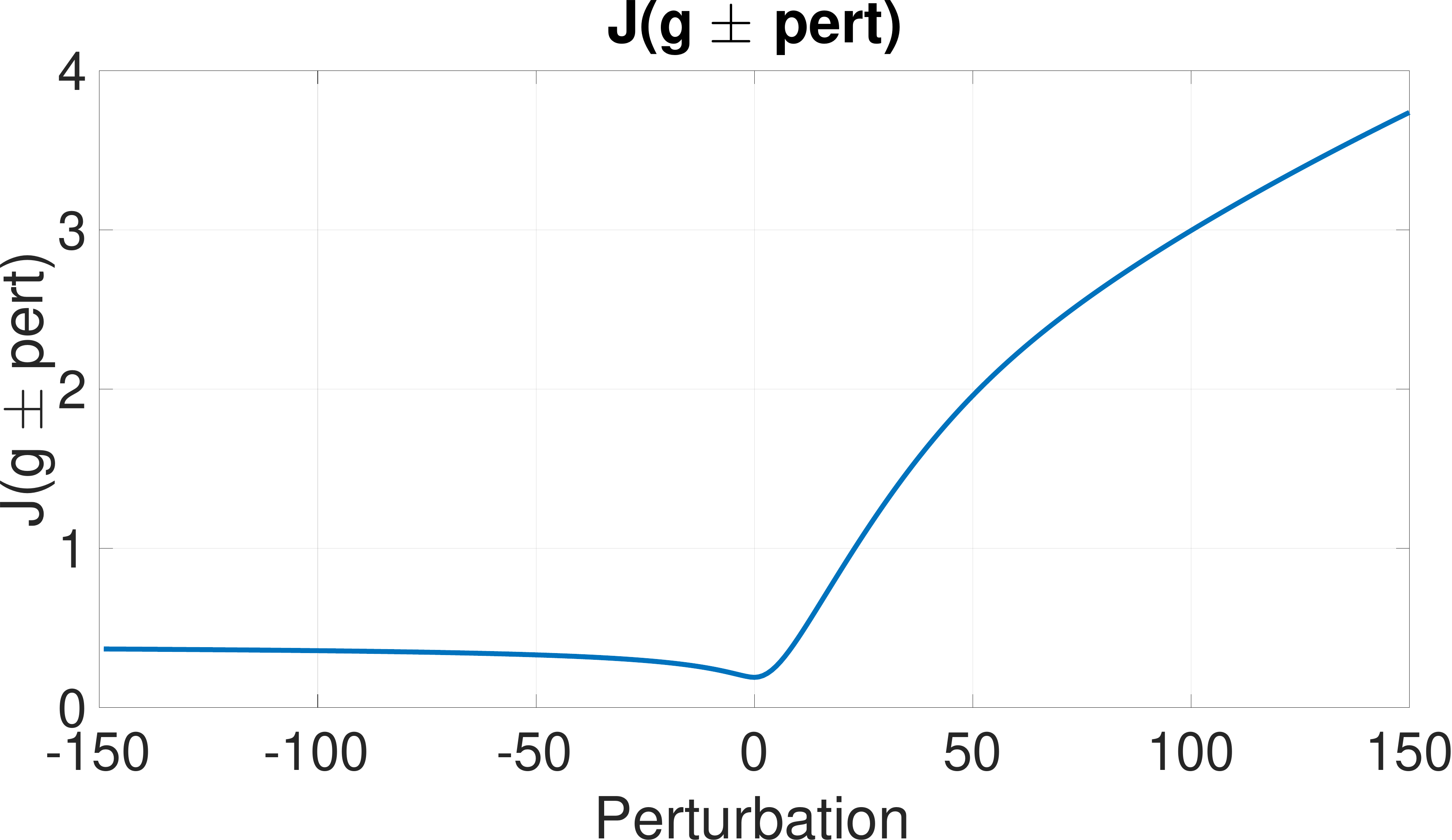}
\caption{Results for Bilinear boundary conditions with $\Omega_o=[-1,1]$. Top row: Dynamics of control $g$ (left) and the associated variables $u$ (center) and variable $v$ (right). Bottom row: Evolution of the cost functional $J(g)$ (left), evolution of $\|\nabla J(g)\|$ (center) and the influence of perturbing the obtained control (right).}\label{fig:BBCcase1}
\end{center}
\end{figure}

\begin{figure}[H]
\begin{center}
\includegraphics[width=0.328\textwidth]{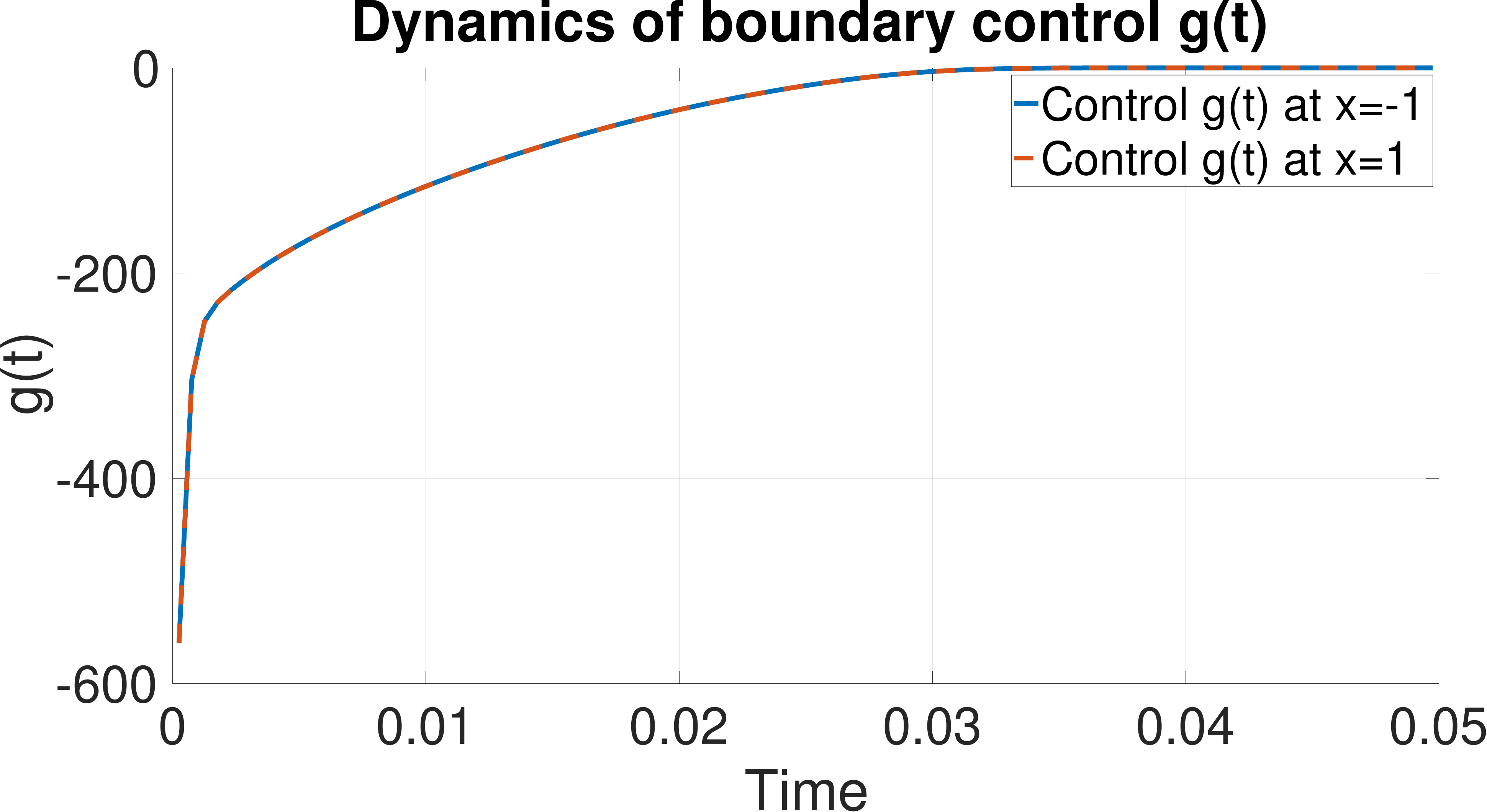}
\includegraphics[width=0.328\textwidth]{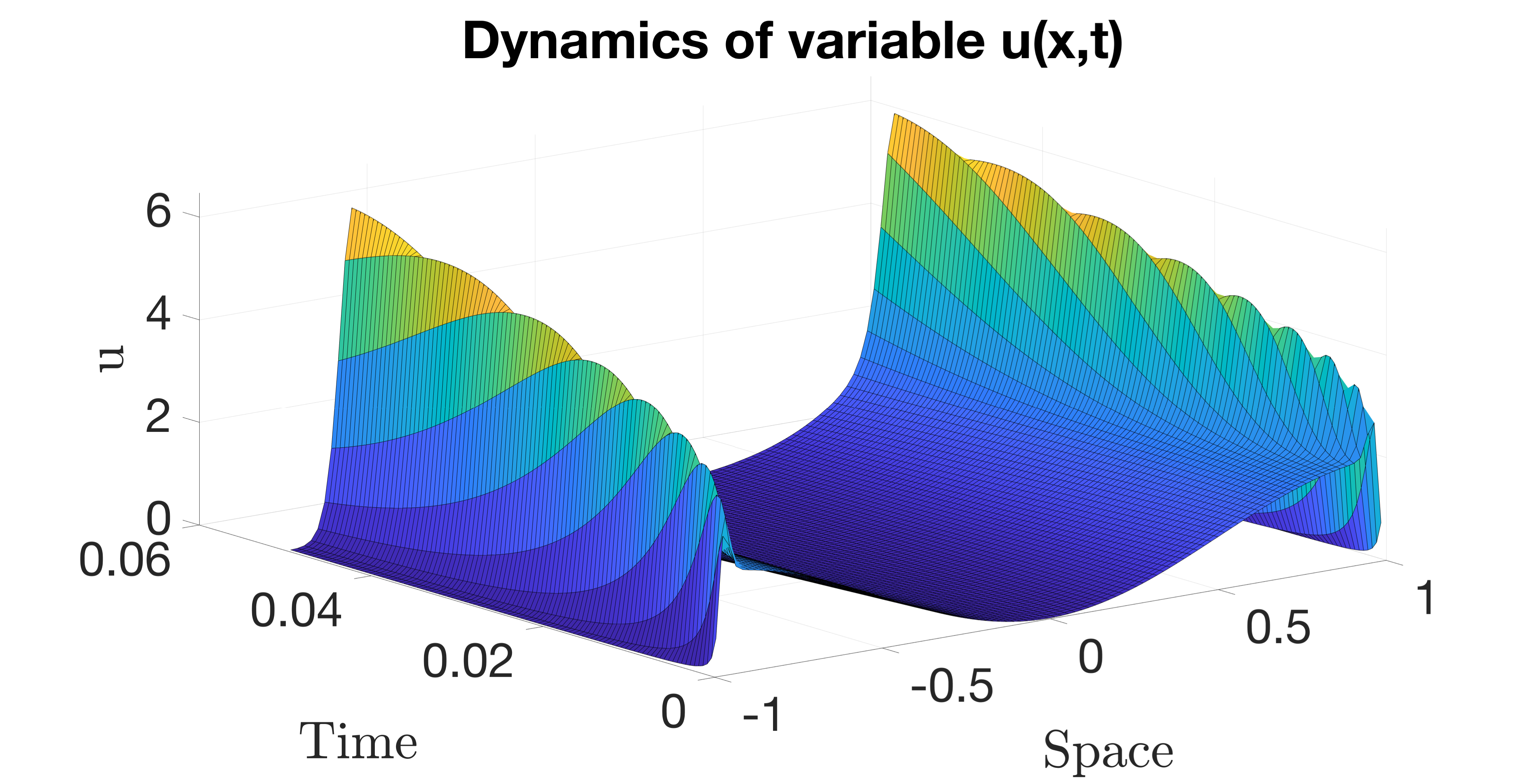}
\includegraphics[width=0.328\textwidth]{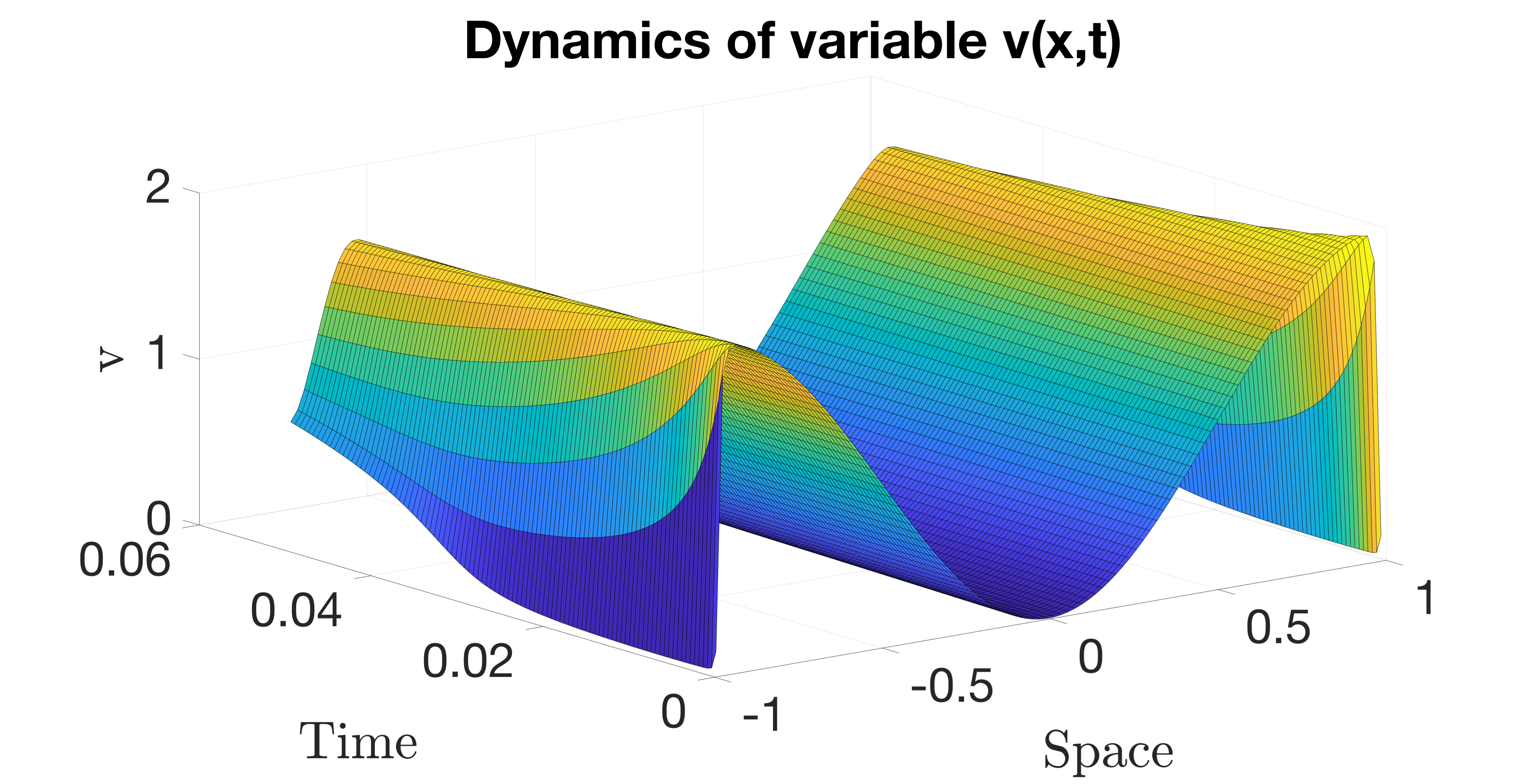}
\\
\includegraphics[width=0.328\textwidth]{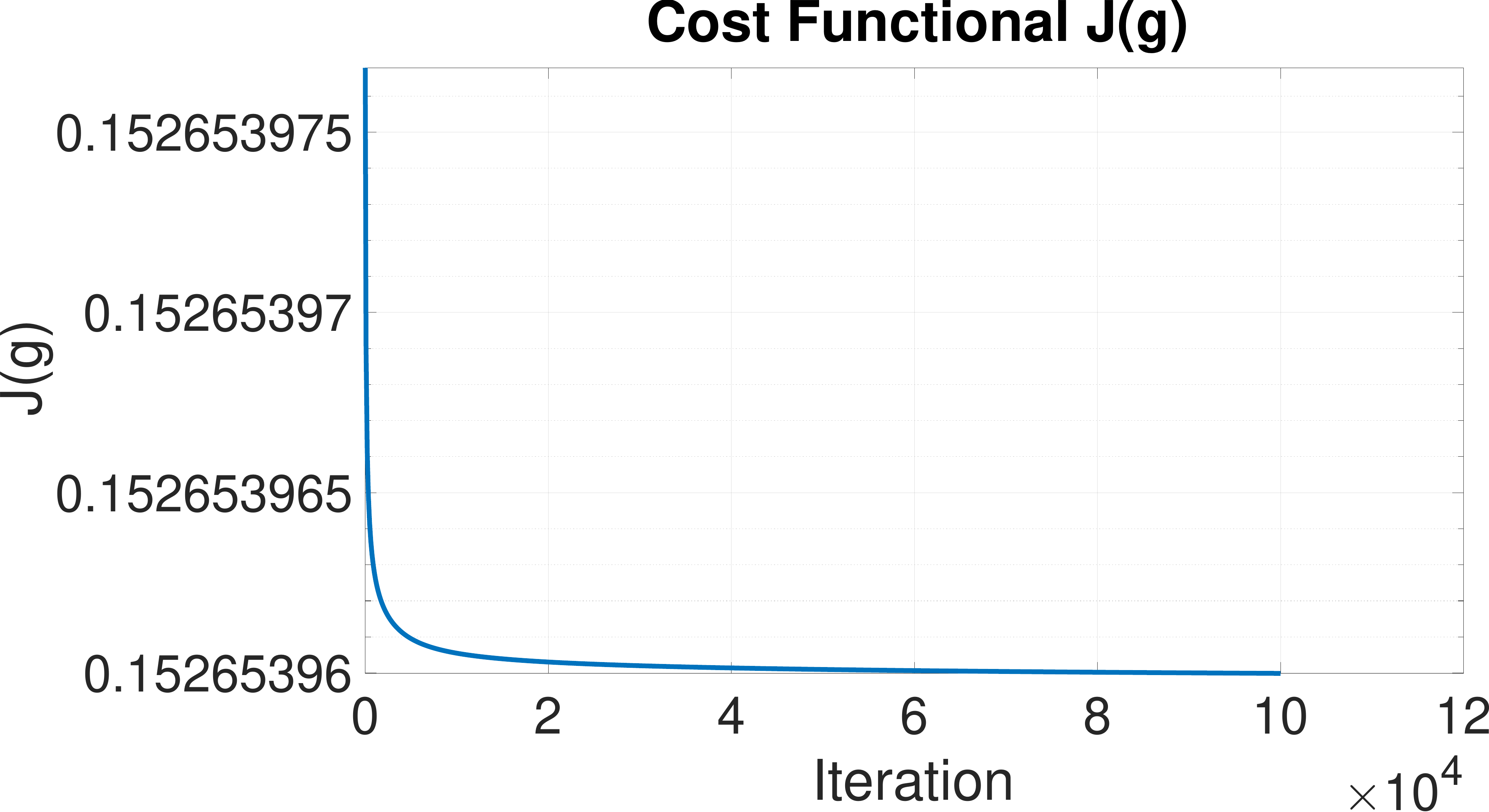}
\includegraphics[width=0.328\textwidth]{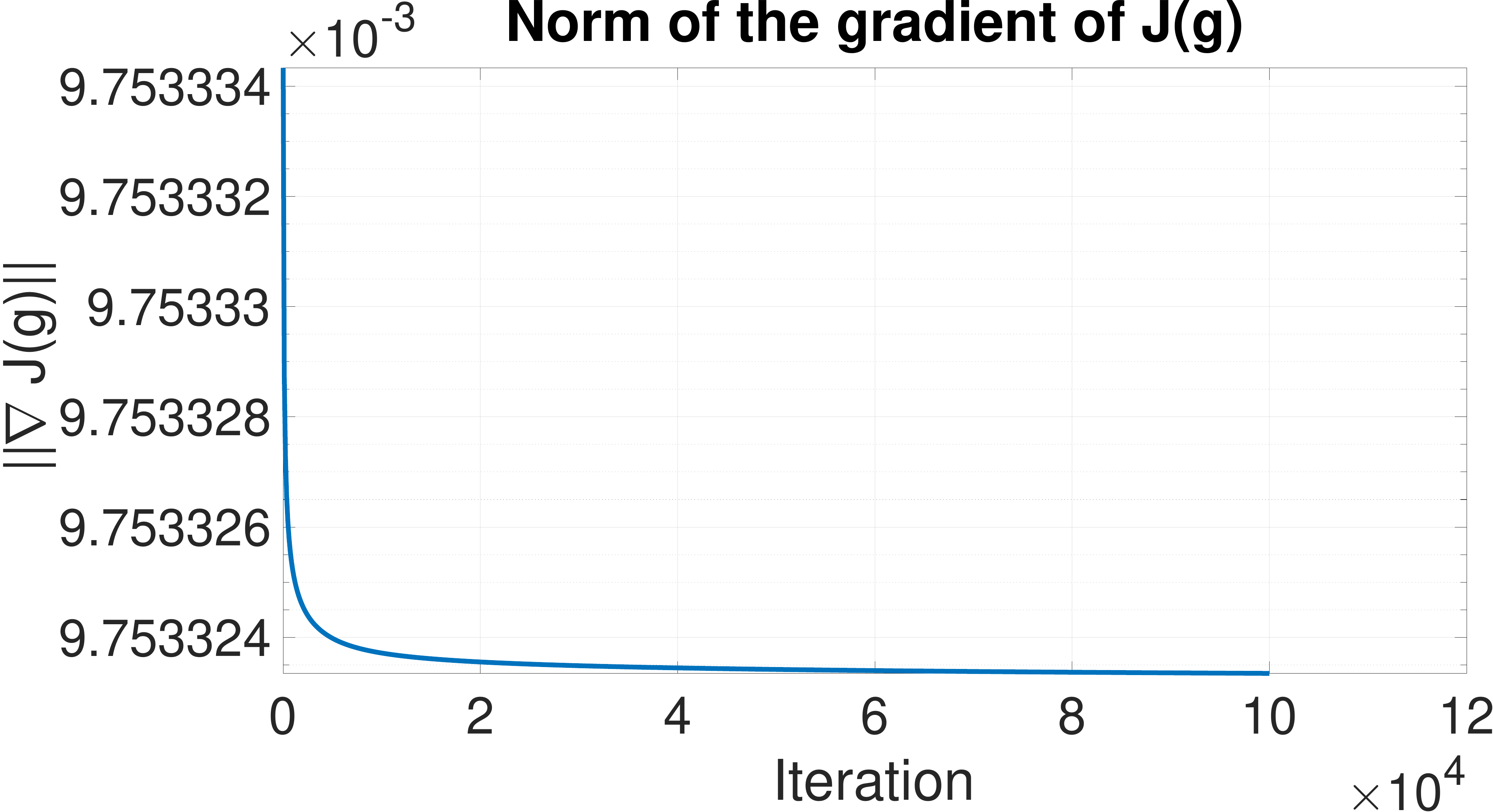}
\includegraphics[width=0.328\textwidth]{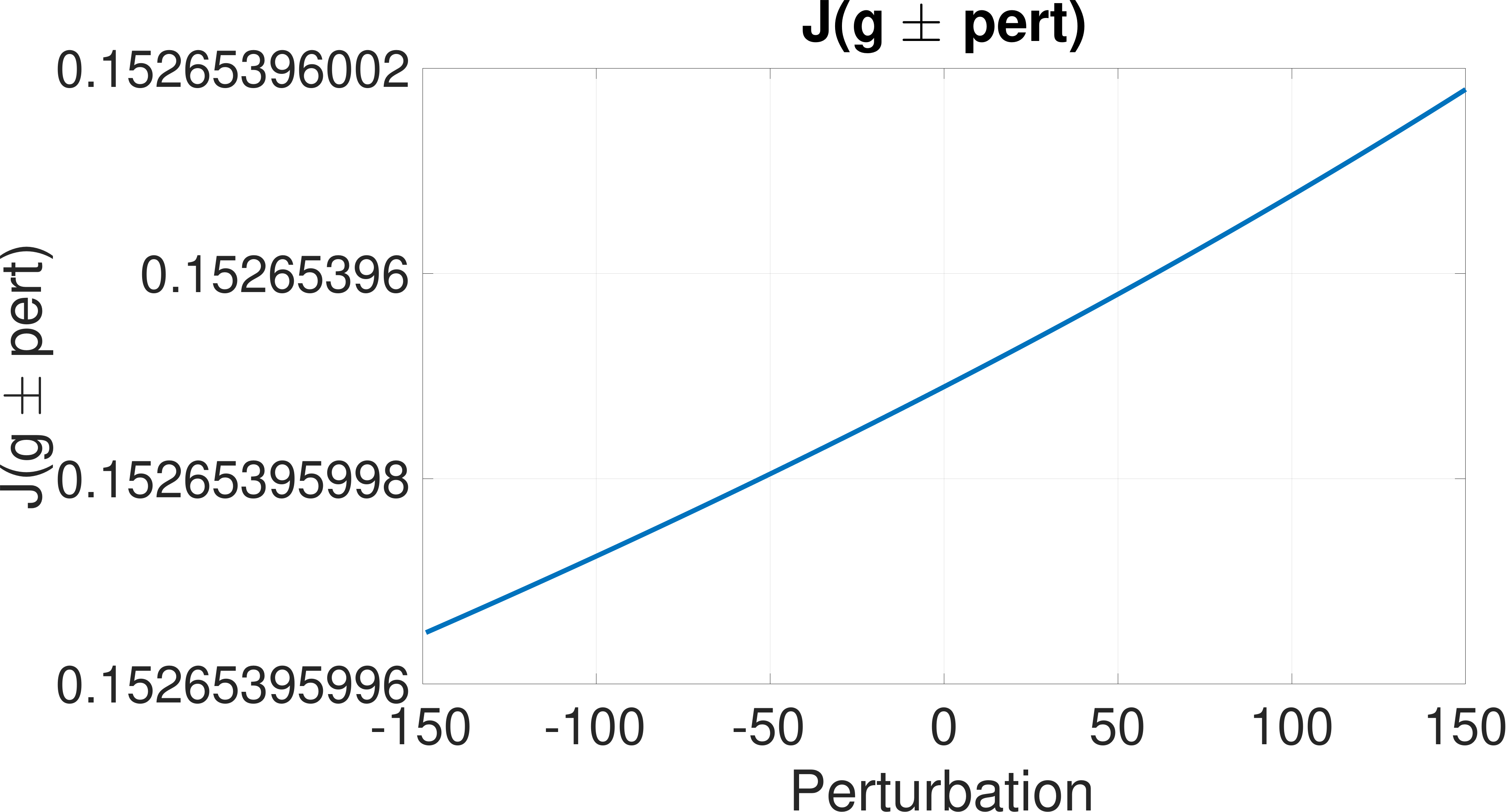}
\caption{Results for Bilinear boundary conditions with $\Omega_o=[-0.5,0.5]$. Top row: Dynamics of control $g$ (left) and the associated variables $u$ (center) and variable $v$ (right). Bottom row: Evolution of the cost functional $J(g)$ (left), evolution of $\|\nabla J(g)\|$ (center) and the influence of perturbing the obtained control (right).}\label{fig:BBCcase2}
\end{center}
\end{figure}

\subsubsection{Robin boundary control}

In this case we compare the effects on the choice of the observation domain when we consider the Robin boundary conditions presented in \eqref{eq:typesBoundaryControls} (taking the permeability parameter as $\sigma=1$). First, we present in Figure~\ref{fig:RBCcase1} the results observing in the whole domain, that is  $\Omega_o=[-1,1]$. In this case, the control on both sides  oscillates during the first half of the time interval (adding/removing signal $v$ hence oscillations on the values of $u$ close to the endpoints of the spatial domain are produced) while in the second part  the control gets to the quantity of signal on the boundary hence the flux vanishes. Anyway, as in the case of bilinear boundary conditions, the effect of the Robin boundary control is very limited. Moreover, the functional $J(g)$ decreases at the beginning of the computations but then it remains constant for the rest of the iterations, while the norm of its gradient just oscillates (at the final iteration $\|\nabla J(g)\|=0.041$). Finally, it seems that the system achieves a local minimum by how the functional behaves when evaluated at perturbations of the obtained control (adding control produces to go away of the minimum much faster than removing control).

Additionally, the results only observing in an interior subdomain, $\Omega_o=[-0.5,0.5]$, are presented in Figure~\ref{fig:RBCcase2}. Interestingly, in this case the computed control takes always the value zero, which produces the system to remove variable $v$ through the boundary during the whole time interval, producing changes on variable $u$ close to the boundary. Since control $g$ is constant (equal to zero) on both boundary points, then both the value of $J(g)$ and its gradient remains constant in all iterations. We also observe by perturbing the obtained control that the system have achieved a local minimum in this case.

Finally, comparing two different observation domains, we have observed that the 
Robin boundary control is more limited  when the observation is near to the boundary, because if the control acts it produces changes near to the boundary, which is an opposite effect with respect to the constant desired state.

\begin{figure}[H]
\begin{center}
\includegraphics[width=0.328\textwidth]{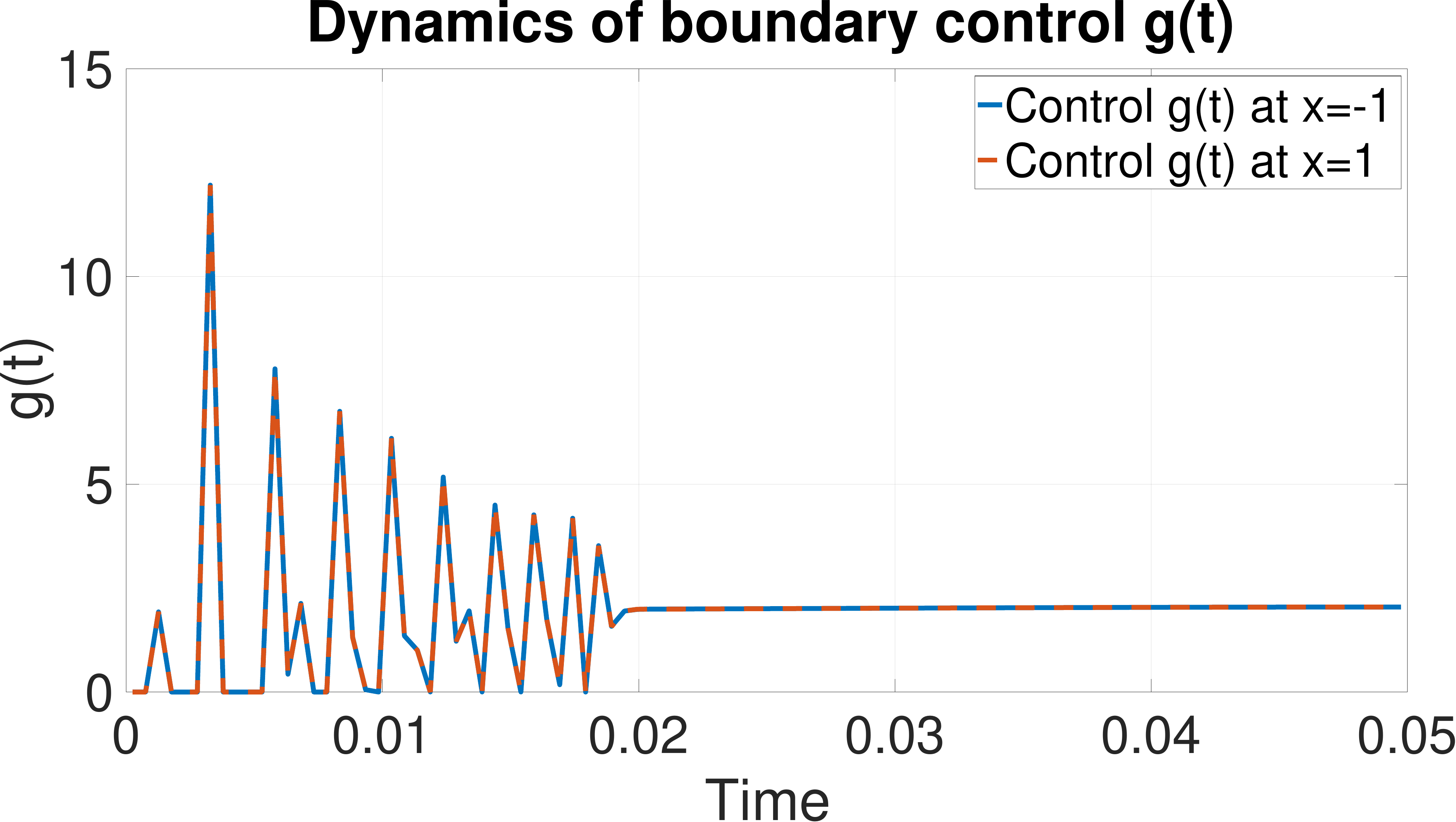}
\includegraphics[width=0.328\textwidth]{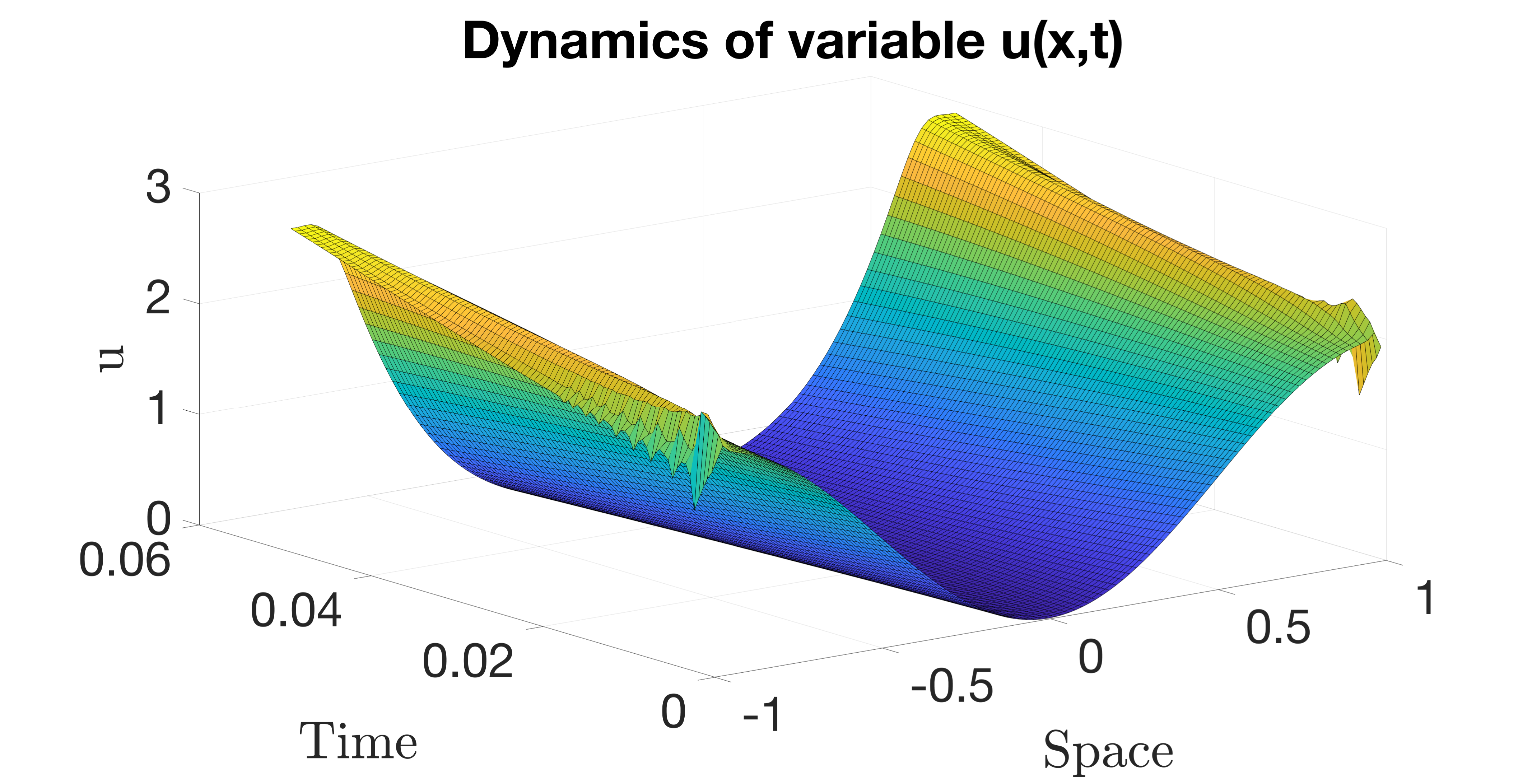}
\includegraphics[width=0.328\textwidth]{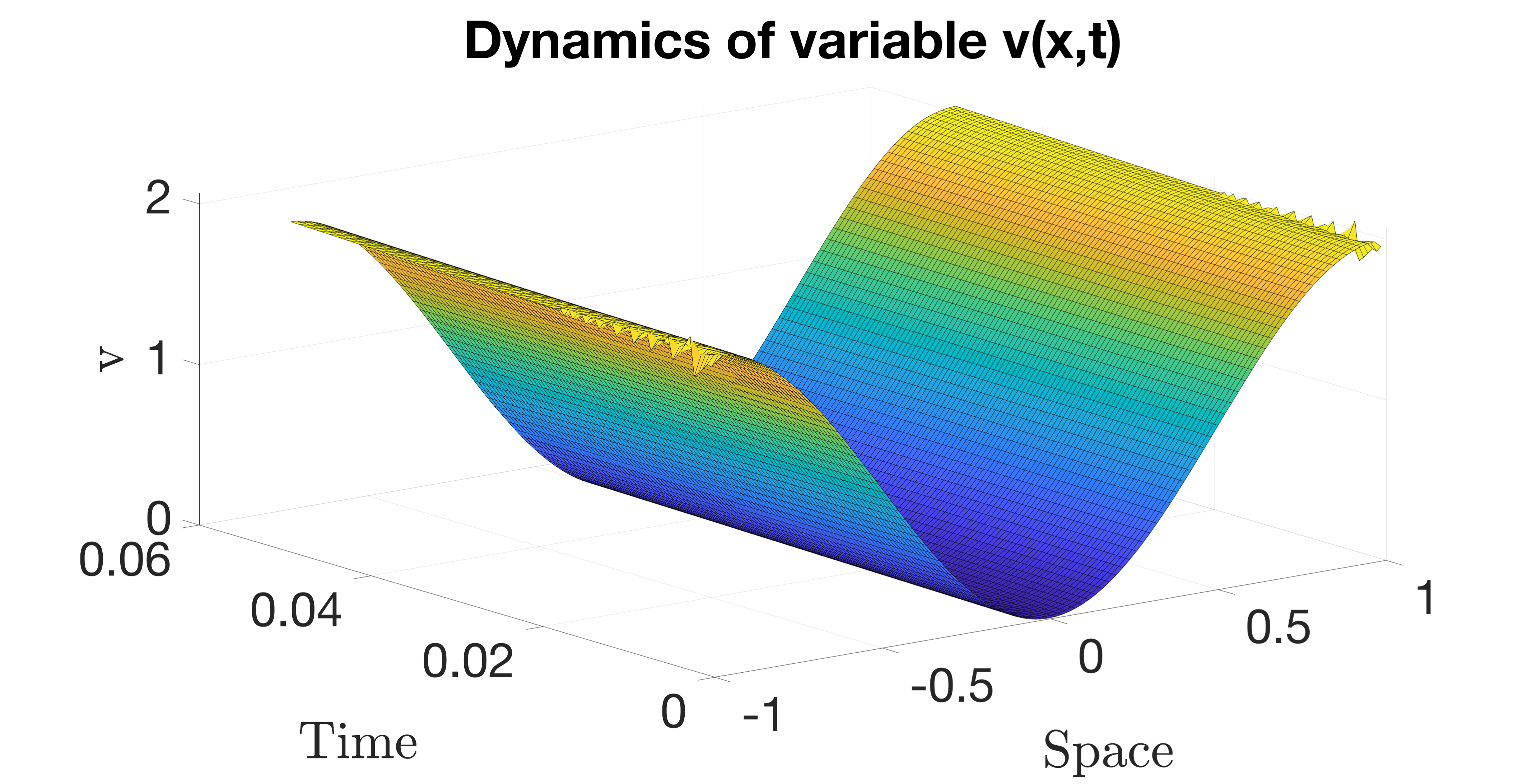}
\\
\includegraphics[width=0.328\textwidth]{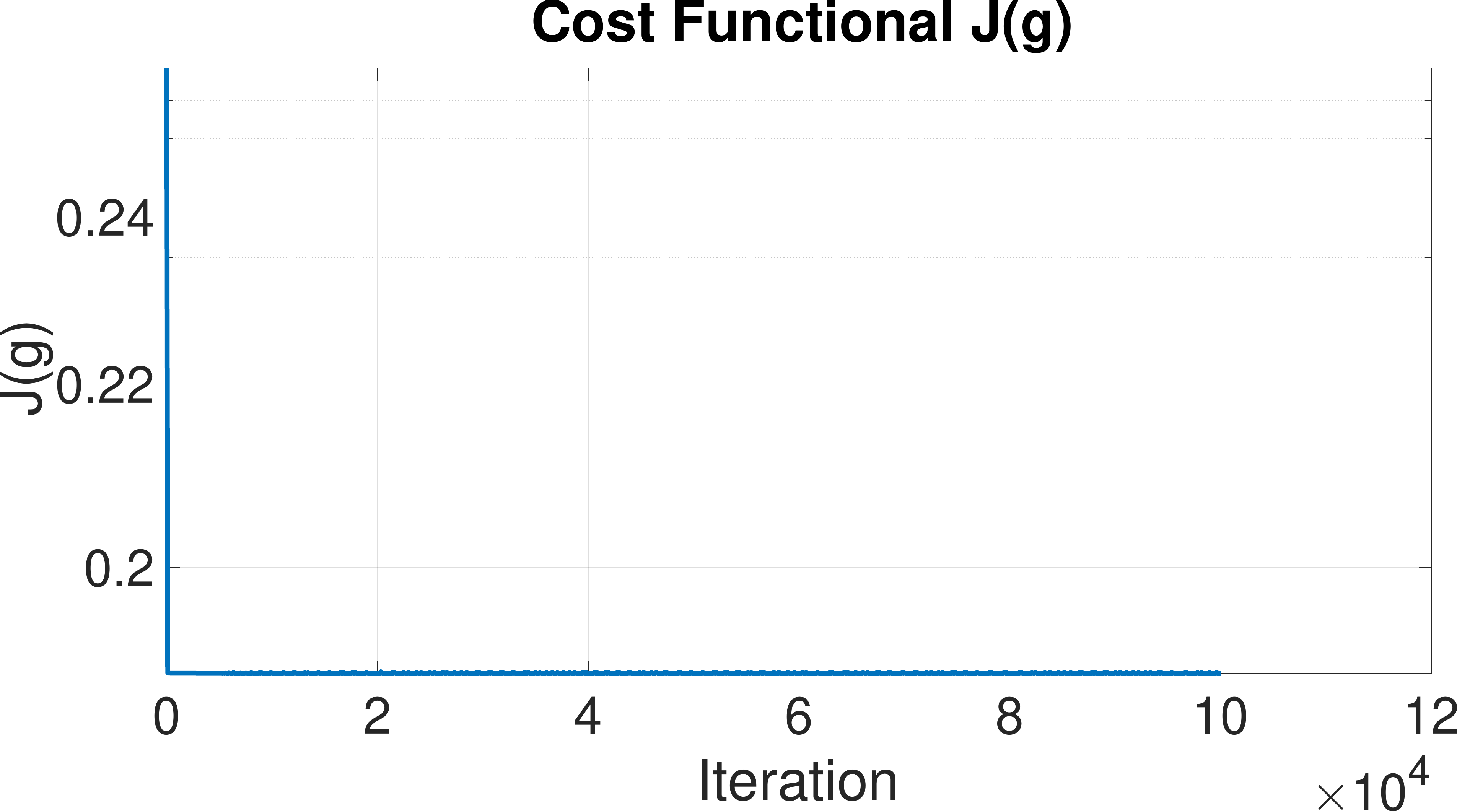}
\includegraphics[width=0.328\textwidth]{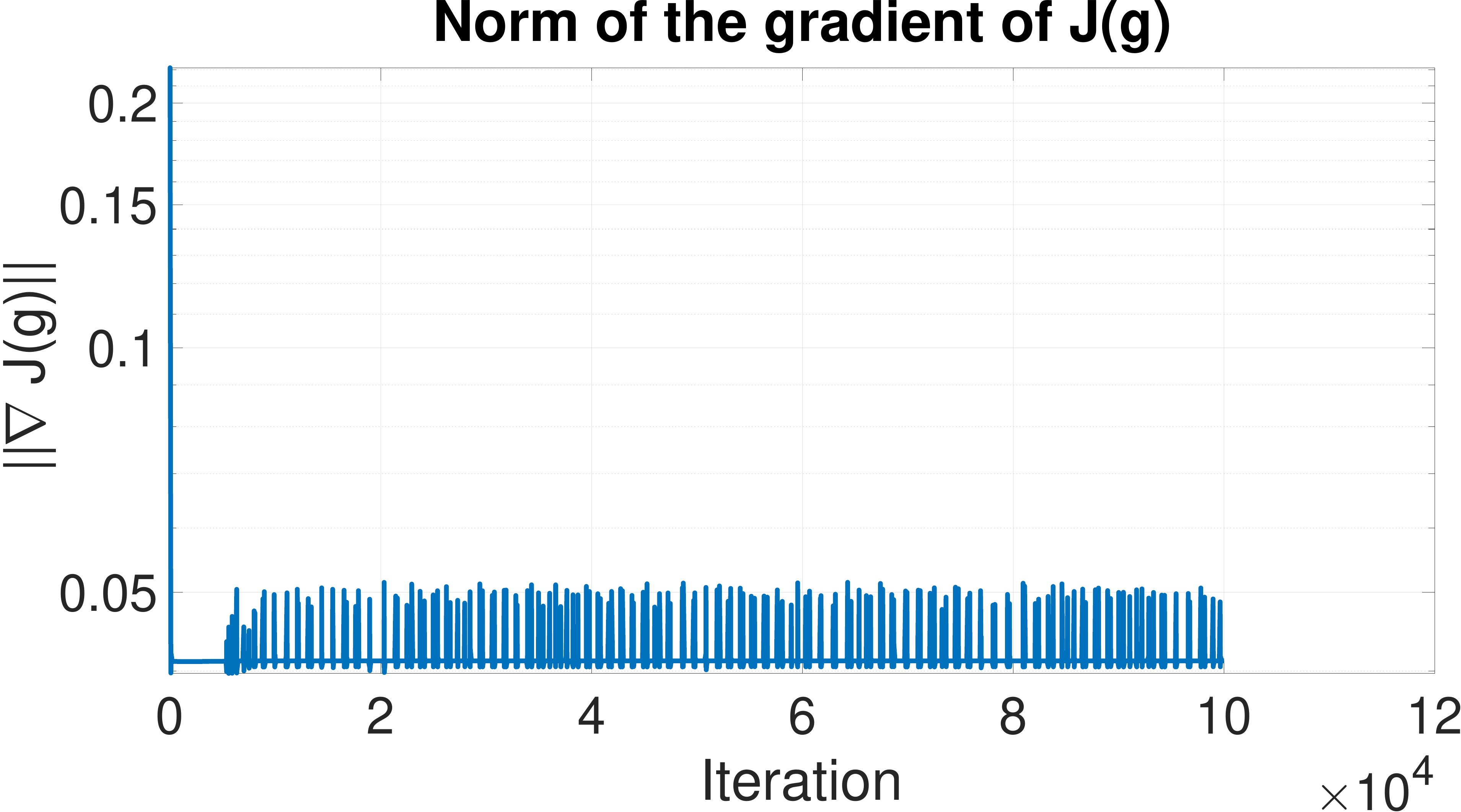}
\includegraphics[width=0.328\textwidth]{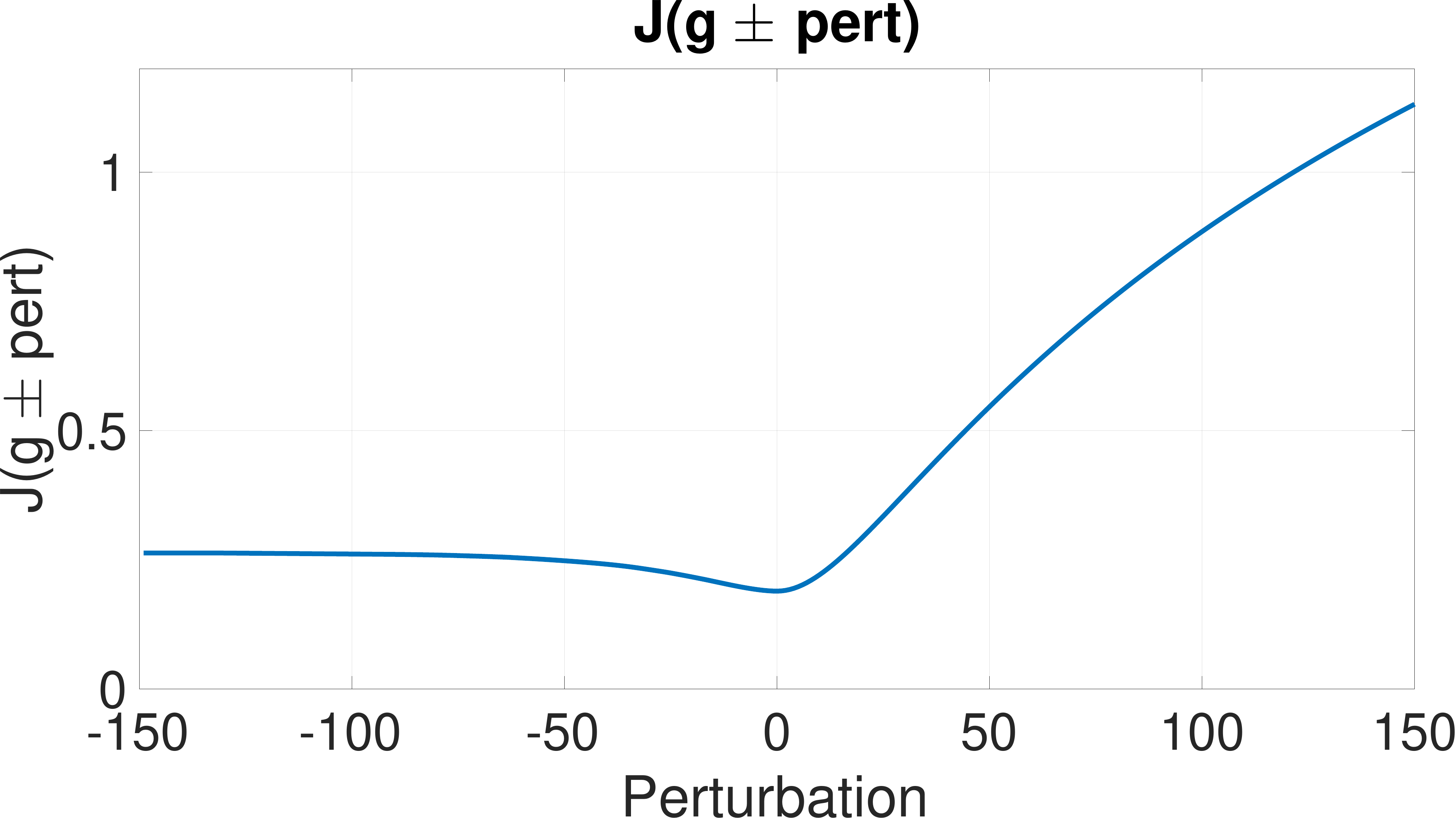}
\caption{Results for Robin boundary conditions with $\Omega_o=[-1,1]$. Top row: Dynamics of control $g$ (left) and the associated variables $u$ (center) and variable $v$ (right). Bottom row: Evolution of the cost functional $J(g)$ (left), evolution of $\|\nabla J(g)\|$ (center) and the influence of perturbing the obtained control (right).}\label{fig:RBCcase1}
\end{center}
\end{figure}

\begin{figure}[H]
\begin{center}
\includegraphics[width=0.328\textwidth]{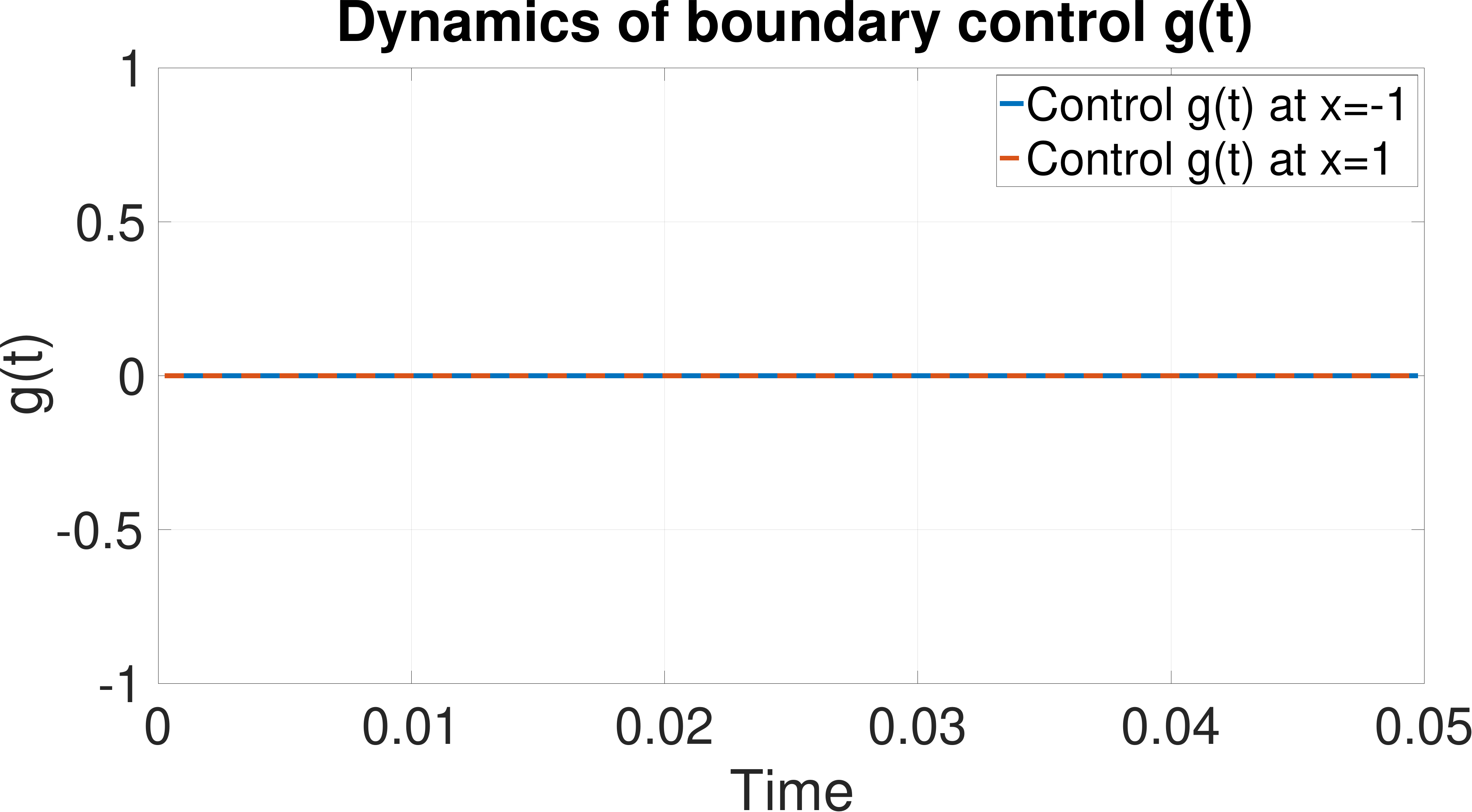}
\includegraphics[width=0.328\textwidth]{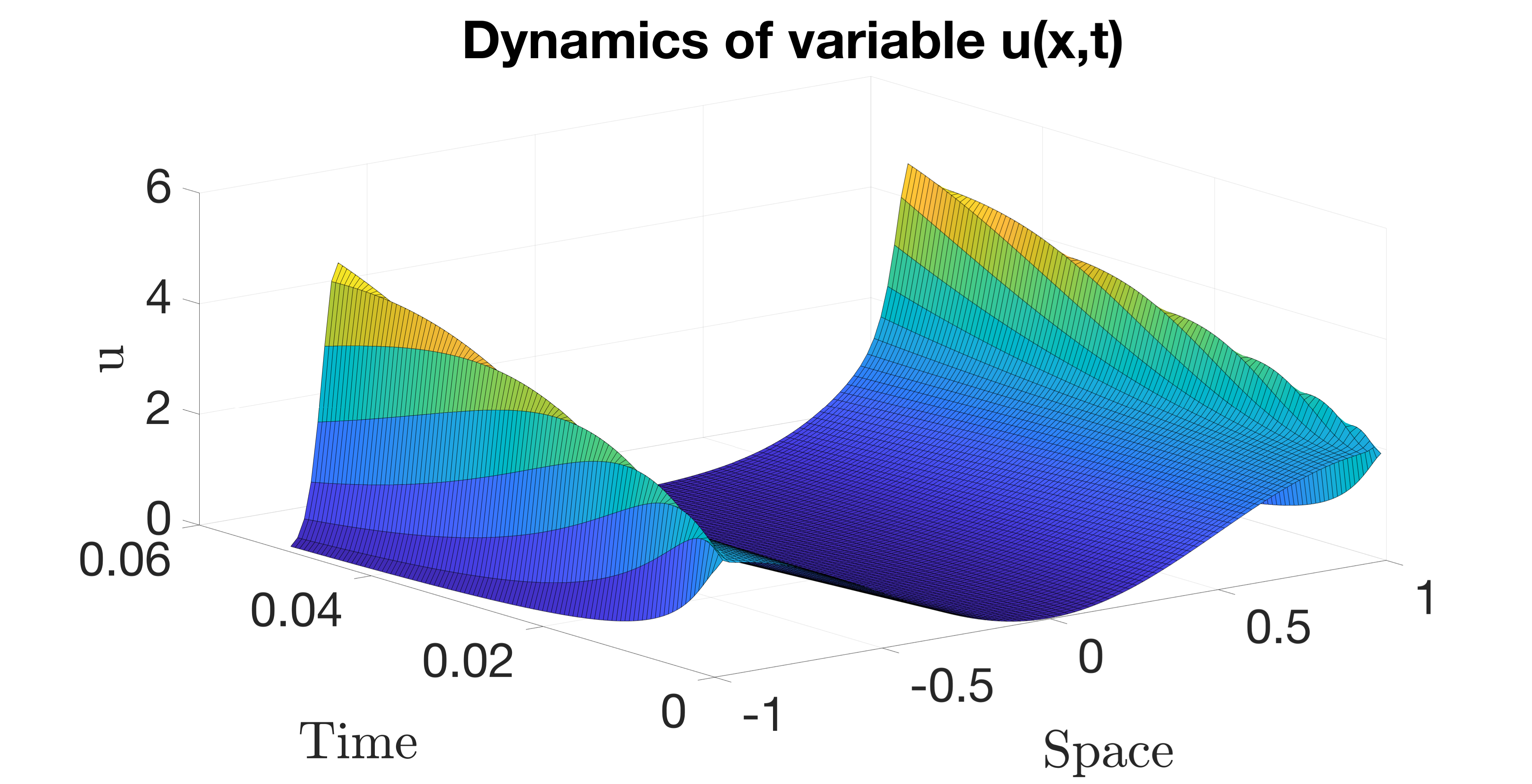}
\includegraphics[width=0.328\textwidth]{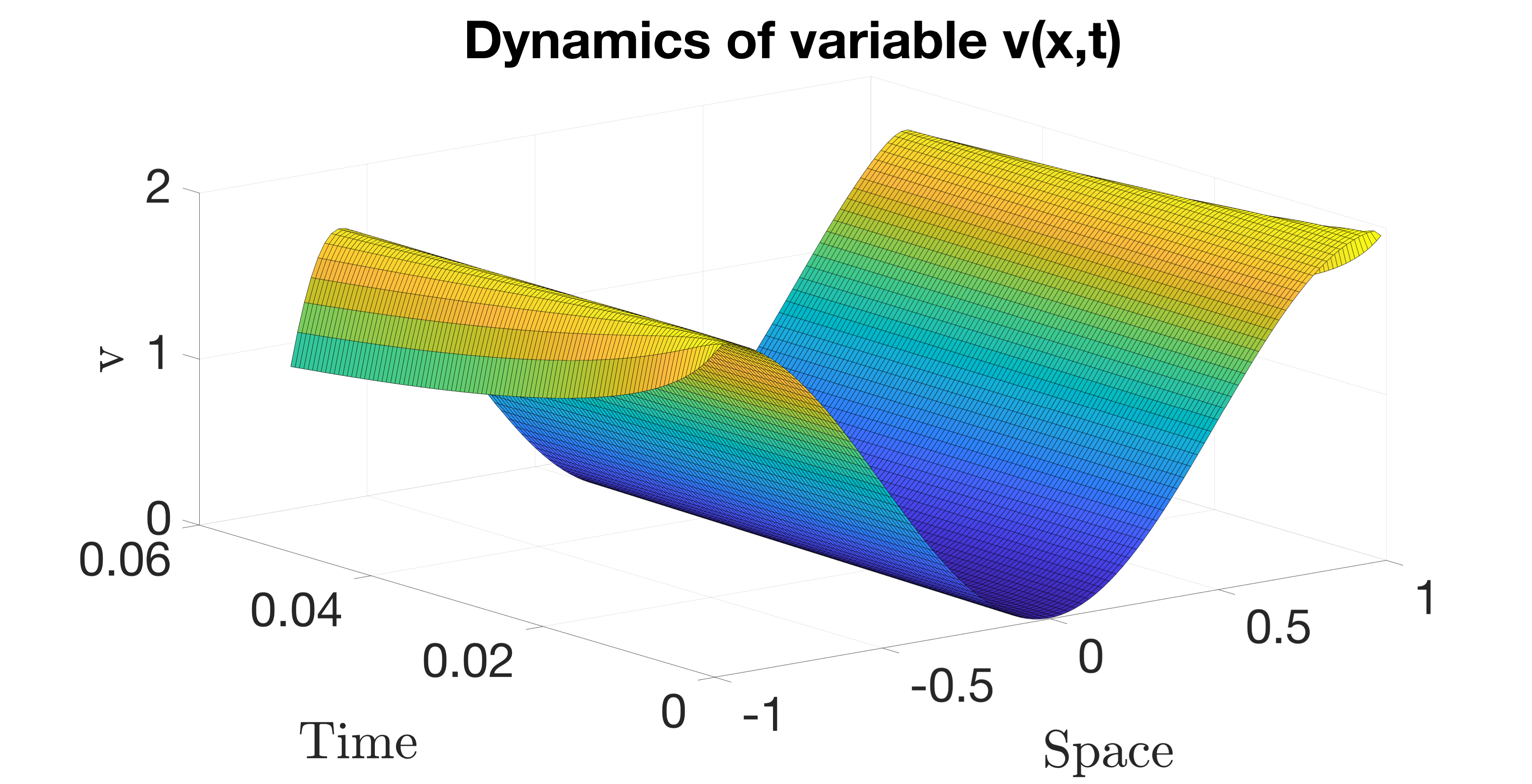}
\\
\includegraphics[width=0.328\textwidth]{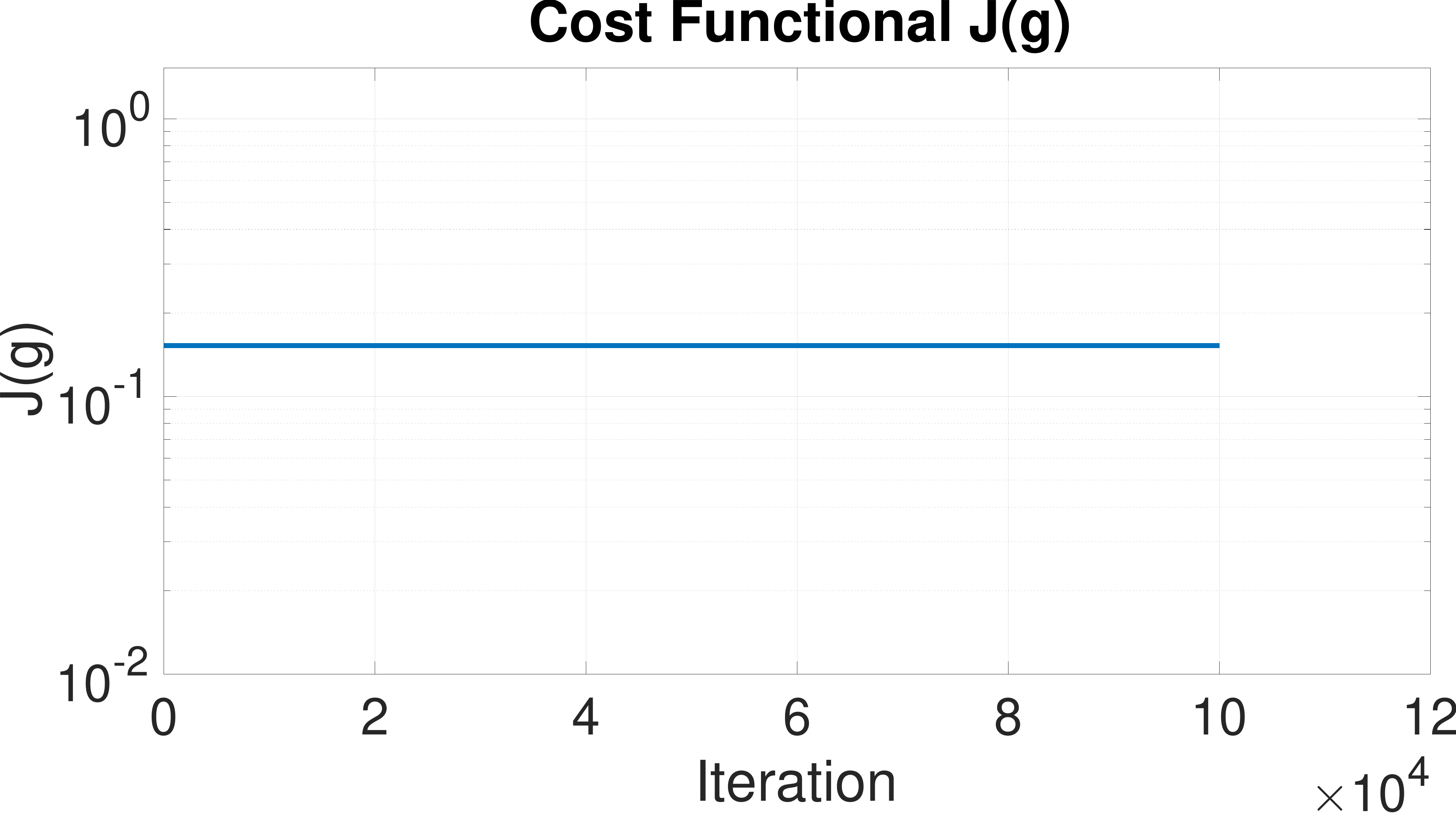}\includegraphics[width=0.328\textwidth]{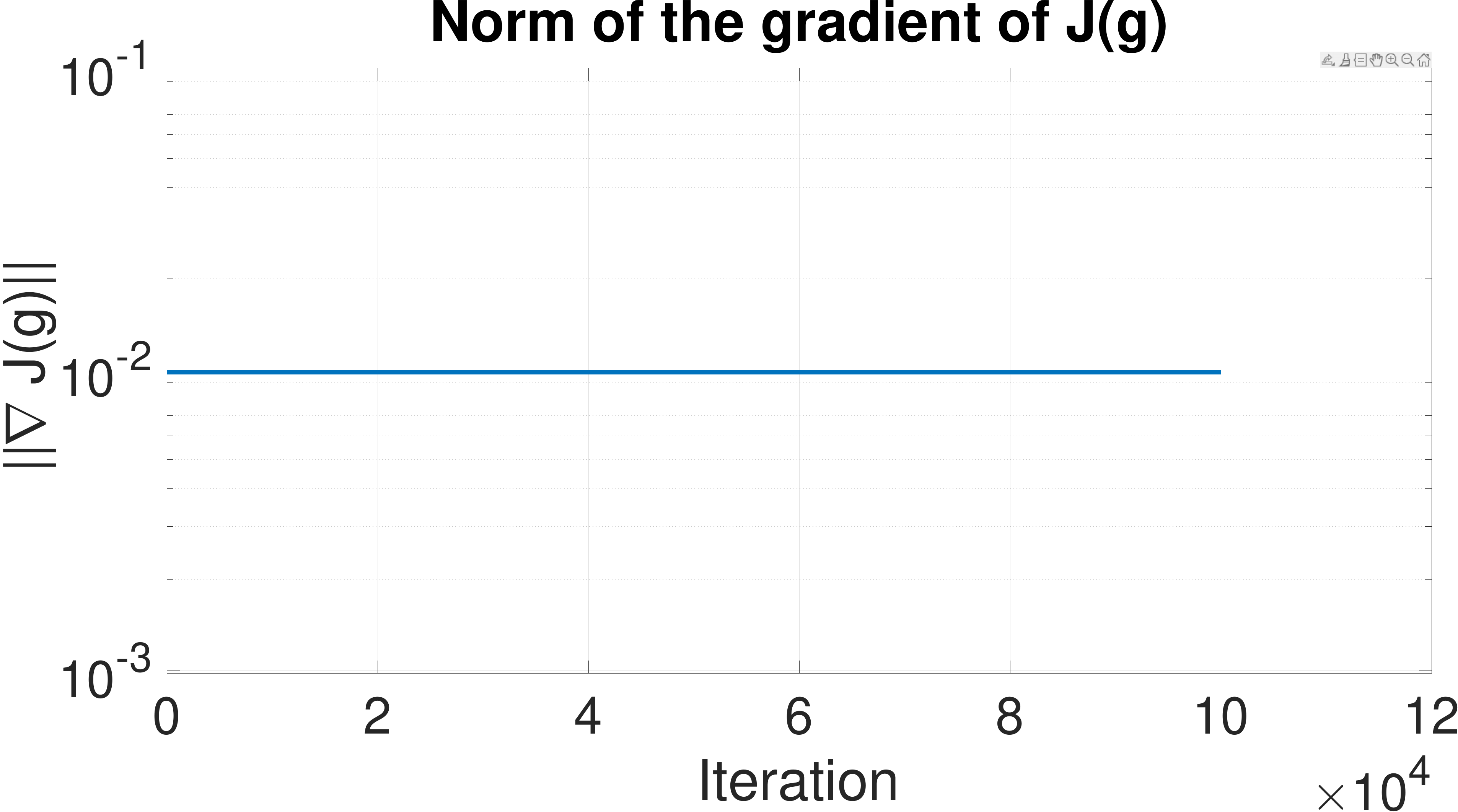}
\includegraphics[width=0.328\textwidth]{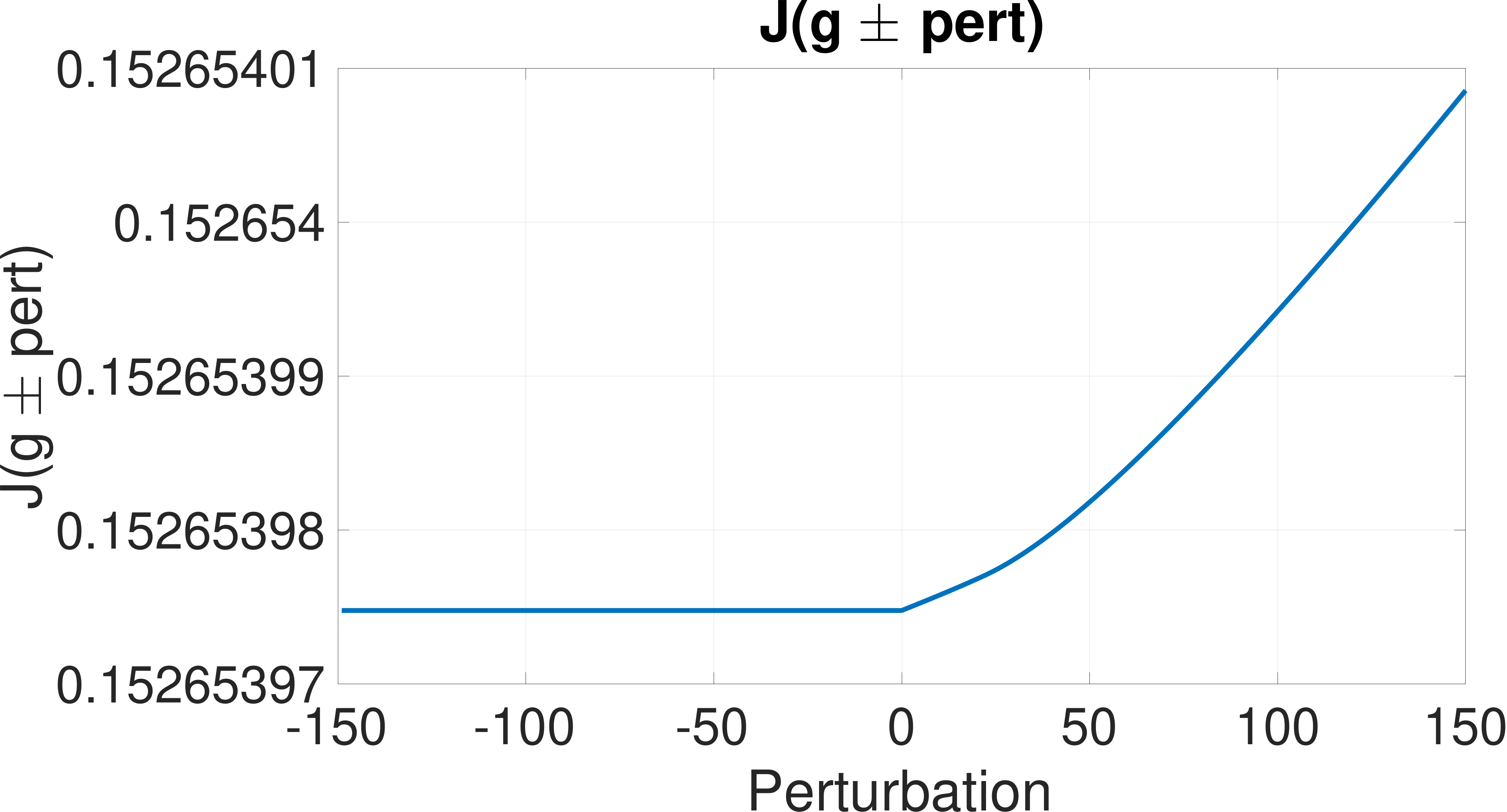}
\caption{Results for Robin boundary conditions with $\Omega_o=[-0.5,0.5]$. Top row: Dynamics of control $g$ (left) and the associated variables $u$ (center) and variable $v$ (right). Bottom row: Evolution of the cost functional $J(g)$ (left), evolution of $\|\nabla J(g)\|$ (center) and the influence of perturbing the obtained control (right).}\label{fig:RBCcase2}
\end{center}
\end{figure}

\section{Conclusions}\label{sec:conclusions}

In this work we have derived a new optimal control algorithm for the Keller-Segel chemo-attraction system based on the \textit{discretize-then-optimize} approach, that is, we first discretize the state problem system and then we develop the optimal control algorithm. This approach has the advantage of allowing us to compute exactly the discrete gradient of the reduced cost, which is used later on to construct the discrete optimal control. We have combined these ideas with the usage of the Adam scheme for the minimization process, which results in a very efficient algorithm. By performing simulations in a systematic way we have been able to verify the applicability of this new approach to complicated situations. Moreover, we have been able to get insights on the dynamics of this nonlinear problem, being the main conclusion that in the case of a bilinear distributed control, although the control is imposed on the concentration of the chemical ($v$-variable), this control is able to greatly affect the dynamics of the density of the live organisms ($u$-variable). On the other hand, the case of a boundary control in the $v$-variable doesn't seem to be able to affect much the dynamics of the $u$-variable, leading to the conclusion that the applicability in practice of this type of boundary controls 
seems to be very limited. In particular, we have observed that  both  choices  of 
 boundary control (bilinear and Robin) have more limited effect  when the observation is near to the boundary than when the observation is made in the interior.  This seems to be related with the fact that if the control acts, it produces changes near to the boundary, going away from  the constant desired state.

\section*{Acknowledgements}
F. Guill\'en-Gonz\'alez and M.A. Rodr\'iguez-Bellido have been partially supported by Grant  I+D+I PID2023-149182NB-I00 funded by MICIU/AEI/10.13039/501100011033 and, ERDF/EU. F. Guill\'en-Gonz\'alez and M.A. Rodr\'iguez-Bellido thank IMUS-Maria de Maeztu grant CEX2024-001517-M - Apoyo a Unidades de Excelencia Mar\'ia de Maeztu for supporting this research, funded by MICIU/AEI/ 10.13039/501100011033.

\section*{Conflicts of Interest}

The authors declare they have no known competing financial interests that could have appeared to influence the results reported in this work.


\end{document}